\documentclass[a4paper,11pt,reqno,a4paper,oneside]{amsart}

\usepackage{mathtools,amsmath,amssymb,amsfonts,amsthm,mathrsfs}
\usepackage[pdftex]{graphicx}
\usepackage{graphicx}
\mathtoolsset{showonlyrefs}
 \pdfoutput=1

\DeclareGraphicsExtensions{.bmp,.png,.pdf,.jpg,}

\usepackage{bm}
\usepackage{verbatim}
\usepackage{color}
\usepackage{mathrsfs}
\usepackage{pgfplots}

\newcommand{\RR}{\mathbb R}
\newcommand{\NN}{\mathbb N}

\newcommand{\sgn}{{ \rm sgn }}
\newcommand{\esssup}{{\rm ess \;sup\,}}
\newcommand{\essinf}{{\rm ess \;inf\,}}

\newcommand{\limess}[1]%
{

\begin{array}[t]{c}
{\rm ess\, lim}\\
{\scriptstyle #1}
\end{array}

}

\newcommand{\supess}[1]%
{

\begin{array}[t]{c}
{\rm ess\, sup}\\
{\scriptstyle #1}
\end{array}

}
\newcommand{\infess}[1]%
{

\begin{array}[t]{c}
{\rm ess\, inf}\\
{\scriptstyle #1}
\end{array}

}

\newcommand{\JJ}{{\mathcal J}}

\def\R{{\mathbb R}}
\def\N{{\mathbb N}}

\def\a{\alpha}
\def\b{\beta}

\def\d{\delta}

\def\r{\varrho}

\def\s{\sigma}

\def\l{\lambda}
\def\p{\partial}
\def\O{\Omega}
\def\e{\varepsilon}
\def\v{\varphi}

\def\o{\omega}
\def\mc{\mathcal}
\def\mf{\mathfrak}

\numberwithin{equation}{section}

\theoremstyle{definition}

\theoremstyle{plain}
\newtheorem{theorem}{Theorem}[section]
\newtheorem{proposition}{Proposition}[section]
\newtheorem{lemma}{Lemma}[section]
\newtheorem{corollary}{Corollary}[section]

\theoremstyle{definition}
\newtheorem{rem}{Remark}[section]

\begin{document}

\title[Regular versus singular solutions]
{Regular versus singular solutions in quasilinear indefinite problems with
sublinear potentials}

\thanks{J. L\'{o}pez-G\'{o}mez has been  supported by the Research Grant PGC2018-097104-B-100 of the Spanish Ministry of Science, Innovation and Universities and by the Institute of Interdisciplinar Mathematics of Complutense University. \\ \indent P. Omari has been  supported by  ``Universit\`a degli Studi di Trieste--Finanziamento di Ateneo per Progetti di Ricerca Scientifica''.
This research has been performed under the auspices of INdAM-GNAMPA}

\author{Juli\'{a}n L\'{o}pez-G\'{o}mez}
\address{Juli\'{a}n L\'{o}pez-G\'{o}mez:
Instituto Interdisciplinar de Matem\'aticas,
Universidad Complutense de Madrid,
Madrid 28040, Spain}
\email{julian@mat.ucm.es}
\author{Pierpaolo Omari}
\address{Pierpaolo Omari:
Dipartimento di Matematica e Geoscienze,
Universit\`a degli Studi di Trieste,
Via A. Valerio 12/1,
34127 Trieste, Italy}
\email{omari@units.it}

\vspace{0.1cm}

\begin{abstract}
The aim of this paper is analyzing  existence,   multiplicity, and regularity issues   for   the positive solutions of the quasilinear Neumann problem
\begin{equation*}
   \left\{ \begin{array}{ll}
   \displaystyle
   - \big(  {u'}/{\sqrt{1+(u')^2}} \, \big)' =\l a(x) f(u), & \quad 0<x<1,
\\ [1mm]
 u'(0)=u'(1)=0. & \end{array}\right.
\end{equation*}
Here, $\displaystyle   \big({u'}/{\sqrt{1+(u')^2 }}\, \big)' $ is the one-dimensional curvature operator,  $\l\in\R$ is     a   parameter, the weight $a(x) $ changes sign, and,  in most occasions,  the function $f (u)$ has a sublinear potential $F(u)$ at $\infty$. Our discussion displays   the manifold  patterns  occurring for   these solutions,  depending on  the behavior   of the potential $F(u)$ at $u=0$,  and, possibly, at infinity, and of the weight function $a(x)$ at its nodal points.
\end{abstract}

\vspace{0.3cm}


\vspace{0.1cm}

\vspace{0.1cm}

\date{\today}

\maketitle

\section{Introduction}
\label{s1}

\noindent The main aim of this paper is analyzing the set of  positive solutions, regular or singular, of the quasilinear Neumann problem
\begin{equation}
\label{1.1}
   \left\{ \begin{array}{ll}
   \displaystyle
   -\left( \frac{u'}{\sqrt{1+(u')^2}}\right)' =\l a(x) f(u), & \quad 0<x<1, \\[1ex]
   u'(0)=u'(1)=0. & \end{array}\right.
\end{equation}
 Here, $\displaystyle   \big({u'}/{\sqrt{1+(u')^2 }} \, \big)' $ is the one-dimensional curvature operator, $\l\in\R$ is viewed as a parameter, and the functions $a(x) $ and $f (u)$ are generally supposed to satisfy
\smallskip
\begin{itemize}
\it
\item[$(a_1)$] $a\in  L^1(0,1)$, $\displaystyle \int_0^1 a(x)\,dx <0$, and $a(x)>0$ on a set of positive measure,
\end{itemize}
\smallskip
and,  respectively,
\smallskip
\begin{itemize}
\it
\item[$(f_1)$]
$f\in \mc{C}(\R)$,  $f(0)=0$, $f(u)>0$ if $u>0$, and, for some constants $ h>0 $, $p>0$ and $q\in (0,1)$,
\begin{equation}
\label{1.2}
   \lim_{u\to 0^+}\frac{f(u)}{u^{p}}=1\quad \hbox{ and } \quad \lim_{u\to \infty}\frac{f(u)}{u^{-q}}=h.
\end{equation}
\end{itemize}
\smallskip
 Obviously, if $\lim_{u\to 0^+}\frac{f(u)}{u^{p}}=h_0>0$, then, by replacing $f(u)$ with $g(u) = f(u)/h_0$ and $\l$ with $\mu= h_0\l $,  the first condition of \eqref{1.2} is always met; in particular, $f'(0^+)=1$ if $p=1$.
\par
By a  \emph{regular}
solution of \eqref{1.1}, we mean a function $u \in W^{2,1}(0,1),$ which fulfills, for some $\lambda \in \RR$, the equation a.e. in $(0,1)$, as well as  the boundary conditions. It is evident that  $u$ is a regular solution of \eqref{1.1} if and only if it satisfies
\begin{equation}
\label{1.3}
   \left\{ \begin{array}{ll} -u'' = \l a(x)f(u)g(u'), & \quad 0<x<1, \\
  u'(0)=u'(1)=0, & \end{array}\right.
\end{equation}
where
\begin{equation}
\label{1.4}
  g(v)=(1+v^2)^\frac{3}{2}.
\end{equation}
In this paper we will also use, as we previously did in \cite{LOR1, LOR2, LGO-19, LGO-JDE, LGO-20}, the notion of bounded variation solution of \eqref{1.1}. A function $u  \in BV(0,1)$ is  said to be a \emph{bounded variation} solution of    \eqref{1.1} if the next identity holds
\begin{align}
\label{1.5}
\int_0^1 \frac{Du^a  D\phi^a}{\sqrt{1+ (Du^a)^2}}\,dx
+\int_0^1  \frac{Du^s}{ |Du^s | }   \, D\phi^s
=\int_0^1  \l a f(u)   \phi\,dx
\end{align}
for every $\phi \in BV(0,1)$ such that $ |D\phi ^{s}|$ is absolutely continuous with respect to
 $|Du^{s}|$. Here, for any $v\in BV(0,1)$, $ Dv =  {D v}^a  dx +  {D v}^s$ is the Lebesgue--Nikodym decomposition, with respect to the Lebesgue measure $dx$ in $\RR$,  of the Radon measure $Dv$ in its absolutely continuous part $  {Dv}^a dx$, with density function ${Dv}^a$,  and  its singular part  ${D v}^s $. Further,  $ \frac{ {Dv}^s}{ |{Dv}^s|}$  denotes the density function of  ${Dv}^s$ with respect to its absolute variation $|{Dv}^s|$.  We refer, e.g.,  to  \cite{Anz83, AmFuPa, LOR1} for additional details on these concepts.
\par
It is   apparent that   any regular solution  is a bounded variation solution.  When a bounded variation solution is not regular, it is called \emph{singular}. Such  solutions may exhibit jumps and, in principle, even more complex behaviors. Throughout this paper, all solutions will be bounded variation solutions, even if not emphasized explicitly.
\par
A solution $u$ of \eqref{1.1} is  said  to be \emph{positive} if $\essinf u \ge 0$ and $\esssup u >0$,  whereas it is said  \emph{strictly positive} if  $\essinf u > 0$.   It is also said  that a pair $(\lambda, u)$ is a positive, or  strictly positive, solution of \eqref{1.1} if $u$ is a positive, or  strictly positive,  solution of \eqref{1.1}, respectively,  for some $\lambda\geq 0$. As this paper focuses attention on positive solutions, all solutions through it will be understood to be positive.
\par
Throughout this paper,  we
denote by $\mc{S}_r^+$ the set of couples $ (\lambda, u)\in[0,\infty)\times\mc{C}^1[0,1]$ such that $(\l,u)$ is a positive regular solution of \eqref{1.1}, together with $(0,0)$ and $(\l_0,0)$, its two possible bifurcation points from the trivial line  $(\l,0)$, $\l\in\RR$. Similarly, we denote by $\mc{S}_{bv}^+$ the set of couples $ (\lambda, u)  \in [0,\infty) \times BV(0,1)$ such that $(\l,u)$ is a positive (bounded variation)  solution of \eqref{1.1}, together with $(0,0)$ and $(\l_0,0)$. By definition, the set of singular positive solutions of \eqref{1.1}, denoted by $\mc{S}_s^+$, is given by
$\mc{S}_s^+=\mc{S}_{bv}^+\setminus \mc{S}_r^+$.
\par
Let us observe that,
according to \eqref{1.2}, the   potential of $f(u)$, defined by $F(u)= \int_0^u f(s)\,ds$
for all $u\in \R$, satisfies
\begin{equation}
\label{Fp}
  \lim_{u\to0^+} \frac{F(u)}{u^{p+1}}
  =\lim_{u\to 0^+} \frac{f(u)}{(p+1) u^p}=
   \frac{1}{p+1}>0
\end{equation}
and, hence, it is \emph{quadratic} at zero if $p=1$, \emph{subquadratic} if $0<p<1$,  and \emph{superquadratic} if $p>1$. Similarly, we have that
\begin{equation}
\label{Fq}
    \lim_{u\to  \infty} \frac{F(u)}{u^{1-q}}
    =\lim_{u\to \infty}\frac{f(u)}{(1-q)u^{-q}}
    =\frac{h}{1-q}>0
\end{equation}
and, therefore, $F(u)$ is \emph{sublinear} at infinity, because $0<1-q<1$.
\par
As the main goal of this paper is analyzing   the existence and the  interplay between the regular and the singular solutions of \eqref{1.1} under  $(f_1)$, this work can be viewed as a natural continuation of \cite{LOR1, LOR2, LGO-19, LGO-JDE, LGO-20} to cover the case where $F(u)$ is sublinear at infinity.
There are strong motivations for studying this problem. A rather thorough discussion is presented in \cite{LOR1, LOR2, LGO-19, LGO-JDE, LGO-20}, together with a wide list of relevant references, including
\cite{LU, Se,  BDGM, Te, CF, Fi, Em, Giu, Ge, Ge85,  GMT, Hu85, Ge85, NS2, CaDMLePa, Se88, Na, CZ91, Ma, CMM96, KuRo, BuGr, Le, BoHaObOmJDE, BoHaObOmTS, ObOmDIE,  ObOmDCDS, ObOmRi, COZ, CDCO,  CDCOOS, ObOm20}.
\par
Note that the condition \eqref{1.2} implies that
\begin{equation}
\label{1.7}
\lim_{u\to \infty}f(u)=0
\end{equation}
and hence $\|f\|_\infty =\max_{u\geq 0} f(u) <\infty$. For the validity of some of the results found in this paper we shall however impose a stronger condition than $(f_1)$. Since $f\in \mc{C}(\R)$ and $f(0)=0$ it follows from \eqref{1.7} that there exists a maximal $M>0$ such that $f(M)=\|f\|_\infty$. The next assumption incorporates into $(f_1)$ the monotonicity of $f(u)$ on each of the intervals $(0,M)$ and $(M,\infty)$:

\smallskip
\begin{itemize}
\it
\item[$(f_2)$] $f(u)$ satisfies { $(f_1)$}, $f \in \mc{C}^1[M,\infty)$,  and it is increasing in $(0,M)$ and decreasing in $(M,\infty)$.
\end{itemize}
\smallskip
For any given $p>0$, $q \in (0,1)$ and $M>0$, the  function
\begin{equation}
\label{1.9}
  f(u)= \left\{ \begin{array}{ll} u^p & \quad \hbox{if} \;\; 0 \leq u \leq M,\\[1ex]
  \tfrac{M^{p+q}}{u^q}
  & \quad \hbox{if} \;\; u > M
\end{array}\right.
\end{equation}
provides us with a paradigmatic example of function satisfying $(f_2)$. By regularizing it around $M$ it is very easy
to construct a family of functions satisfying $(f_2)$, with the same shape as \eqref{1.9} at $u=0$ and $u=\infty$, and such that $f \in\mc{C}^1(0,\infty)$. In some cases, when using bifurcation methods,  more regularity will be   necessary, as to require that
\smallskip
\begin{itemize}
\it
\item[$(f_3)$] $f(u) $  satisfies  $f\in  \mc{C}^1(\RR)$, $f(0)=0$, $f'(0)=1$, $f(u)>0$ if $u>0$, and, for some constants $ h>0 $ and $q\in (0,1)$, $\lim_{u\to \infty}\frac{f'(u)}{-q u^{-q-1}}=h>0$ holds.
\end{itemize}
\smallskip
Moreover, in some circumstances we will   replace  $(a_1)$ with the stronger condition
\smallskip
\begin{itemize}
\it
\item[$(a_2)$]
$a\in L^\infty(0,1)$, $\displaystyle \int_0^1a(x)  \, dx<0$, and there is $z\in (0,1)$ such that $a(x)>0$ for a.e. $x\in (0,z)$ and $a(x)<0$ for a.e.  $x\in (z,1)$.
\end{itemize}
\smallskip
\noindent
When
$(a_2)$ holds,  by \cite[Cor. 3.7]{LGO-19},  any positive singular solution $(\l,u)$ of \eqref{1.1} satisfies
\begin{equation*}
\begin{split}
  u|_{[0,z)} & \in W^{2,\infty}_\mathrm{loc}[0,z) \cap W^{1,1}(0,z) \;\; \hbox{and is concave}, \\
   u|_{(z,1]} & \in W^{2,\infty}_\mathrm{loc}(z,1] \cap W^{1,1}(z,1) \;\; \hbox{and is convex}.
\end{split}
\end{equation*}
Moreover, $u'(x)<0$ for every $x \in (0,z)$, $u'(x)\le0$ for every $x \in (z,1)$, $u'(0)=u'(1)=0$ and
$u'(z^-)=u'(z^+)=-\infty$. Therefore, in this case, singular solutions can only develop jumps at $z$, the node of the function $a(x)$.
\par
  Throughout this paper, for any given $r<s$ and $V\in L^\infty(r,s)$,  we  denote by $ \s[-D^2+V(x);\mc{B},(r,s)]$, with $D^2= \tfrac{d^2}{dx^2}$, the lowest eigenvalue of the boundary value problem
\begin{equation*}
   \left\{ \begin{array}{ll} -w''+V(x)w = \tau w, & \quad r<x<s, \\[1ex]
  \mc{B}w(r)=\mc{B}w(s)=0, & \end{array}\right.
\end{equation*}
where $\mc{B}$ stands either for the Neumann boundary operator, $\mc{N}$, or the Dirichlet
boundary operator, $\mc{D}$.  If $(f_1)$ holds with $p\ge1$, and $u$ is a strictly positive regular solution of  \eqref{1.1}, i.e.,  \eqref{1.3} holds, then $u$ must be a principal eigenfunction associated with  the eigenvalue
\begin{equation*}
   \s[-D^2-\l a(x)\tfrac{f(u)}{u}g(u');\mc{N},(0,1)] =0
\end{equation*}
and hence, e.g., by \cite[Thm. 7.10]{LG13}, $\min u >0$, i.e., $u$ is strictly positive.
Moreover, if $f(u)$ satisfies $f\in \mc{C}^1(0,\infty)$ and $(f_1)$ with $p\ge1$, and $(a_1)$ holds, then,  from \cite[Prop. 1.1]{LOR1} (see also \cite[Lem. 2.1]{LOR2}), we find that  $\l\geq 0$  if \eqref{1.1} possesses a strictly positive solution. Actually, the solution is constant in $[0,1]$ if $\l=0$. Thus, non-constant positive solutions of \eqref{1.1} can only arise for $\l>0$. However,
the situation is quite different if $(f_1)$ holds with $0< p <1$. Indeed, as   pointed out in \cite[Rem. 1.8]{LOR1}, in this case {\it dead core} solutions may occur, thus provoking the possible existence of positive regular solutions even for $\lambda <0$. This  can be shown by slightly modifying \cite[Ex. 2]{LOR1} as follows.
In this work we anyhow restrict our analysis   to   the case $\lambda\ge0$.
\par
\smallskip
\noindent {\bf Example.} Let  $f\in \mc{C}(\R) $  be such that $f(u) =  \sqrt{u}$ if $0 \le u \le 1$.
Then, the function   defined  by
\begin{equation*}
u(x) =
\begin{cases}
\displaystyle
\tfrac{1}{12^2}{( \tfrac{2}{3^4} -x^4)}  &\quad  \text{if } 0\le x \le \tfrac{1}{3},
\\[2mm]
\displaystyle
\tfrac{1}{12^2}{(x-\tfrac{2}{3})^4}  &\quad  \text{if } \tfrac{1}{3} < x \le  \tfrac{2}{3},
 \\[2mm]
\displaystyle
0 &\quad  \text{if } \tfrac{2}{3} < x \le 1,
\end{cases}
\end{equation*}
satisfies $u\in W^{2,\infty}(0,1)$ and it is a positive regular solution of \eqref{1.1}
with   $\lambda =-1$ and   $a\in L^\infty(0,1) $ defined  by
\begin{equation*}
 a(x)=
\begin{cases}
\displaystyle
\left( \frac{u'}{\sqrt{1+(u')^2}}\right)' \frac{1}{\sqrt{u(x)}}  & \text{if }0\le x < \tfrac{2}{3}, \; x\neq \tfrac{1}{3},
 \\[1ex]
\displaystyle
A
 & \text{if } \tfrac{2}{3} \le x \le 1,
\end{cases}
\end{equation*}
for every constant $A\in\R$. In particular,   so that  $(a_1)$ holds,  the constant $A$ can be chosen  so that $\int_0^1 a(x) \,dx <0$. It is worthy observing  that, for this choice of the functions $f(u)$ and $a(x)$, by \cite[Thm. 9.1]{LOR2}, the problem \eqref{1.1} also admits positive regular solutions for sufficiently small $\lambda >0$.
\par
\smallskip
We describe below the main findings of this paper which concern, in the following order,  with  non-existence, existence, multiplicity and regularity properties of the positive solutions  of \eqref{1.1}.
  \par
In Section \ref{s2} we establish that \eqref{1.1} cannot admit a singular solution for sufficiently small $\l>0$ under conditions $(f_1)$ and $(a_1)$, regardless the size of $p>0$. Actually, \eqref{1.1} cannot admit any solution, neither  singular nor regular,  for sufficiently small $\l>0$ if $p\geq 1$. These conclusions are optimal, because, due to
 \cite[Thm. 9.1]{LOR2},
 the problem \eqref{1.1} has at least one    positive regular solution  $(\l,u_\l)$,   for sufficiently small $\l > 0$, if $p\in (0,1)$.
\par
In Section \ref{s3} we show, for quadratic potentials at the origin ($p=1$),  the existence of two components of regular  solutions of \eqref{1.1} bifurcating from the trivial line
at $\l=0$ and at $\l =\l_0$. Throughout this paper, $\l_0>0$ stands for the  positive principal eigenvalue of the linear weighted eigenvalue problem
\begin{equation}
\label{i.10}
\left\{ \begin{array}{l} -\v'' = \l a(x)\v,  \quad 0<x<1,
\\[1ex] \v'(0)=\v'(1)=0. \end{array}\right.
\end{equation}
According to a classical result of Brown and Lin \cite{BL} (see
\cite[Ch. 9]{LG13} for the general theory), besides $\l=0$, the  problem \eqref{i.10} admits under $(a_1)$ a unique positive eigenvalue $\l_0>0$, with  a strictly positive eigenfunction $\v$,  unique up to a positive multiplicative constant.
\par
By the bifurcation theorems in \cite{LGO-19}, these components are subcomponents of some components of the set $\mc{S}_{bv}^+$ of the positive bounded variation solutions of \eqref{1.1}. For superlinear and linear potentials at infinity it is
already known from  \cite{LGO-JDE, LGO-20} that the regular solutions can develop singularities along these components. The existence of positive bounded variation solutions of \eqref{1.1}  for all $\l >\l_0$ is guaranteed by Theorem  \ref{thiii.1}.
\par
Section \ref{s4}  deals with  subquadratic potentials at the origin ($0<p<1$).
In this case the existence of positive bounded variation solutions of \eqref{1.1} for all $\l >0$ follows from \cite[Thm. 1.2]{LOR1}.
The main goal of  Section \ref{s4}  is   establishing the existence of a component of   the set  $\mc{S}_r^+$ of  positive regular solutions bifurcating from $(0,0)$,  which is done  in Theorem \ref{thiv.2}. This result   complements \cite[Thm. 9.1]{LOR2}.   As the proof of  \cite[Thm. 9.1]{LOR2} is based on the direct method  of   calculus of variations it does not guarantee such   structure information,  provided  instead by Theorem \ref{thiv.2} which relies  on the construction of sub- and supersolutions and the use of topological degree methods.
\par
Section \ref{s5} focuses on superquadratic potentials at the origin ($p>1$).
From \cite[Thm. 1.5]{LOR1} and \cite[Thm. 10.1]{LOR2}  the existence of two positive solutions, one of them regular and small, can be inferred for  sufficiently large $\l >0$.
Theorem \ref{thv.2} establishes    the existence of a component of   the set  $\mc{S}_r^+$ of  positive  regular  solutions of \eqref{1.1} bifurcating from $0$ as $\l \to \infty$.
 The proof of Theorem \ref{thv.2} relies on some elementary topological techniques based on the theory of superlinear indefinite problems developed in \cite{AL98}.
\par
Section  \ref{s6} ascertains, for every $p>0$, the limiting profile of the regular solutions
of \eqref{1.1} that are separated away from zero as $\l\to  \infty$, should they exist,
when $f(u)$ and $a(x)$ satisfy  $(f_2)$ and $(a_2)$. The assumption $(a_2)$ entails  that any regular solution  $u$ of \eqref{1.1}  is decreasing in $[0,1]$, and thus
$ u(0)=\|u\|_{L^\infty(0,1)}. $
These solutions satisfy $u(0)>M$ (see $(f_2)$) for sufficiently large $\l>0$ and
grow up to infinity on $[0,z)$ at least at the rate $C_1 \l^{1/q}$, while they decay to zero on $(z,1]$ at least as $C_2 \l^{-1/p}$, for some constants $C_1>0$, $C_2>0$, as sketched in
Figure \ref{Fig01}.
\par
\begin{figure}[!ht]
\centering{
  \includegraphics[scale=0.6]{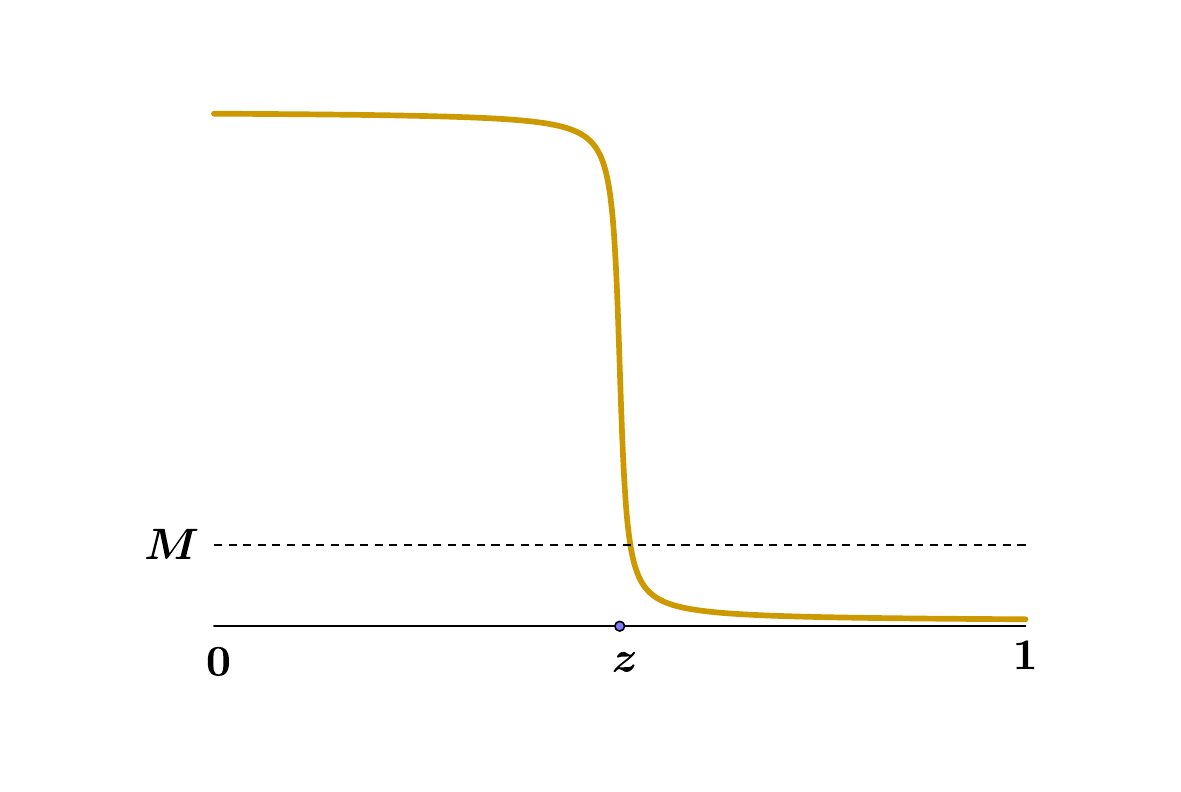}  }
{\caption{The profile of the solutions for large $\l>0$.}
\label{Fig01}}
\end{figure}
 As, according to Theorem \ref{thv.2}, \eqref{1.1} possesses a subcontinuum of the set of regular solutions of \eqref{1.1} consisting of small solutions,  these solutions  must be left outside the mathematical analysis of Section \ref{s6} for the validity of all our results therein.
\par
Section \ref{s7}    carries out a detailed discussion of the existence, and non-existence,  of singular  solutions.   Precisely, Theorem~\ref{th7.1} establishes  a general    criterion that allows to  ascertain the local regularity of the bounded variation solutions of the equation
\begin{equation}
   -\left( \frac{u'}{\sqrt{1+(u')^2}}\right)' =h(x),
\end{equation}
where $h\in  L^1(0,1)$ satisfies   $h(x)\ge 0$ a.e. in $(z-\d_1,z)$ and $h(x) \le 0$ a.e. in $(z,z+\d_2)$ for some  $z \in (0,1)$ and $\d_1, \d_2>0$, according to the behavior of $h(x)$ near its nodal point $z$. Based on this result,   Theorem \ref{th7.2} shows  that, under condition $(a_2)$, the problem \eqref{1.1} cannot admit   singular solutions   if
\begin{equation}
\label{i.12}
\hbox{either}\;\; \int_{0}^z \left( \int_x^za(t)\,dt\right)^{-\frac{1}{2}}dx =\infty,\;\;
\hbox{or}\;\; \int_z^{1} \left( \int_x^za(t)\,dt\right)^{-\frac{1}{2}}dx =\infty.
\end{equation}
This condition measures how smooth is  the function $a(x)$ on the left, or on the right,
of $z$; it is easily seen that   it holds  whenever $a(x) $ is differentiable at $z$.
  Surprisingly, Theorem \ref{th7.2}  holds true regardless the particular behavior
of $f(u)$ at zero and at infinity, just requiring   $f(u)$ to be continuous and positive.  Thus, Theorem \ref{th7.1} is a quite general and versatile result, applying   to a large variety of situations.  In particular, it completes and sharpens, very substantially,  some of our previous findings in  \cite{LOR1,  LOR2,  LGO-19, LGO-JDE, LGO-20}.   As a direct consequence of this regularity result, the global bifurcation diagrams of the set of positive solutions   of \eqref{1.1}  look like  those superimposed in Figure \ref{Fig02} according to the decay rate of the potential at the origin, measured by $p>0$,   and at infinity. In the global bifurcation diagrams  plotted   in Figure \ref{Fig02}, as well as in the remaining figures,  we are plotting the value of
parameter $\l$ versus the $L^\infty$-norm of $u$. Thus, each point on the plotted curves stands for a particular solution  $(\l,u)$ of \eqref{1.1}. In these bifurcation diagrams, continuous lines are filled in by regular solutions, while dashed lines consist of singular solutions.
Thanks to Theorem \ref{th7.2}, the problem \eqref{1.1}
cannot admit singular solutions not only for sublinear potentials at infinity but also for superlinear, or asymptotically linear,  potentials at infinity.
Thus, these findings complete, when $0<p\neq1$, the  regularity result of \cite{LGO-20} as well as the main theorem  of \cite{LGO-JDE}, where the non-existence of singular solutions was only established for sufficiently small $\l>0$.

\begin{figure}[!ht]
\centering{
	\includegraphics[scale=0.475]{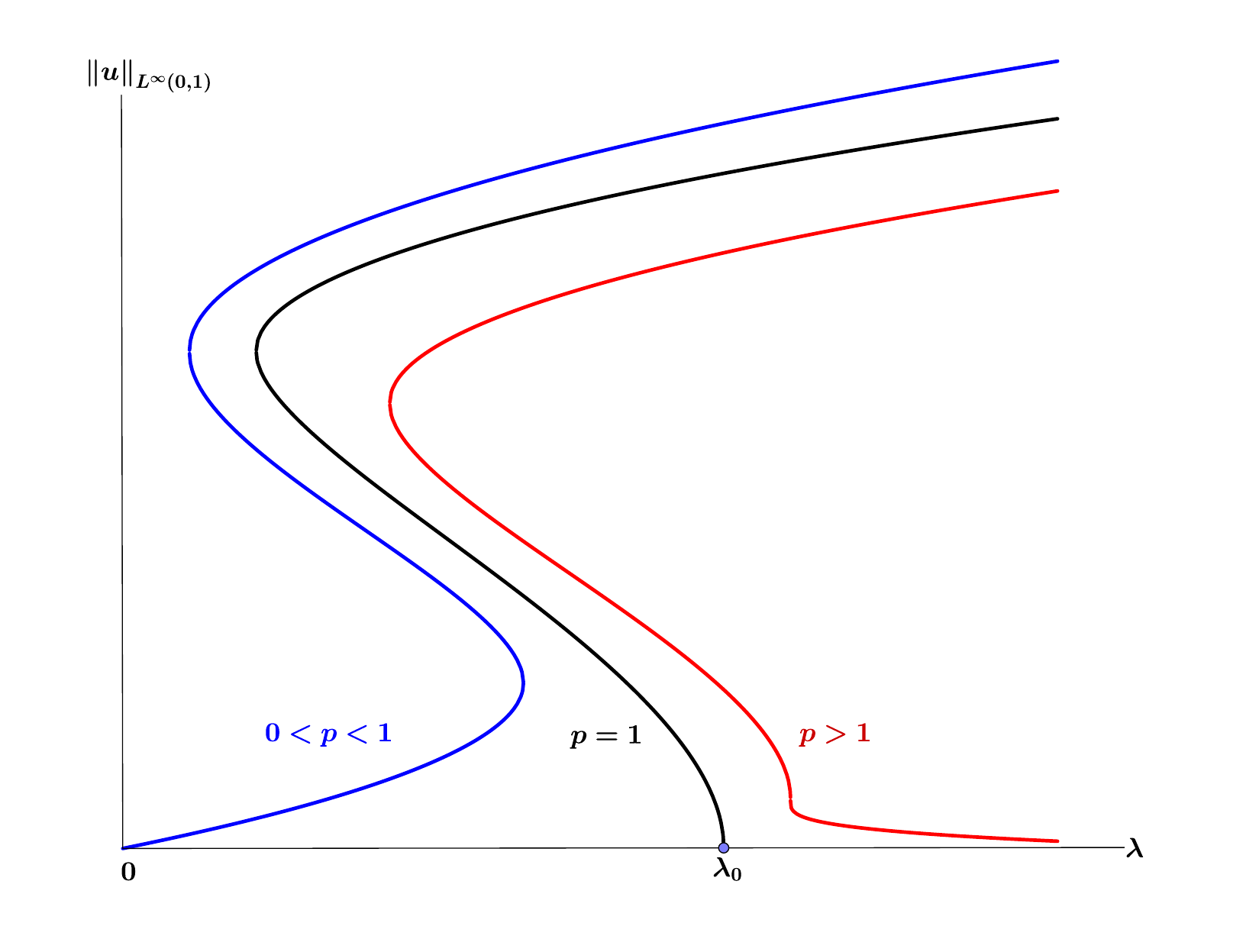}
\includegraphics[scale=0.475]{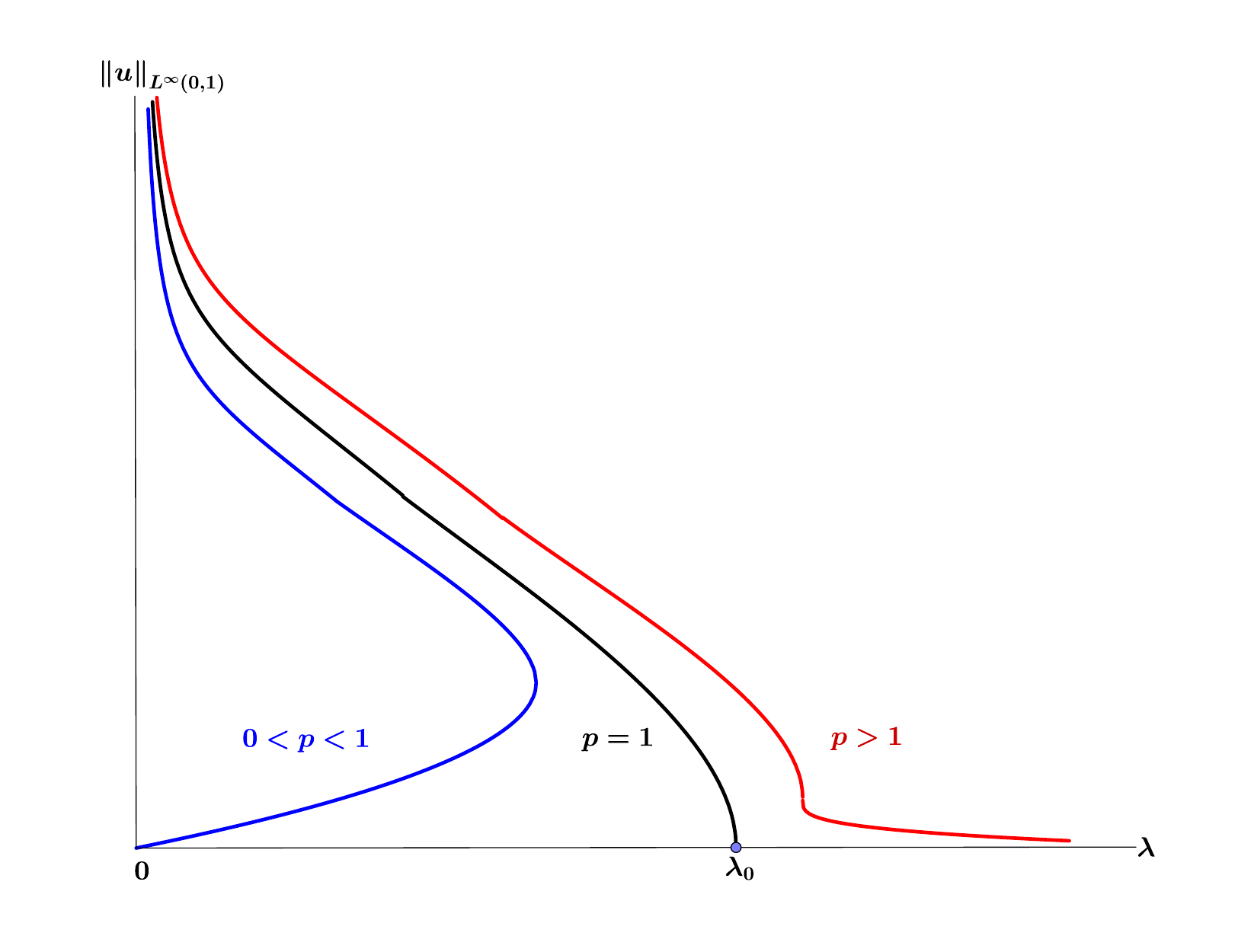}}
{\caption{Global bifurcation diagrams when $F(u)$ is sublinear at infinity   (left), or $F(u)$ is superlinear at infinity (right), and $a(x)$ satisfies $(a_2)$ and \eqref{i.12}, according to the nature of $F(u)$ at the origin: subquadratic (blue, $0<p<1$), quadratic (black, $p=1$),   or superquadratic (red, $p>1$).}
\label{Fig02}}
\end{figure}
\par
The global bifurcation diagrams   of Figure \ref{Fig02} are in  full agreement with the existence and nonexistence results of \cite{LOR1} and \cite{LOR2}, as well as with   the   new findings of this paper.
\par
According to Theorem \ref{th7.2}, the  condition \begin{equation}
\label{i.13}
\int_0^z \left( \int_x^za(t)\,dt\right)^{-\frac{1}{2}}dx <\infty\quad \hbox{and}\quad
\int_z^{1} \left( \int_x^za(t)\,dt\right)^{-\frac{1}{2}}dx<\infty
\end{equation}
is necessary for the existence of a singular solution.
This condition holds, for example, if
\begin{equation}
\label{i.14}
  \limess{x\to  z^-} a(x)
  >0>
  \limess{x\to z^+}a(x).
\end{equation}
According to the results of \cite{LOR2},      the small solutions of \eqref{1.1}
must be regular. So, a further goal of Section \ref{s7} is analyzing the formation of singularities from these
small regular solutions as  $\l$ varies.   By  Theorem~\ref{th7.3}, under conditions $(a_2)$ and \eqref{i.13}, there are examples of functions $f(u)$ satisfying $(f_1)$ for which \eqref{1.1} possesses singular solutions. Moreover, regardless $f(u)$, when $a(x)$ satisfies $(a_2)$ and \eqref{i.14}, then, any sufficiently large solution of \eqref{1.1} for sufficiently large $\l$
must be singular. Therefore, the solutions of \eqref{1.1} whose existence is guaranteed by \cite[Thm. 1.4, Rem. 1.9]{LOR1} for  sufficiently large $\lambda >0$ must be singular. Figure \ref{Fig03} provides us with   six  admissible bifurcation diagrams when the function $a(x)$ satisfies $(a_2)$ and \eqref{i.14},   according to the nature of $F(u)$ at infinity.
\par
\begin{figure}[!ht]
\centering{
\includegraphics[scale=0.475]{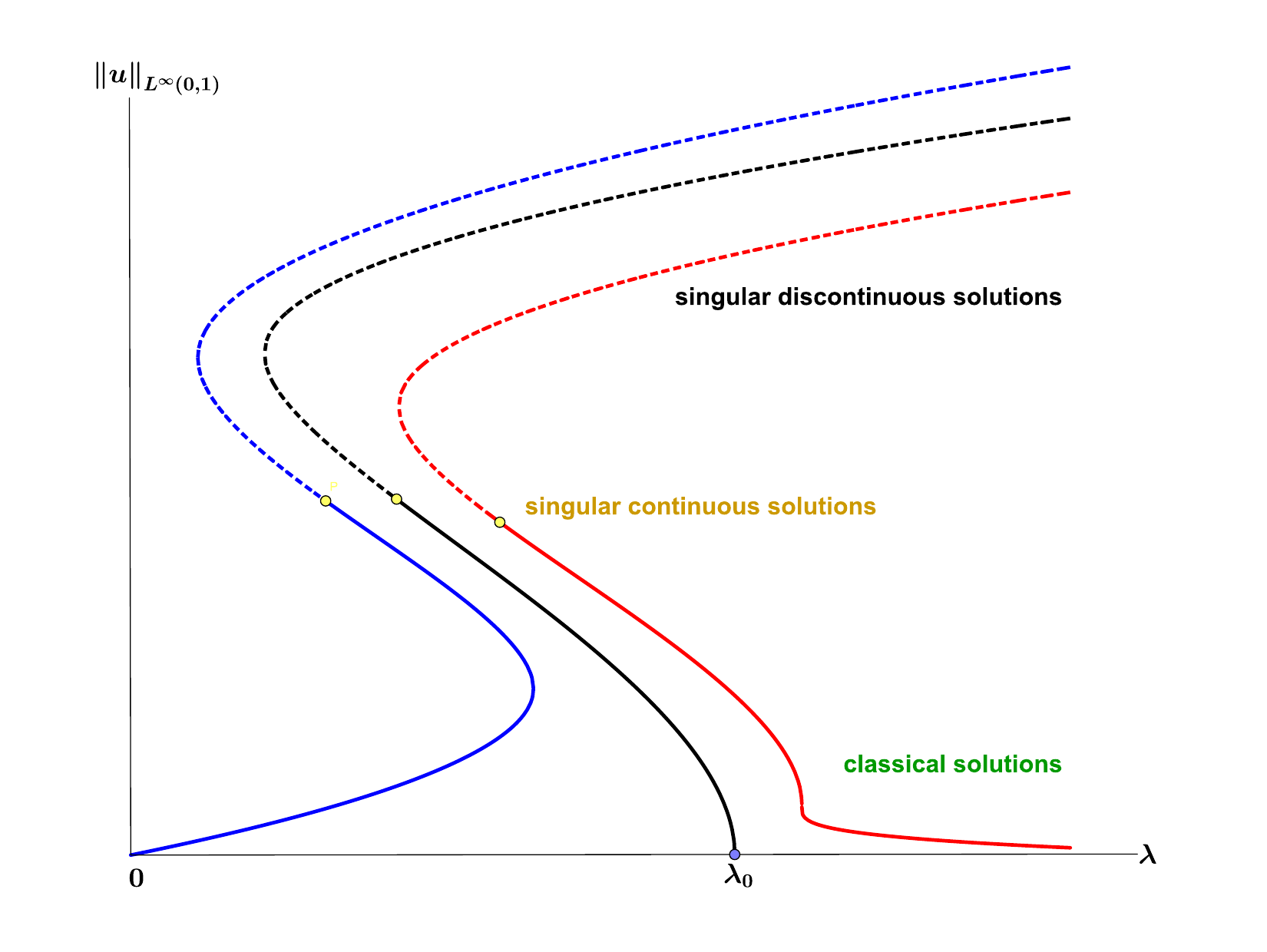}
\includegraphics[scale=0.475]{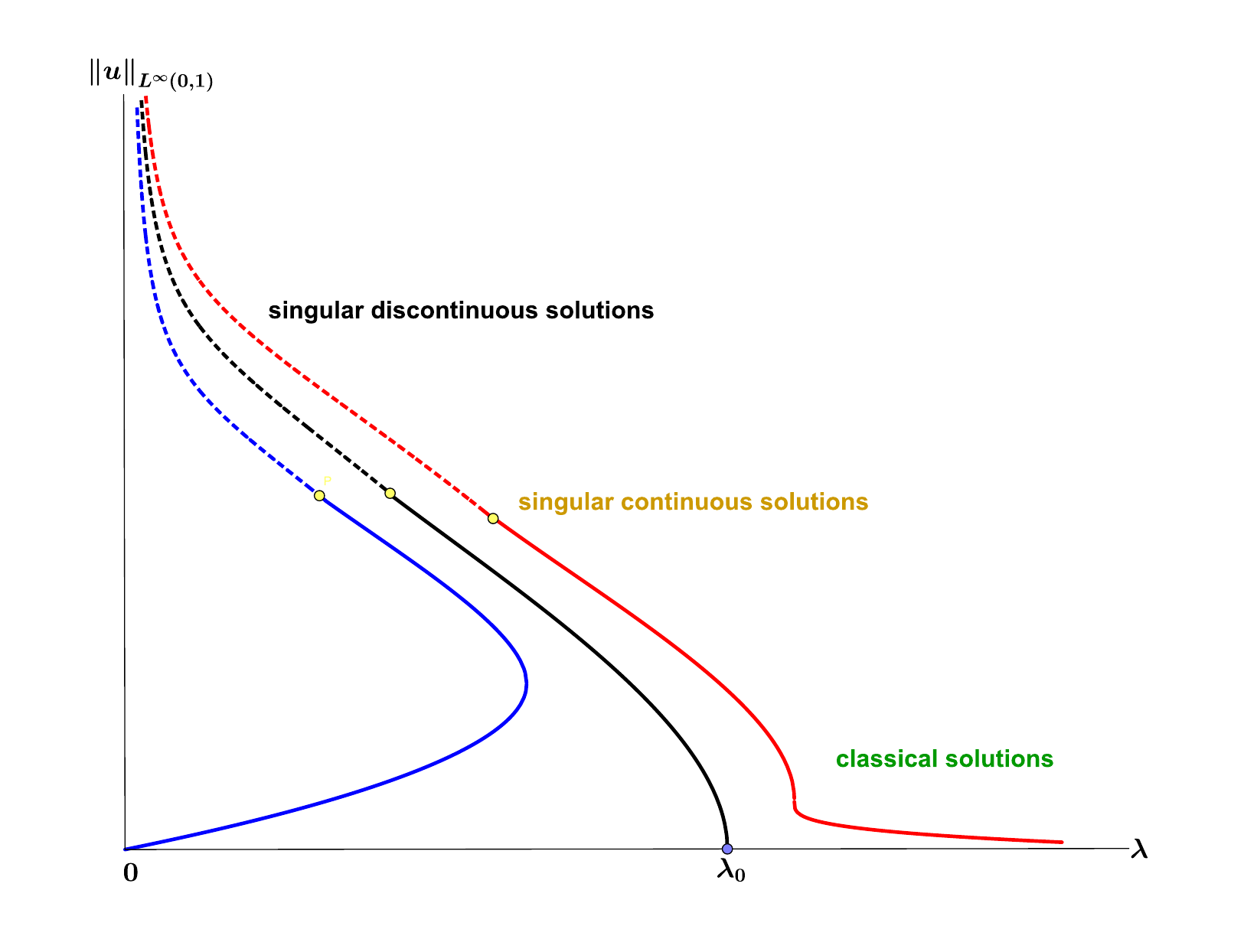}}
{\caption{Global bifurcation diagrams when $F(u)$ is sublinear   at infinity (left), or superlinear at infinity (right), and  $a(x)$ satisfies $(a_2)$ and \eqref{i.13}, according to the nature of $F(u)$ at the origin: subquadratic (blue, $0<p<1$), quadratic (black, $p=1$),  or superquadratic  (red, $p>1$).}
\label{Fig03}}
\end{figure}

In strong contrast with the   situations described in Figure \ref{Fig02}, under condition \eqref{i.14}, the small solutions of \eqref{1.1} are regular, whereas the solutions far away from zero, for sufficiently large $\l>0$, may become singular. Therefore, the small regular solutions on each of the
components plotted in Figure \ref{Fig03}  develop singularities as they become sufficiently large at the points of the bifurcation diagrams separating continuous and  dashed lines. Although the results of Section \ref{s7} guarantee the existence of singular solutions for sufficiently large $\l>0$ for sublinear potential at infinity, we were unable to guarantee the formation of singularities from the small regular solutions along the solution components of \eqref{1.1} just assuming the weaker condition \eqref{i.13}.
 So, this problem remains open here.

\section{Non-existence of   solutions  for small $\lambda >0$  when $p\ge 1$}
\label{s2}

\noindent
This section establishes the non-existence of positive solutions (regular or singular) for sufficiently small $\l > 0$ when $p\ge 1$. Recall that in this context the  positive solutions of \eqref{1.1}  are actually strictly positive. Our first result establishes  the non-existence of   singular solutions of \eqref{1.1} when $f(u)$ is globally bounded in $[0,\infty)$ and $\l\ge0$ is sufficiently small.

\begin{lemma}
\label{leii.1}
Assume $(f_1)$ and $(a_1)$.
Then, the problem \eqref{1.1} has no  positive singular solution for sufficiently small $\l\ge 0$.
\end{lemma}

\begin{proof}
Let $u$ be a positive bounded variation solution of \eqref{1.1} for some $\l\ge0$.
Set $h(x) = a(x)  f(u(x)), $ for a.e. $x \in [0,1]$.
Hence, $u$ is  a solution of the problem
\begin{equation}
   \left\{ \begin{array}{ll}
   \displaystyle
   -\left( \frac{u'}{\sqrt{1+(u')^2}}\right)' =\l h(x), & \quad 0<x<1, \\[1ex]
   u'(0)=u'(1)=0. & \end{array}\right.
\end{equation}
Since $\|h\|_{L^1(0,1)} \le\| f\|_\infty   \| a \|_{L^1(0,1)}$, there exists  $\overline \l>0$ such that
$\l \|h\|_{L^1(0,1)} <1$, for all $\l\in [0,\overline \l)$.   Thus, by the regularity result \cite[Cor. 3.5]{LGO-19}, $u \in W^{2,1}(0,1)$ and  therefore $u$ is a regular  solution of \eqref{1.1}.
\end{proof}

The next result provides information on the asymptotic behavior of the positive, necessarily regular, solutions as $\l \to 0$.

\begin{lemma}
\label{leii.2}
Assume $(f_1)$ and $(a_1)$.
Let  $\{(\l_n,u_n)\}_{n\geq 1}$ be a sequence of positive  regular solutions of   \eqref{1.1}  such that $\l_n>0$  for all   $n\geq 1$ and
\begin{equation}
\label{ii.1}
  \lim_{n\to \infty}\l_n=0.
\end{equation}
 Then,  one has that
 \begin{equation}
\label{ii.2}
  \lim_{n\to \infty}u_n=0 \quad \hbox{in } W^{2,1}(0,1).
\end{equation}
\end{lemma}
\begin{proof}
Let  $\{(\l_n,u_n)\}_{n\geq 1}$  be any  sequence of  positive regular solutions of \eqref{1.1} such that  $\l_n>0$, for all $n\ge 1$, and \eqref{ii.1} holds.
Let us set, for every $n$,
$$
\displaystyle  \psi_n  =  \frac{-u_n'}{\sqrt{1+ (u_n')^2} }\in W^{1,\infty}(0,1).
$$
Pick any $x  \in (0,1]$. Integrating the  equation   of  \eqref{1.1} in $[0,x]$ yields
$$
 \psi_n(x)  =\l_n \int_0^{x} a(t)f(u_n(t))\,dt
$$
and hence
\[
  \|\psi_n\|_{L^\infty(0,1)}   \leq \l_n \|f\|_\infty \|a\|_{L^1(0,1)}.
\]
Consequently,   by \eqref{ii.1}, we find  that $\lim_{n\to \infty} \|\psi_n\|_{L^\infty(0,1)}= 0$ and therefore
\begin{equation}
\label{ii.3}
 \lim_{n\to \infty} \|u'_n\|_{L^\infty(0,1)}=0 .
\end{equation}
For each $n$, let $x_n \in [0,1] $ be such that
\begin{equation}
\label{ii.4}
  u_n(x_n)=\|u_n\|_{L^\infty(0,1)}.
\end{equation}
Let us write, for all $n\geq 1$ and $x\in [0,1]$,
\begin{equation}
\label{ii.5}
u_n(x) = u_n(x_n) + \int_{x_n}^x u'_n (t) \, dt .
\end{equation}
For a subsequence, still labeled by $n$,  we have that either
\begin{equation}
\label{ii.6}
\lim_{n\to \infty} u_n(x_n) = \infty,
\end{equation}
or
\begin{equation}
\label{ii.7}
 \lim_{n\to \infty} u_n(x_n) =  u_\omega \in [0,\infty).
\end{equation}
In the former  case,  thanks to \eqref{ii.3}, \eqref{ii.4}  and \eqref{ii.6}, we infer from \eqref{ii.5} that
\begin{equation*}
  \lim_{n\to \infty}\frac{u_n(x)}{\|u_n\|_{L^\infty(0,1)}}   = 1 \quad \hbox{uniformly in}\;\; [0,1].
\end{equation*}
  By $(f_1)$, this implies that
\[
  \lim_{n\to \infty} \frac{f(u_n(x))}{\|u_n\|_{L^\infty(0,1)}^{-q}}=\lim_{n\to \infty}
  \frac{f(u_n(x))}{u_n^{-q}(x)}\lim_{n\to \infty} \frac{u_n^{-q}(x) }{\|u_n\|_{L^\infty(0,1)}^{-q}} = h
  \quad \hbox{uniformly in}\;\; [0,1].
\]
On the other hand, integrating the equation of \eqref{1.1} in $[0,1]$ yields, for all $n\geq 1$,
\begin{equation}
\label{ii.8}
  \int_0^1 a(x)f(u_n(x))\,dx=0
 \end{equation}
and, hence, $\int_0^1a(x)\frac{f(u_n(x))}{\|u_n\|_{L^\infty(0,1)}^{-q}}\,dx=0$. Thus, letting $n\to \infty$ and using $(f_1)$, we obtain that $ h \int_0^1a(x) \, dx=0$. As $h>0$, this contradicts $(a_1)$, which  requires  $\int_0^1a(x)  \, dx <0$.  So, \eqref{ii.6} cannot occur. Consequently, the condition \eqref{ii.7} holds. In this case,  we   infer from \eqref{ii.5} and \eqref{ii.3}  that $\{u_n\}_{n\geq 1}$ converges to $ u_\omega$ in $C^1[0,1]$. Hence, letting $n\to \infty$ in \eqref{ii.8} yields $f(u_\o)\int_0^1a(x)\,dx =0$. Consequently, since  $\int_0^1a(x) \, dx <0$, we get $f(u_\o)=0$. By $(f_1)$, we necessarily have that $u_\o =0$. Therefore,  we can conclude from  \eqref{ii.5} that $\{u_n\}_{n\ge 1}$ converges to $ 0$ in $C^1[0,1]$, and actually, by \eqref{1.3}, in $W^{2,1}(0,1)$. This ends the proof.
\end{proof}

The next result establishes the non-existence of positive solutions of \eqref{1.1} if  $p\ge 1$
when $\l>0$ is small.  Whereas in case $0<p<1$, by   \cite[Thm. 9.1]{LOR2} and Lemma \ref{leii.1},    \eqref{1.1} possesses only regular solutions for sufficiently small $\l>0$.

\begin{theorem}
\label{thii.1}
Assume $(f_1)$, with $p\ge1$, and $(a_1)$.
Then, the problem \eqref{1.1} has no positive solutions for sufficiently small $\l>0$.
\end{theorem}

\begin{proof}
Suppose by contradiction that there exists a sequence $\{(\l_n,u_n)\}_{n\geq 1}$ of
(strictly) positive solutions   of \eqref{1.1} with $\l_n >0$ for all $n$ and such that
$\lim_{n\to\infty}\l_n=0$. By Lemmas \ref{leii.1} and \ref{leii.2},  we can suppose that all these solutions are regular and \eqref{ii.2} holds. Let us set, for every  $n\geq 1$  and a.e. $ x\in [0,1]$,
 \begin{equation*}
  a_n(x) =  a(x)\left[1+(u_n'(x))^2\right]^\frac{3}{2}\frac{f(u_n(x))}{u_n^p(x)}.
\end{equation*}
Then, due  to \eqref{1.3},   each $u_n$ satisfies
\begin{equation}
\label{ii.9}
  \left\{ \begin{array}{l} -u_n''=\l_na_n(x) u_n^p,\quad 0<x<1, \\[1ex]
  u_n'(0)=u_n'(1)=0.\end{array}\right.
\end{equation}
By \eqref{ii.2}, we have that $\lim_{n\to\infty} \left[1+(u_n'(x))^2\right]^\frac{3}{2} = 1$ uniformly in $[0,1]$, and, by \eqref{1.2},
\begin{equation}
\label{ii.10}
  \lim_{n\to \infty} \frac{f(u_n(x))}{u_n^p(x)}=1\quad \hbox{uniformly in}\;\; [0,1].
\end{equation}
Thus,  from these facts, we  infer that
\begin{equation}
\label{ii.11}
  \lim_{n\to \infty}a_n = a \quad \text{in } L^1(0,1) .
\end{equation}
Subsequently, we define, for every  $n\geq 1$,  $v_n  =  \frac{u_n}{\|u_n\|_{L^\infty(0,1)}}$. By \eqref{ii.9},   each $v_n$ satisfies
\begin{equation*}
  \left\{ \begin{array}{l} -v_n''= \l_n a_n(x) u_n^{p-1}v_n,\quad 0<x<1, \\
  v_n'(0)=v_n'(1)=0.\end{array}\right.
\end{equation*}
and thus
$$
  \|v''_n\|_{L^1(0,1)} \le \l_n \|a_n\|_{L^1(0,1)} \|u_n\|^{p-1}_{L^\infty(0,1)}.
$$
Hence, from \eqref{ii.1}, \eqref{ii.11}, \eqref{ii.2} and  the assumption  $p\geq 1$, we find that
$\lim_{n\to \infty}v''_n=0$ in $L^1(0,1)$. Writing down, for every    $n\geq 1 $   and $x\in [0,1]$,
$$
v'_n(x) = v'_n(0) + \int_0^x v_n''(t) \, dt
\quad \hbox{and} \quad v_n(x) = v_n(x_n)+ \int_{x_n}^x v_n'(t) \, dt,
$$
where $x_n\in [0,1]$ is taken so that $v_n(x_n) = \|v_n\|_{L^\infty(0,1)} =1$, it is easily seen that
\begin{equation}
\label{ii.12}
  \lim_{n\to \infty}v_n=1 \quad  \hbox{in }\; {W}^{2,1}(0,1).
\end{equation}
As $\int_0^1 a(x) f(u_n(x)) \, dx =0$ holds for every   $n\geq 1$, by  \eqref{ii.10} and  \eqref{ii.12}, we get
\begin{align*}
0 &= \lim_{n\to\infty}\int_0^1 a(x) f(u_n(x)) \, dx
\\&= \lim_{n\to\infty} \frac{1} {\|u_n\|^p_{L^\infty(0,1)} } \int_0^1 a(x) f(u_n(x)) \, dx
\\
&= \lim_{n\to\infty}  \int_0^1 a(x) \frac{f(u_n(x))}{u^p_n(x)} \frac{u^p_n(x)} {\|u_n\|^p_{L^\infty(0,1)}  } \, dx
\\
&=\int_0^1 a(x) \left( \lim_{n\to\infty}   \frac{f(u_n(x))}{u^p_n(x)} \right) \left(  \lim_{n\to\infty}   v_n^p(x)  \right) dx
=\int_0^1 a(x) \, dx,
\end{align*}
which is impossible, because we are assuming, by $(a_1)$, that $\int_0^1a(x)\,dx<0$. This
contradiction ends the proof.
\end{proof}

\section{  Global Bifurcation from $(\l,u)=(0,0)$ and  $(\l,u)=(\l_0,0)$ when $p=1$} \label{s3}

\noindent  Our main goal in this section is to prove, under assumptions  $(a_1)$ and $(f_1)$, with $p=1$,  the existence of connected components of the set of the positive   solutions of \eqref{1.1}, which   are indeed strictly positive if they are regular, or  if condition $(a_2)$ holds. Thus,
we  generally suppose that the functions $a(x)$ and $f(u)$  satisfy $(a_1)$  and $(f_1)$, with $p=1$, except in the last theorem, where $(a_2)$  and $(f_3)$, with $p=1$, are assumed. In the subsequent analysis  the weighted eigenvalue problem \eqref{i.10} plays a pivotal role.
\par
We start recalling that,  thanks to   \cite[Thm. 1.4, Rem. 1.9]{LOR1}, the problem \eqref{1.1} admits positive solutions   for  sufficiently large $\l>0$.
Some changes in the proof yield   the following sharper result, which seems optimal in
the sense that \eqref{1.1} might not admit any positive solution for $\l\leq \l_0$.

\begin{theorem}
\label{thiii.1}
Assume $(f_1)$, with $p=1$, and $(a_1)$. Then, for every $\l > \l_0$,  the problem \eqref{1.1} has at least one   positive  solution.
\end{theorem}

\begin{proof}
Fix any $\l  >\l_0$. We will find a positive bounded variation solution $u$ of \eqref{1.1}  as a global minimizer of
 the functional $\JJ : BV(0,1) \to \RR$  defined by
$$
\JJ(u)  = \int_0^1 (\sqrt{1+ (Du^a(x))^2}- 1) \,dx
+\int_0^1   |Du^s |   - \l \int_0^1 a(x)\, F(u(x)) \, dx.
$$
It is plain that, without  loss of generality,  we can suppose that $F(u)$ is an even function.

We first prove that, under   $(f_1)$ and $(a_1)$, $\JJ$ is coercive and bounded from below in $BV(0,1)$. Indeed, setting   $\kappa = \frac{h}{1-q} $, the condition
\eqref{Fq} entails that
for every $\varepsilon >0$ there exists $c_\varepsilon >0$ such that
\begin{equation}
\label{iii.1}
    \left|F(u) - \kappa |u|^q \right| \le \varepsilon  \, |u|^q + c_\varepsilon \quad \hbox{for all}\;\; u\in\RR.
\end{equation}
Hence, setting
\[
  r = \int_0^1 u(x) \, dx\quad \hbox{and}\quad w=u-r \quad \hbox{for every}\;\; u\in BV(0,1),
\]
it follows from the Jensen inequality that
\begin{equation}
\label{iii.2}
\begin{split}
 \JJ(u) = &\int_0^1 (\sqrt{1+ (Dw^a(x))^2} -1) \, dx   +\int_0^1   |Dw^s |  -
  \l\int_0^1 a(x) \, F(u(x)) \, dx
\\ \ge &  \,  \sqrt{1+\|Dw^a\|_{L^1(0,1)}^2}   -1 +  \int_0^1   |Dw^s |
   - \l \int_0^1 a(x) \, F(u(x)) \, dx.
\end{split}
\end{equation}
On the other hand, since $q\in (0,1)$,    by the Poincar\' e-Wirtinger  inequality, we find that, for a.e. $x\in [0,1]$,
\begin{equation*}
|u(x)|^q-|r|^q \leq \big| |w(x)+r|^q-|r|^q \big|
   \leq |w(x) |^q \leq \|w\|_\infty^q \leq \|Dw\|^q,
\end{equation*}
where
$$
  \|Dw\| = \int_0^1 |Dw^a(x)| \, dx  + \int_0^1 |D^sw|
$$
is the variation of $w$.  Thus,   thanks to \eqref{iii.1}, we  find that
\begin{align*}
  \int_0^1 a (x) F(u(x)) \, dx & = \int_0^1 a (x) \left(F(u(x)) - \kappa |u(x)|^q \right)\, dx \\ &
\quad +\kappa  \int_0^1 a (x)  \left( |u(x)|^q -  |r|^{q} \right) \, dx+\kappa   |r|^{q}\int_0^1 a(x) \, dx \\ & \leq \int_0^1 |a(x)|   \left( \e|u(x)|^q+c_\varepsilon \right)\,dx +\kappa \|a\|_{L^1(0,1)}  \|Dw\|^q + \kappa   |r|^{q}\int_0^1 a(x)\, dx
\\ &  \leq \|a\|_{L^1(0,1)} \, \left( (\varepsilon+\kappa)\|Dw\|^q +\e  |r|^q +c_\varepsilon \right)+
\kappa   |r|^{q}\int_0^1 a(x) \, dx .
\end{align*}
Consequently, applying this estimate to \eqref{iii.2} easily yields
\begin{align*}
\JJ(u)  \geq \;   &   \|Dw\| - \l \|a\|_{L^1(0,1)}  ( \varepsilon +\kappa)  \|Dw\|^q \\
    & - \l  \left( \kappa \int_0^1  a(x)\, dx  + \varepsilon\|a\|_{L^1(0,1)} \right)   |r|^{q}- \l c_\varepsilon \|a\|_{L^1(0,1)} -1.
\end{align*}
Thus, as we are assuming that $\int_0^1a(x) \, dx <0$, we can take    $\varepsilon >0$ so small that
\begin{equation}
 \kappa \int_0^1  a(x)\, dx  + \varepsilon\|a\|_{L^1(0,1)}<0.
\end{equation}
Hence, it is plain  that we can find  two
 constants $A>0, B>0$  such that
\begin{equation}
\label{iii.3}
\JJ(u) \ge
 A \big(\|Dw\| +  |r|^q \big) - B.
 \end{equation}
Condition \eqref{iii.3} implies that
\begin{equation*}
\lim_{\|u\|_{BV(0,1)} \to +\infty} \JJ(u)  =+\infty
\quad \text{and} \quad
\inf_{u\in BV(0,1) } \JJ(u) >- \infty .
\end{equation*}
Since  the functional $\JJ$ is lower semicontinuous with respect to the $L^1$-convergence in $BV(0,1)$, it is a classical fact (see, e.g., \cite{Em}) that $\JJ$ admits a global minimizer $u  \in BV(0,1)$. Moreover, by \cite{Anz83},   any minimizer of $\JJ$ is a bounded variation solution of the problem \eqref{1.1}.
\par
Next, we will prove that, thanks to the choice  $\l >\l_0$, $u$ is non-trivial. To this end, it suffices to show that $\JJ(u)<0$. Condition \eqref{Fp}, with $p=1$,  implies that,  for every sequence $\{s_n\}_{n\ge1}$, with $s_n>0$ for all $n\ge 1$,  such that
\begin{equation*}
\lim_{n\to\infty} s_n = 0
\qquad\text{and}\qquad
\lim_{n\to\infty} \frac{F(s_n)}{s_n^2} = \frac{1}{2},
\end{equation*}
one has that
\begin{equation*}
\lim_{n\to\infty} \frac{F( s_n\varphi(x))}{s_n^2\varphi^2(x)} = \frac{1}{2} \quad \text{uniformly in } x \in [0,1].
\end{equation*}
 Thus, we get
\begin{align*}
\lim_{n\to\infty}
\int_0^1  & \Big( \frac{(D^a\varphi(x))^2}{1+ \sqrt{1+s_n^2 (D^a\varphi(x))^2}}
   -  \l a(x)  \, \frac{F(s_n\varphi(x))}{s_n^2\varphi^2(x)} \varphi^2(x) \Big) \,  dx
\\
&
= \frac{1}{2} \int_0^1 \Big( (D^a\varphi(x))^2   -\l  a(x) \varphi^2(x) \Big) \,  \mathrm{d}x
  = \frac{1}{2} \int_0^1 \Big( 1 -\frac{\l}{\l_0}  \Big) (D^a\varphi(x))^2(x)  \,  dx  < 0.
\end{align*}
We therefore  can  conclude that
\[
\JJ(s_n\varphi ) = s_n^2
\int_0^1 \Big( \frac{(D^a\varphi(x))^2}{1+ \sqrt{1+s_n^2 (D^a\varphi(x))^2}}
-  \l a(x)  \, \frac{F(s_n\varphi(x))}{s_n^2\varphi^2(x)} \varphi^2(x) \Big) \,  dx <0,
\]
  for large $n$. This clearly implies that  $\JJ(u)<0$.
\par
Finally, we  show that $u$ can be chosen to be positive. Indeed, since
$$
\JJ(|u|) = \JJ(u) \quad \hbox{for all}\;\; u\in BV(0,1),
$$
we see that if $u$ is a global minimizer of $\JJ$, then $|u|$ is a global minimizer too.
\end{proof}

We recall that $\mc{S}_{bv}^+$ denotes  the set of couples $ (\lambda, u)  \in [0,\infty) \times BV(0,1)$ such that $(\l,u)$ is a positive (bounded variation) solution of \eqref{1.1},  together with $(0,0)$ and $(\l_0,0)$, its two possible bifurcation points from the trivial line  $(\l,0)$, $\l\in\RR$. Similarly, $\mc{S}_r^+$ stands for
  the set of couples $ (\lambda, u)~\in~[0,\infty) \times \mc{C}^1[0,1]$ such that $(\l,u)$ is a positive regular solution of \eqref{1.1}, together with $(0,0)$ and $(\l_0,0)$. Finally, $ \mc{S}_s^+=\mc{S}_{bv}^+\setminus \mc{S}_r^+$ is the set of the singular positive solutions of \eqref{1.1}.

 The following result, going back to \cite[Thm. 3.1 and 3.2]{LOR2},
establishes the existence of two  components of $\mc{S}_r^+$ bifurcating
from $(\l,0)$ at $\l=0$ and at $\l=\l_0$. By a {\em component} of $\mc{S}^+_r$ it is meant a
closed connected subset  of $\mc{S}^+_r$, equipped with the topology of  $\RR\times   \mc{C}^1[0,1]$,
which is  maximal for the inclusion. Note that the regularity
requirements on $f(u)$ in the next result have been slightly relaxed with respect to those imposed in \cite{LOR2}; they anyhow allow to apply the results in \cite{LG01}, in particular \cite[Thm. 6.4.3]{LG01}, to achieve the conclusions. Subsequently, we denote by $\mc{P}_\l$ the $\l$-projection operator,
$\mc{P}_\l(\l,u)=\l$.

\begin{theorem}
\label{thiii.2} Assume that
$f\in\mc{C}(\R) \cap \mc{C}^1(-\eta,\eta)$, for some $\eta>0$, $f'(0)=1$,
and   { $(a_1)$}. Then, the following assertions hold:
\begin{enumerate}
\item[{\rm (a)}] there exists an unbounded  component $\mf{C}_{r,\l_0}^+$
of $\mc{S}^+_r$ such that
\begin{itemize}
\item $(\l_0,0)\in  {\mf{C}}_{r,\l_0}^+$;
\item $\mc{P}_\l(\mf{C}_{r,\l_0}^+)\subseteq [0,\infty)$;
\item $\l=\l_0$ if $(\l,0)\in \mf{C}_{r,\l_0}^+$ with $\l \neq 0$;
\item $\min u >0$  if $(\l,u)\in \mf{C}_{r,\l_0}^+\setminus\{(0,0),(\l_0,0)\}$.
\end{itemize}
\item[{\rm (b)}]  there exists an unbounded component $\mf{C}_{r,0}^+$ of $\mc{S}^+_r$
such that
\begin{itemize}
\item $\{0\} \times [0, \infty) \subseteq \mf{C}_{r,0}^+$;
\item $\mc{P}_\l(\mf{C}_{r,0}^+)\subset [0,\infty)$;
\item $\l=\l_0$ if $(\l,0)\in \mf{C}_{r,0}^+$ with $\l \neq 0$;
\item $\min u >0$   if $(\l,u)\in \mf{C}_{r,0}^+\setminus\{(0,0),(\l_0,0)\}$.
\end{itemize}
\end{enumerate}
  Moreover,
when $(f_1)$ holds, and hence $F(u)$ is sublinear at infinity, we have that
\begin{equation}
\label{2.1}
  \mf{C}_{r,0}^+\cap \mf{C}_{r,\l_0}^+=\emptyset,
\end{equation}
and, in particular, $(0,0)\notin \mf{C}_{r,\l_0}^+$ and $(\l_0,0)\notin \mf{C}_{r,0}^+$.
\end{theorem}

 The last assertion of Theorem \ref{thiii.2} is a direct consequence of Theorem \ref{thii.1}  and shows that,  much like in the cases when $F(u)$ is superlinear at infinity, or asymptotically linear at infinity, which have been previously treated in \cite{LOR2}, \cite{LGO-19} and \cite{LGO-20}, also when $F(u)$ is sublinear at infinity the two components $\mf{C}_{r,0}^+$ and $ \mf{C}_{r,\l_0}^+$ are disjoint.

When, in addition, $f \in \mc{C}^2(-\eta,\eta)$, then one can invoke
\cite[Thm. 1.7]{CR} in order to complement Theorem \ref{thiii.2} with the next result, of a local nature, which basically goes back to
\cite[Thms. 4.1 and 4.2]{LOR2}. Theorem \ref{thiii.3}  also corrects a wrong assertion made in \cite[Thm 4.2]{LOR2} concerning the bifurcation directions.

\begin{theorem}
\label{thiii.3}
Assume that $f\in\mc{C}(\R)\cap \mc{C}^\nu(-\eta,\eta)$, for some $\eta >0$ and $\nu\geq 2$, $f'(0)=1$,   and $(a_1)$. Then, in a neighborhood of $(\l,u)=(0,0)$, the component
$\mf{C}_{r,0}^+$ consists of the curve $\{(0,\kappa) : \kappa \in [0,\kappa_0)\}$ for some $\kappa_0>0$. Similarly, setting
$$
  V  =  \left\{ v \in \mc{C}^1[0,1]\; : \; \int_0^1 v(x)\v(x) \,dx =0\right\},
$$
where $\v$ is any positive eigenfunction associated with \eqref{i.10}, there exist $\e>0$ and two maps of class $\mc{C}^{\nu-1}$, $\l : (-\e,\e)\to \R$ and $v: (-\e,\e)\to V$, such that
\begin{itemize}
\item[{\rm (i)}] $\l(0)=\l_0$ and $v(0)=0$;
\item[{\rm (ii)}] $(\l(s),s(\v+v(s)))$ solves \eqref{1.1} for all $s \in (-\e,\e)$;
\item[{\rm (iii)}] in a neighborhood of $(\l,u)=(\l_0,0)$, $\mf{C}_{r,\l_0}^+$ consists of the smooth arc of curve $(\l(s),s(\v+v(s)))$, with $s\in [0, \varepsilon)$.
\end{itemize}
Moreover, the following holds:
\begin{align}
\label{iii.4}
  \l'(0)
  =  -  \l_0  f''(0)\frac{\int_0^1\v(x) (\v'(x))^2  \,dx}{\int_0^1 (\v'(x))^2  \,dx}
\end{align}
and, if $\nu\geq 3$ and $f''(0)=0$,
\begin{align}
\label{iii.5}
  \l''(0)   = - \frac{ f'''(0)   \int_0^1 \v^2(x)(\v'(x))^2\,dx+
  \int_0^1  (\v'(x))^4 \,dx} {\int_0^1 (\v'(x))^2  \,dx}.
\end{align}
Thus, the component $\mf{C}_{r,\l_0}^+$ bifurcates subcritically at $\l=\l_0$ if $f''(0)> 0$, or if $f''(0)=0$ and
$$
f'''(0)  > - \frac{\int_0^1 (\v(x)')^4\,dx} {\int_0^1 \v^2(x) (\v'(x))^2 \,dx},
$$
 while it does it supercritically if $f''(0)<0$, or if $f''(0)=0$ and
$$
f'''(0)  < - \frac{\int_0^1 (\v(x)')^4\,dx} {\int_0^1 \v^2(x) (\v'(x))^2 \,dx}.
$$
\end{theorem}

\begin{proof}
Since assertions (i)--(iii) follow from \cite[Thms. 4.1 and 4.2]{LOR2},
we only provide the proof of formulas \eqref{iii.4} and \eqref{iii.5}. In the course of this proof,
  in order to simplify the notation, the dependence on $x$ is not indicated. Set
\begin{equation*}
     u(s)=s(\v+v(s)) \qquad \hbox{for all}\;\; s\in (-\e,\e).
\end{equation*}
Substituting   $(\l,u)=(\l(s),u(s))$  in \eqref{1.3} and dividing by $s$, we find that
\begin{align*}
  -(\v+sv_1 + o(s))''  =(\l_0+s\l_1&+o(s)) a(x) (\v+sv_1+o(s)) \cr & \cdot
  \Big[ 1+ \frac{f''(0)}{2}s(\v+sv_1+o(s))+o(s)\Big]
  \Big[1+\frac{3}{2}(\v')^2s^2+o(s^2) \Big]
\end{align*}
for sufficiently small $s$,  where
$$ \l_1 = \l'(0), \quad v_1=\frac{dv}{ds}(0).$$
Particularizing at $s=0$, we get
\begin{equation}
\label{iii.6}
  -\v'' = \l_0 a \v,
\end{equation}
which is true by the definition of $\l_0$ and $\v$. Identifying terms of order $s$ yields
\[
  -v_1''=\l_0 a v_1 +\l_0 \frac{f''(0)}{2}a\v^2+ \l_1 a \v.
\]
Multiplying this equation by $\v$ and integrating by parts in $(0,1)$, we find from \eqref{iii.6} that
\begin{equation}
\label{iii.7}
  \frac{1}{2}\l_0f''(0)\int_0^1a\v^3 \, dx + \l_1 \int_0^1 a \v^2 \, dx =0.
\end{equation}
On the other hand, multiplying \eqref{iii.6} by $\v$ and $\v^2$, respectively, and integrating by parts in $(0,1)$, we get
\begin{equation}
\label{iii.8}
 \l_0 \int_0^1 a \v^2 \, dx=-\int_0^1 \v''\v \, dx=\int_0^1 (\v')^2\, dx >0
\end{equation}
and
\[
  \l_0 \int_0^1 a \v^3 \, dx=-\int_0^1 \v''\v^2 \, dx=\int_0^1 \v' (\v^2)'\, dx=2\int_0^1\v (\v')^2 \, dx,
\]
as $\v$ is positive and not constant.
Hence,  by eliminating $\l_1$ in \eqref{iii.7}, thanks to \eqref{iii.8},  we find that
\begin{equation}
\l'(0) = \l_1 = - \frac{1}{2}\l_0 f''(0) \frac{\int_0^1a\v^3 \, dx} {\int_0^1 a \v^2 \, dx}
=-\l_0 f''(0)  \frac{\int_0^1\v  (\v' )^2  \,dx}{\int_0^1 (\v' )^2  \,dx},
\end{equation}
thus proving  \eqref{iii.4}.
\par
Subsequently, we suppose $\nu \geq 3$ and $f''(0)=0$. Then, by \eqref{iii.4} we have that $\l_1=0$ and hence $-v_1''=\l_0 a v_1$. Thus, there exists $\a\in\R$ such that $v_1=\a\v$. Therefore, since $v_1\in V$, we find that $\a=0$, which implies $v_1=0$. Consequently, substituting $(\l(s),u(s))$ in \eqref{1.3} and dividing by $s$  yields
\begin{align*}
  -(\v+s^2v_2  + o(s^2))''  = (\l_0&+s^2\l_2+o(s^2)) a(x) (\v+s^2v_2+o(s^2)) \cr & \cdot
  \Big[ 1
  + \frac{f'''(0)}{6}s^2(\v +s^2v_2+o(s^2))^2+ o(s^2)\Big]
   \Big[1+\frac{3}{2} s^2(\v')^2+o(s^2) \Big],
\end{align*}
  where
$$
  \l_2 = \frac{1}{2} \frac{d^2\l}{d s^2}(0),\quad v_2= \frac{1}{2} \frac{d^2 v}{d s^2}(0).
$$
Consequently, identifying terms of order $s^2$, we obtain that
\begin{equation}
\label{iii.9}
  -v_2''=\l_0 a v_2 +\frac{3}{2} \l_0 a \v (\v')^2+ \l_2 a \v + \frac{f'''(0)}{6} \l_0a \v^3.
\end{equation}
Thus,  multiplying \eqref{iii.9} by $\v$ and integrating by parts in $(0,1)$ gives
\[
  \frac{3}{2}\l_0 \int_0^1 a \v^2(\v')^2 \, dx+\l_2 \int_0^1 a \v^2 \, dx
  +  \frac{f'''(0)}{6} \l_0 \int_0^1 a  \v^4 =0
\]
and hence, as  $ \int_0^1  a \v^2\, dx>0$ by \eqref{iii.8},
\begin{equation}
\frac{1}{2}\l''(0) = \l_2=  -  \frac{ \frac{1}{6}f'''(0) \l_0 \int_0^1a \v^4\,dx
+
 \frac{3}{2}\l_0  \int_0^1 a  \v^2(\v')^2\,dx}{\int_0^1 a \v^2\, dx}.
\end{equation}
On the other hand, multiplying \eqref{iii.6} by $\v^3$ and $\v (\v')^2$, respectively, and integrating by parts in $(0,1)$, we  get
\begin{equation}
 \l_0 \int_0^1 a \v^4 \, dx= -\int_0^1 \v''\v^3 \, dx = 3 \int_0^1 \v^2 (\v')^2 \, dx
\end{equation}
and
\[
  \l_0 \int_0^1 a \v^2 (\v')^2 \, dx
  =   -\int_0^1 \v (\v')^2\v''  \, dx
   =    -\int_0^1 \v   \frac{d}{dx}\Big(\frac{1}{3}(\v')^3\Big) \, dx
  = \frac{1}{3} \int_0^1 (\v')^4 \, dx>0.
\]
Thus,   by using \eqref{iii.8}, we can conclude that
\begin{equation}
\l''(0) =    -  \frac{  f'''(0) \int_0^1 \v^2 (\v')^2 \, dx +
\int_0^1 (\v')^4 \, dx}{\int_0^1 (\v' )^2  \,dx}
\end{equation}
and, therefore, \eqref{iii.5} is proven. The statements concerning the bifurcation directions are obvious consequences of  \eqref{iii.4} and \eqref{iii.5}.
\end{proof}

The next global bifurcation result holds true for bounded variation solutions of \eqref{1.1}.

\begin{theorem}
\label{thiii.4}
Assume    $(a_2)$ and $(f_3)$ with $p=1$.  Then, there  exist two  subsets  of  $\mc{S}_{bv}^+$, $\mf{C}^+_{bv,0}$ and $\mf{C}^+_{bv,\lambda_0}$, such that, for every $\rho>2$,

\begin{itemize}
\item $   \mf{C}^+_{bv,0} = \{0\} \times [0,\infty)$;

\item
$\mf{C}^+_{bv,0} \cap \mf{C}^+_{bv,\lambda_0} =\emptyset$;

\item
$\mf{C}^+_{bv,\lambda_0}$ is maximal in $\mc{S}^+_{bv}$ with respect to the inclusion, is connected in  $\RR \times BV(0,1)$,  having endowed $ BV(0,1)$ with the topology of the strict convergence  (cf. \cite[Def. 3.14]{AmFuPa}), and is  unbounded in $\R\times L^\rho(0,1)$;

\item
$ (\lambda,0) \in
\mf{C}^+_{bv,\lambda_0}$ if and only if $\l=\l_0$;

\item $\essinf u >0$ if $(\lambda, u) \in
\mf{C}^+_{bv,\lambda_0}$ with $u\neq0$;

\item there exists  a neighborhood $U$ of $(\lambda_0,0)$ in   $\RR \times L^\rho(0,1)$
such that $\mf{C}^+_{bv,\lambda_0} \cap U$ consists of regular solutions of \eqref{1.1}, i.e.,
\begin{equation}
\label{iii.10}
   \mf{C}^+_{bv,\l_0} \cap U =  \mf{C}^+_{r,\l_0} \cap U.
\end{equation}
\end{itemize}
\end{theorem}

\begin{proof}
Condition $(f_3)$ implies that, for every $\rho>2$, there exists a constant $\kappa
>0
$ such that
$$
   |f'(u)| \le \kappa \, ( |u|^{\r-2} + 1)\qquad \hbox{for all}\;\; u \in\RR.
$$
Therefore, Theorem \ref{thiii.4} is a direct consequence of \cite[Thm.1.1]{LGO-19} and of  Theorem \ref{thii.1}.
\end{proof}

\begin{rem}
\label{reiii.1}
According to
\eqref{iii.10}, the small bounded variation solutions of \eqref{1.1}  must be regular solutions, and thus
$\mf{C}^+_{r,\lambda_0} \subseteq \mf{C}^+_{bv,\lambda_0}.$
One of the main goals of this paper is ascertaining, whether, or not, $ \mf{C}^+_{r,\lambda_0}$
is a proper subcomponent of $\mf{C}^+_{bv,\lambda_0}$. Note that, whenever  $\mf{C}^+_{r,\lambda_0} \varsubsetneq \mf{C}^+_{bv,\lambda_0}$,
regular solutions develop singularities  along the same component.
\end{rem}

\section{Bifurcation   from   $(\l,u)= (0,0)$  when $0<p<1$}
\label{s4}

\noindent
Throughout this section, we   assume that the functions $a(x)$ and $f(u)$   satisfy $(a_2)$ and $(f_1)$ with $0<p<1$, respectively.  The main goal of this section is establishing the existence of a component of the set  $\mc{S}_r^+$   of positive regular solutions bifurcating from $(0,0)$.  Our starting point is the  next result which is a consequence of \cite[Thm. 9.1]{LOR2}.

\begin{theorem}
\label{thiv.1}
Assume $(f_1)$, with $p\in (0,1)$, and $(a_2)$. Then,  there exists $\eta>0$ such that,
for every $\l\in (0,\eta)$, the problem \eqref{1.1} has at least one  positive regular solution, $u_\l$.
Moreover, one has that
\begin{equation}
\label{iv.1}
 \lim_{\l\to 0} \|u_\l\|_{\mc{C}^1[0,1]}=0,
\end{equation}
regardless each particular choice of $u_\l$.
\end{theorem}

As the proof of  \cite[Thm. 9.1]{LOR2} is based on the direct method  of   calculus of variations, Theorem \ref{thiv.1} does not guarantee the existence of a component of $\mc{S}_r^+$ containing   these solution pairs $(\l,u_\l)$. By relying instead on the  construction of sub- and supersolutions and on the use of the topological degree,  we can complement  Theorem \ref{thiv.1} as follows.

\begin{theorem}
\label{thiv.2}
Assume $(f_1)$, with $p \in (0,1)$, and $(a_2)$.
Then, there is a component $\mf{C}_{r,0}^+$ of  $\mc{S}_r^+ $  such that
$[0,\l_*)\subseteq \mc{P}_\l(\mf{C}_{r,0}^+)$, for some $\l_*>0$,
and
\eqref{iv.1}  holds, for every $(\l,u_\l)\in\mf{C}_{r,0}^+$.
\end{theorem}

\begin{proof}
Without loss of generality, we can suppose in the course of this proof that $f\in\mc{C}(\R)$ is an odd function.
By performing the change of variable
\begin{equation}
\label{iv.2}
 u=\e v,\qquad \e= \l^\frac{-1}{p-1},
\end{equation}
the problem \eqref{1.1}, or \eqref{1.3}, can be equivalently written in the form
\begin{equation}
\label{iv.3}
\left\{ \begin{array}{ll}
-v'' =a(x)|v|^{p }\, \sgn(v) \, g(\e v' ) \,  h(\e v ),
\qquad  0<x<1, \\[1ex] v'(0)=v'(1)=0,
\end{array}\right.
\end{equation}
where $g$ is defined in   \eqref{1.4} and
\begin{equation}
\label{h}
 h(u) = \left\{ \begin{array}{ll} \frac{f(u)} {|u|^{p} \sgn(u)}  & \quad \hbox{if} \;\; u\neq0 ,\\[1ex]
 1   & \quad \hbox{if} \;\; u =0. \end{array}\right.
\end{equation}
 According to \eqref{1.2}, the problem \eqref{iv.3} perturbs, as $\e>0$ separates away from $0$, from the semilinear   problem
\begin{equation}
\label{iv.4}
\left\{
\begin{array}{ll} -v'' =a(x) \, |v|^{p} \, \sgn(v),
\qquad 0<x<1, \\[1ex] v'(0)=v'(1)=0. \end{array}\right.
\end{equation}
We claim that the problem \eqref{iv.4} admits a   subsolution $\alpha$ and a supersolution  $\beta$, with $\alpha(x) <  \beta(x) $ for all $x \in [0,1],$ such that every possible solution $v$
of  \eqref{iv.4}, with $v \ge \alpha $ in $[0,1]$,
satisfies $v(x) > \alpha (x) $ for all $x\in [0,1]$ and, similarly,
every possible solution $v$
of  \eqref{iv.4}, with $v \le \beta $ in $[0,1]$,
satisfies $v(x) < \beta(x) $ for all $x\in [0,1]$.
This means that $\alpha$ and  $\beta$ are strict sub- and  supersolutions according to, e.g., \cite[Ch. III]{DCH}.

\smallskip
\noindent
{\it Construction of a subsolution.}  Let $\mu_1$ be  the unique  positive eigenvalue, with an associated positive eigenfunction $\varphi_1$, of the weighted   problem
\begin{equation*}
   \left\{ \begin{array}{ll} -\v''= \mu \,a(x) \, \v, & \quad 0<x<z, \\[1ex]
  \v(0)=\v(z)=0.  & \end{array}\right.
\end{equation*}
Then, pick  $c>0$  so small  that
\begin{equation}
\label{iv.5}
   \mu_1 [c\varphi_1(x)]^{1-p} \leq 1 \qquad \hbox{for all}\;\; x\in [0,z]
\end{equation}
and define
\begin{equation}
\label{iv.6}
  \alpha (x)  =   \left\{ \begin{array}{ll}  c\,\varphi_1(x) & \quad \hbox{if }   0\le x \le z, \\[1ex]
 c \, \varphi_1'(z) \, (x-z) & \quad \hbox{if }z<x\le 1.
 \end{array}\right.
\end{equation}
It is clear that $\alpha \in W^{2, \infty}(0,1)$ and,  since  $p\in(0,1)$,
by \eqref{iv.5} and \eqref{iv.6}, it satisfies
\begin{align}
-  \alpha'' (x)  &= -c \, \varphi''_1(x)  =  a(x) \, \mu_1\,  c \, \varphi_1(x)  = a(x) \, \mu_1\,  [c \, \varphi_1(x)]^{1-p}[c\v_1(x)]^p
\\[1ex] &  \le a(x) [ c \, \varphi_1(x) ]^p
= a(x) \alpha^p(x) \quad   \hbox{for a.e. }\; x\in  (0,z)
\label{iv.7}
\end{align}
and
$$
-  \alpha'' (x) = 0 \le  a(x)\, |\alpha (x) |^p \, \sgn(\alpha (x))\quad
\hbox{for a.e. }\; x\in  (z,1).
$$
Further, we have that
$$
  \a'(0) = c \varphi_1'(0) > 0,\quad \a'(1) = c \varphi_1'(z) <0.
$$
\par
Now, we will show that any  solution $v$ of  \eqref{iv.4} such that $v \ge \a $ in $[0,1]$ also   satisfies $v(x) > \a(x) $ for all $x\in [0,1]$. Indeed, set $w = v-\a $ and suppose, by contradiction, that $\min w = 0$. Let  $x_0\in [0,1]$ be such that $w(x_0) = 0$. Since $w\geq 0$ in $[0,1]$ and
\[
  w'(0) = -\a'(0) < 0 < -\a'(1) = w'(1),
\]
it follows that $x_0\in (0,1)$. Hence, $w'(x_0)=0$.
\par
Suppose that $x_0\in (0,z]$. Then, as $v\geq \alpha \ge 0$ in $[0,z]$ and $p>0$,  we can infer from \eqref{iv.7} that
\begin{align*}
-  w'' (x) & =-v''(x)+\a''(x)
\ge a(x) \, (v^p(x) -  \alpha^p(x)) \ge 0 \quad   \hbox{for a.e. }\;  x \in (0,z).
\end{align*}
Thus, $w$ is  concave in $[0,z]$ and hence
$$
w(x) \le w(x_0) +  w'(x_0) (x-x_0) = 0 \quad \text{for all  } x \in [0,z].
$$
This implies that $w=0$ in  $[0,z]$, contradicting   $w'(0) <0$.  Therefore, $x_0\in (z,1)$.  Since $w(x_0)=0$, we have  that
\[
    v(x_0) = \a (x_0) =c\v_1'(z)(x_0-z)<0.
\]
Thus, there exists an  interval $J\subseteq (z,1) $, with $x_0\in J$, such that $v(x)<0$ if $x\in J$ and
\begin{equation}
\label{iv.8}
-  w'' (x) = - v''(x) = a(x) \, |v(x)|^p \, \sgn(v(x))  > 0 \quad
   \hbox{for a.e.  } x \in J.
\end{equation}
Hence, $w$ is concave in $J$. Arguing as above, we find that $w=0$ in $J$,
thus contradicting the  strict inequality in \eqref{iv.8}. Thus, we have proved that $
 v(x)> \a(x)$ for all $x\in [0,1]$.

\smallskip
\noindent
{\it Construction of a supersolution.}
For every  $k>0$,   let $z_k$ denote  the  unique solution of the linear problem
\begin{equation}
\label{iv.9}
\left\{ \begin{array}{ll}
-z'' =\Big(a(x) - \displaystyle\int_0^1 a(t) \, dt \Big) k^p,
\qquad  0<x<1, \\[2mm] z'(0)=z'(1)=0,\;\;
\displaystyle\int_0^1 z(t) \, dt =0.
\end{array}\right.
\end{equation}
The Poincar\' e--Wirtinger inequality yields
\begin{equation}
\label{iv.10}
  \|z_k\|_{L^\infty(0,1)} \le \|z'_k\|_{L^1(0,1)}   \le \|z'_k\|_{L^\infty(0,1)}  \le  \|z''_k\|_{L^1(0,1)} \le 2 \|a\|_{L^1(0,1)} \, k^p.
\end{equation}
Consequently, since $p\in (0,1)$,
the function $\b$ defined by $\beta  =  z_k + k$ satisfies, for  sufficiently large $k>0$,
$ \min \beta >\max \a >0$. Moreover,   for a.e.  $x\in [0,1]$, we have  that
\begin{align}
-\b''(x) & =-z_k''
  = a(x)k^p-k^p\int_0^1a(t)\,dt
  \\ & = a(x) \b^p(x) + a(x)[ k^p-\b^p(x)]-k^p\int_0^1a(t)\,dt
\end{align}
and hence
\begin{equation}
\label{iv.12}
-\b''(x) = a(x) \b^p(x) + k^p\left[ a(x) \left(1-\left(1+\frac{z_k(x)}{k}\right)^p  \right) -
 \int_0^1 a(t) \, dt\right].
\end{equation}
Using \eqref{iv.10} and  the assumption $p\in (0,1)$, it is easily seen that
\begin{equation*}
 \lim _{k\to\infty} \left[  a(x) \left( 1-\Big(  1+ \frac{z_k(x)}{k}\Big)^p \right)
    \right] = 0 \quad   \hbox{uniformly  a.e. in }   [0,1].
\end{equation*}
Thus, since  $\int_0^1 a(t) \, dt\ <0$, we can conclude from \eqref{iv.12} that,
for sufficiently large $k>0$,
\begin{equation}
\label{iv.13}
-\b''(x) \geq  a(x) \b^p(x)  - \frac{1}{2}k^p \int_0^1 a(t) \, dt   \quad   \hbox{for  a.e. }  x\in  [0,1],
\end{equation}
and hence $-\b''(x)  >  a(x) \b^p(x)$    for a.e.  $x\in  [0,1]$. Thus, the function $\b$ is a
supersolution of  \eqref{iv.4} satisfying the boundary conditions.
\par
Now, we will show that any  solution $v$ of  \eqref{iv.4} such that $v \le \b $ in $[0,1]$ satisfies $v(x) < \b(x) $ for all $x\in [0,1]$.   Indeed, consider  the function  $w = \b- v $ and suppose, by contradiction, that $\min w = 0$. Let $x_0\in [0,1]$ be such that $w(x_0) = 0$.
Then, there exists an  interval $J\subseteq [0,1] $, with $x_0\in J$, such that for a.e $x\in J$
\begin{equation}
\label{iv.14}
|a(x)(\b^p(x) - |v(x)|^p\sgn(v(x)) )| < - \frac{1}{2}k^p \int_0^1 a(t) \, dt
\end{equation}
and hence,  by \eqref{iv.13}, \eqref{iv.4} and \eqref{iv.14},
\begin{align}
  -w''(x) &=-\b''(x)+v''(x)
  \\
  &\geq  a(x) \b^p(x)  -a(x) \, |v(x)|^{p} \, \sgn(v(x)) - \frac{1}{2}k^p \int_0^1 a(t) \, dt    >0.
  \label{iv.15}
\end{align}
Thus,  $w$ is concave in $J$.   Arguing similarly, we find that $w=0$ in $J$. So, contradicting the  strict inequality in \eqref{iv.15}. Therefore, we have shown that $v(x)<  \b(x)$ for all $x\in [0,1]$.

\smallskip
\noindent
{\it Degree computation.}
Note that any solution $v$ of \eqref{iv.4}  such that $\a\leq v\leq \b$ also satisfies
$$
\|v'\|_\infty \le \|a\|_{L^1(0,1)} \max\{|\min \a|^p, ( \max \b )^p\}.
$$
Pick a constant
\[
  C > \|a\|_{L^1(0,1)} \max\{|\min \a|^p, ( \max \b )^p\}
\]
and consider the open bounded subset of $ \mc{C}^1[0,1]$ defined by
$$
\Omega = \{v \in  \mc{C}^1[0,1] : \a(x) < v(x) < \b(x)\;
\hbox{for all}\; x\in [0,1], \;\|v'\|_{L^\infty(0,1)} <C\}.
$$
Since $\alpha(x) < \beta(x) $ for all $x\in [0,1]$, $\Omega$ is non-empty.
Let $\mc{T} : [0,\infty) \times  \mc{C}^1[0,1] \to \mc{C}^1[0,1]$ denote the operator which sends any $(\e, v) \in  [0,\infty) \times  \mc{C}^1[0,1] $
to  the unique solution $w\in W^{2,\infty}(0,1)$ of the linear problem
\begin{equation*}
\left\{ \begin{array}{ll}
-w'' + w = a(x) \, |v|^p \, \sgn(v) \, g(\e v') \, h(\e v) + v,
\qquad  0<x<1,
\\[1mm]
w'(0)=w'(1)=0.
\end{array}\right.
\end{equation*}
It is plain that $\mc{T}$ is completely continuous and its fixed points are precisely the solutions of the problem \eqref{iv.3}. As $\a$ and $\b$ are, respectively, a strict subsolution   and a strict supersolution of  \eqref{iv.4}, by our choice of the constant $C$, it follows that $T(0, \cdot)$
has no fixed points on $\partial \Omega$.
A standard argument (see \cite[Ch. III]{DCH})  also shows  that
$$
{\rm deg}_{LS} (\mc{I-T}(0, \cdot),\Omega,0 ) =1.
$$

\smallskip
\noindent
{\it Existence of  continua.}
The boundedness of  $\partial \Omega$  and the complete continuity of the operator $\mc{T}$ guarantee the existence of some $\e^*>0$
such that  $\mc{T}(\e, \cdot)$ has no fixed point on $\partial \Omega$ for all $\e\in [0,\e^*]$.
Consequently, the Leray-Schauder continuation theorem \cite[p. 63]{LS}
 yields the existence of a
continuum of solutions $(\e, v)$ of the problem   \eqref{iv.3}, where $\e\in [0,\e^*]$  and $v \in \Omega$,
and hence of solutions $(\l, u)$ of the problem   \eqref{1.1}, where $\l=\e^{1-p}\in [0,\l^*]$, with
 $\l^* = (\e^*)^{1-p}$, and $u = \l^\frac{1}{1-p}v$.
 \par
Let us verify that, for each $\e\in [0,\e^*]$,    $v$ is positive and, therefore, for every $\l \in (0,\l^*]$, $u$ is positive. Indeed, otherwise, owing to the definition of $\a$, there should exist $x_0 \in (z,1]$ such that  $v(x_0) = \min v <0$ and $v'(x_0)=0$. Then, one would infer  from \eqref{iv.3} the existence of an interval $J\subseteq (z,1]$, with $x_0\in J$,  such that  $v''(x) <0$   for a.e.   $x\in J $. This is clearly  impossible  at a minimum point which also a critical point.
\par
As in \cite{Rabgb, Ra71}, by the Zorn lemma, this continuum  of positive solutions can   be eventually continued    to a component $\mf{C}_{r,0}^+$ of  the  set of   positive regular solutions of \eqref{1.1}.  Finally, by Lemma \ref{leii.2}, \eqref{iv.1}  holds for sufficiently small $\l > 0$.
\end{proof}

\section{Bifurcation from $(\l,u)= (\infty,0)$ when $p>1$}
\label{s5}

\noindent
Throughout this section, we   assume that the functions $a(x)$ and $f(u)$  satisfy $(a_2)$ and $(f_1)$ with $p>1$, respectively.  In this case, by Theorem \ref{thii.1}, the problem \eqref{1.1} cannot have any solution for sufficiently small $\l>0$.
The main goal of this section is establishing the existence of a component of  positive  regular  solutions of \eqref{1.1} bifurcating from $0$ as $\l \to \infty$.
From \cite[Thm. 1.5]{LOR1} and \cite[Thm. 10.1]{LOR2}    the following result  can be deduced.

\begin{theorem}
\label{thv.1}
Assume $(f_1)$, with $p\in (1,\infty)$, and $(a_2)$. Then, the problem \eqref{1.1} has at least two     positive solution $u_\l, v_\l$  for sufficiently large $\l > 0$. Moreover, $u_\l$ is regular and can be chosen so that $\lim_{\l\to\infty}\|u_\l\|_{C^1[0,1]}=0$.
\end{theorem}

The next  result  complements Theorem \ref{thv.1} by establishing  the existence of a component  of the set $\mc{S}_r^+$    containing  small   solutions for sufficiently large $\lambda >0$.
\begin{theorem}
\label{thv.2}
Assume $(f_1)$, with $p\in (1,\infty)$, and $(a_2)$.
Then,  there is a component $\mf{C}_{r,\infty}^+$  of $\mc{S}_r^+$  such that $(\l^*,\infty) \subseteq \mc{P}_\l(\mf{C}_{r,\infty}^+)$, for some $\l^*>0$,
and
\begin{equation}
\label{v.1}
 \lim_{\l\to \infty} \min\{\|u_\l\|_{\mc{C}^1[0,1]} : (\l, u_\l) \in \mf{C}_{r,\infty}^+\}=0.
\end{equation}
\end{theorem}

Theorems \ref{thv.1} and \ref{thv.2}  confirm that the global bifurcation diagram of \eqref{1.1}  looks like   show  the right (red) plots of Figures \ref{Fig02} and \ref{Fig03}, according to the regularity properties of the function $a(x)$ at $z$.   Our proof of Theorem \ref{thv.2} here is based on some elementary topological techniques based on the theory of superlinear indefinite problems of \cite{AL98}.

\begin{proof}
Like in the proof of Theorem \ref{thiv.2} we   suppose  that $f\in\mc{C}(\R)$ is an odd function  and we make the change of variable \eqref{iv.2}.   Then, the problem \eqref{1.1}, or \eqref{1.3}, can be equivalently written as \eqref{iv.3}, where $g$ and $h$ are defined by  \eqref{1.4} and \eqref{h}, respectively. By \eqref{1.2}, this problem perturbs, as $\e\to 0$, from the semilinear   boundary value problem \eqref{iv.4}, which can be obtained from
\begin{equation}
\label{v.3}
\left\{ \begin{array}{ll} -v'' =\mu v+a(x) |v|^p \sgn(v),\qquad  0<x<1,
\\[1ex] v'(0)=v'(1)=0. \end{array}\right.
\end{equation}
by freezing  the value of the parameter $\mu$ at $\mu=0$; \eqref{v.3} is a simple one-dimensional  prototype of the multidimensional model of \cite{AL98}.
\par
Since $\mu=0$ is a simple algebraic eigenvalue of $-D^2$ under Neumann boundary conditions with
associated eigenfunction $1$, the local index of zero changes as  $\mu$ crosses zero (see, e.g.,
\cite[Thm. 5.6.2]{LG01}). Thus, thanks to \cite[Thm. 7.1.3]{LG01}, there is a component $\mf{C}_{\mu,0}^+$ of the set of positive solutions of \eqref{v.1} in $\R\times \mc{C}^1[0,1]$   such that $(\mu,v)=(0,0)\in \mf{\overline C}_{\mu,0}^+$. Since $p$ might vary in the interval $(1,2)$, we do not have the required regularity to apply the local bifurcation theorem in \cite{CR}. Let $(\mu_n,v_n)$, $n\geq 1$, be a sequence of solutions of $\mf{C}_{\mu,0}^+$, with $v_n\neq 0$, such that
\begin{equation}
\label{v.4}
  \lim_{n\to\infty}(\mu_n,v_n)=(0,0)\quad \hbox{in}\;\; \R \times \mc{C}^1[0,1].
\end{equation}
Then, as it will become apparent below, we have that
\begin{equation}
\label{v.5}
  \lim_{n\to\infty}\frac{\mu_n}{\|v_n\|_{\infty}^{p-1}}=-\int_0^1a(x)\,dx >0
\end{equation}
and hence  $\mu_n>0$ for sufficiently large $n$. In particular, $\mf{C}_{\mu,0}^+$ bifurcates  {supercritically} from $(\mu,u)=(0,0)$. To prove \eqref{v.5} one can argue as follows. Since
\begin{equation}
\label{v.6}
  -v_n''=\mu_n v_n +a(x)v_n^p,\qquad n\geq 1,
\end{equation}
we have that
\[
  \frac{v_n}{\|v_n\|_{\infty}} = (-D^2+1)^{-1}\left[ \frac{v_n}{\|v_n\|_{\infty}}+
  \mu_n \frac{v_n}{\|v_n\|_{\infty}}+a(x) \frac{v_n}{\|v_n\|_{\infty}}v_n^{p-1}\right],
\]
where $(-D^2+1)^{-1}$ stands for the resolvent operator of $-D^2+1$  under homogeneous Neumann boundary conditions. As $(-D^2+1)^{-1}$ is compact,  there exists a subsequence of $\v_n := \frac{v_n}{\|v_n\|_{\infty}}$, $n\geq 1$, relabeled by $n$, such that $\lim_{n\to \infty}\v_n=\v \in \mc{C}^2[0,1]$ in $\mc{C}^1[0,1]$. By \eqref{v.4}, letting $n\to \infty$ it is easily seen that necessarily $\v=1$ and, since this argument can be repeated along any subsequence,
it becomes apparent that
\begin{equation}
\label{v.7}
  \lim_{n\to \infty}\v_n = 1 \quad \hbox{in}\;\; \mc{C}^1[0,1].
\end{equation}
On the other hand, integrating \eqref{v.6} in $[0,1]$ and dividing by $\|v_n\|_{\infty}^p$ yields
\[
  \frac{\mu_n}{\|v_n\|_{\infty}^{p-1}}\int_0^1 \frac{v_n(x)}{\|v_n\|_{\infty}} \,dx =-\int_0^1a(x)   \left( \frac{v_n(x)}{\|v_n\|_{\infty}} \right)^p \,dx.
\]
Consequently, letting $n\to \infty$ in this identity, \eqref{v.5} follows readily from \eqref{v.7}.
This shows that $\mf{C}_{\mu,0}^+$ bifurcates towards the right at $\mu=0$.
In other words, there is neighborhood $U$ of $(0,0)$ such that $\mu>0$ if $(\mu,v)\in \mf{C}_{\mu,0}^+\cap U$.
\par
Suppose that \eqref{iv.4} admits a regular positive solution, $(\mu,v)$. Then, since $p>1$, $v$ is strictly positive and hence $(-D^2-\mu)v= a(x)v^p(x)>0$ for all $x\in [0,z]$.
Moreover, $v(0)>0$ and $v(z)>0$.
Thus, $v$ provides us with a positive   strict supersolution of $-D^2-\mu$ in $(0,z)$ under homogeneous Dirichlet boundary conditions and, due to \cite[Thm. 7.10]{LG13},
$$
  \s[-D^2-\mu;\mc{D},(0,z)]=\left(\frac{\pi}{z}\right)^2-\mu >0.
$$
So, $\mu < \mu_*:=\left(\frac{\pi}{z}\right)^2$. Therefore, \eqref{v.1} cannot admit any regular solution if $\mu \geq \mu_*$. In particular, $\mc{P}_\mu(\mf{C}_{\mu,0}^+)\subset (-\infty,\mu_*)$.
Moreover, by the generalized a priori bounds of \cite[Sect. 4]{AL98}, as we are working with a one-dimensional problem, for every compact interval $K\subset \R$, there is a constant $C=C(K)>0$ such that
$\|v\|_{\mc{C}^1[0,1]}\leq C$ for any positive solution $(\mu,v)$ of \eqref{v.1}  with $\mu\in K$. Therefore, setting $\mu_c := \max_{(\mu,v)\in\mf{C}_{\mu,0}^+}\mu$, we have that $\mu_c\in (0,\mu_*]$ and that $ \mc{P}_\mu(\mf{C}_{\mu,0}^+)=(-\infty,\mu_c]$, as illustrated in Figure \ref{Fig04}, where  we are plotting $\mu$, in abscisas, versus $\|v\|_{\mc{C}^1[0,1]}$  in ordinates. Thus,   each  solution $(\mu,v)$ of \eqref{v.1} is represented by a single point on some of the components plotted in the figure.   Naturally, \eqref{v.1} might have other components of positive solutions, like  $\mf{D}^+$.
\par
\begin{figure}[!ht]
\centering{
  \includegraphics[scale=0.6]{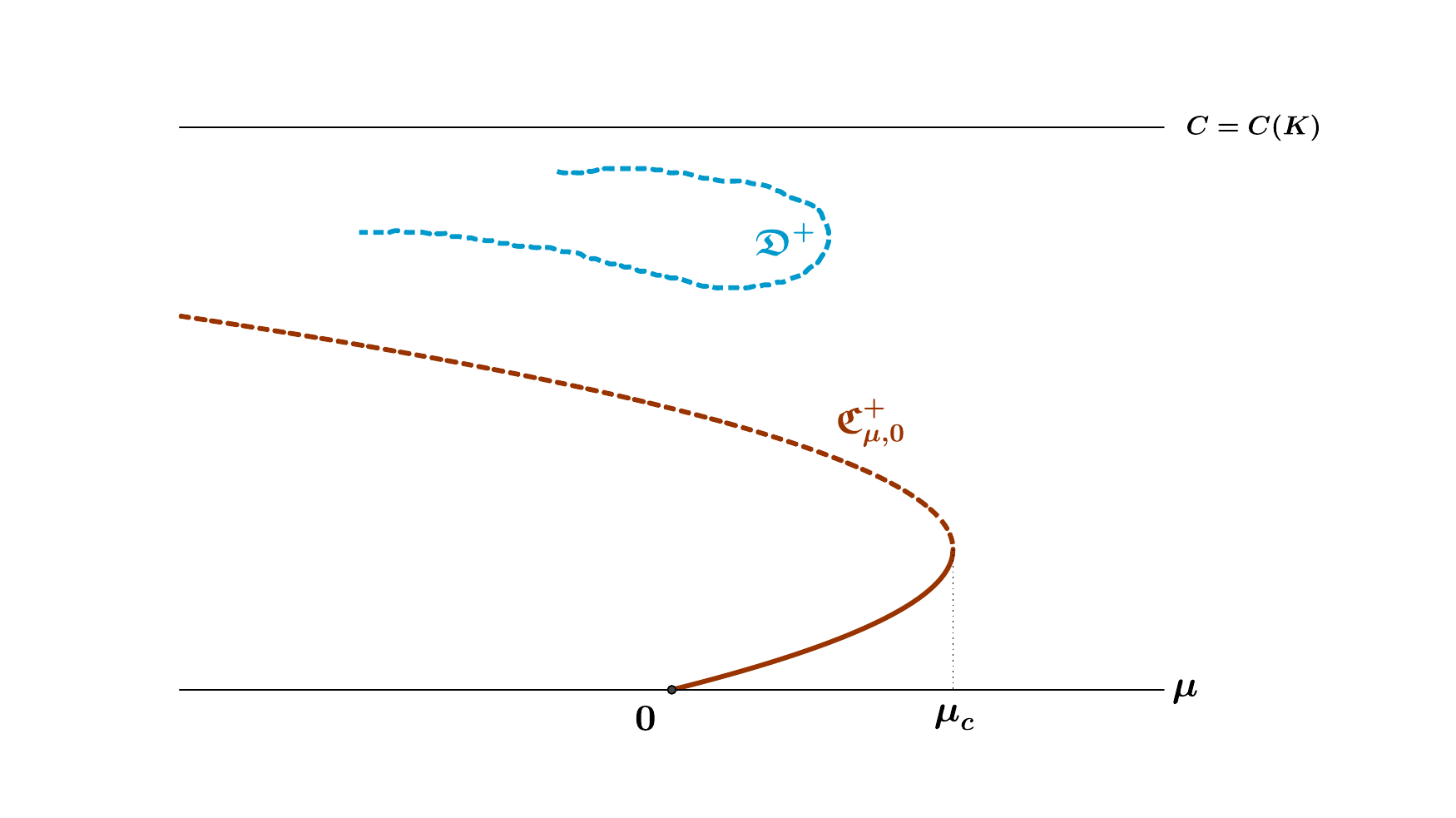}  }
{\caption{The components $\mf{C}_{\mu,0}^+$ and $\mf{D}^+$ of
$\mc{S}_r^+$; $\mf{C}_{\mu,0}^+$ bifurcates supercritically from $(\mu,0)$ at $\mu=0$ and goes backwards at some $\mu_c>0$.}
\label{Fig04}}
\end{figure}

\par
 As the main technical device to get the uniqueness of the stable solution in \cite{GL00, GL01} is the Picone identity  \cite{Picone}, which due to \cite[Lem. 9.3]{FRLG} remains true for Neumann boundary conditions, the theory of \cite{GL00, GL01} can be adapted \emph{mutatis mutandis} to our present setting to show that the unique stable positive solutions of \eqref{v.1} are the minimal solutions of \eqref{v.1} for $\mu>0$ (see \cite{FRLGB}),  i.e., those on the piece of $\mf{C}_{\mu,0}^+$ plotted with a continuous line in Figure \ref{Fig04}. The remaining solutions,   plotted with a dashed line, are linearly unstable.
 \par
Subsequently,  the positive regular solutions of \eqref{v.3} as regarded as positive fixed points of the compact operator
$\mc{K}:\R\times \mc{C}[0,1]\to \mc{C}[0,1]$ defined by
\begin{equation}
\label{v.8}
\mc{K}(\mu,v)=(-D^2+1)^{-1}[(\mu+1)v +a(x)|v|^p \sgn(v)].
\end{equation}
Let $B$ denote any bounded open subset of $\mc{C}^1[0,1]$ containing all non-negative fixed points $(\mu,v)$ of \eqref{v.8},  with $\mu \in [-1,\mu_*]$. It exists by the uniform a priori bounds on compact subintervals of $\mu$ and the non-existence  for $\mu \geq  \mu_*$. Since $B$ contains all non-negative fixed points of
$\mc{K}(\mu,\cdot)$ for all $\mu \in [-1,\mu_*]$, the fixed point index of $\mc{K}(\mu,\cdot)$ on $B$
with respect to the cone $P$ of nonnegative functions in $\mc{C}^1[0,1]$   is well defined.  Moreover,
\begin{equation}
\label{v.9}
  i_P(\mc{K}(\mu,\cdot), B)=0 \qquad \hbox{for all}\;\; \mu \in [0,\mu_*].
\end{equation}
Indeed, by the invariance by homotopy of the index, for every $\mu \in [0,\mu_*]$, we have that
\[
   i_P(\mc{K}(\mu,\cdot), B)= i_P(\mc{K}(\mu_*,\cdot), B)=
   i_P(\mc{K}(\mu_*,\cdot),0),
\]
because $0$ is the unique fixed point of $\mc{K}(\mu_*,\cdot)$ in $P$. For ascertaining the spectral radius of the linearized operator $D\mc{K}(\mu_*,0)$, suppose $\r\in\R$ is an eigenvalue of $D\mc{K}(\mu_*,0)$ associated with a positive eigenfunction $\v$. Then,
\[
  (-D^2+1)^{-1}[(\mu_*+1)\v]=\r \v,
\]
which can be equivalently expressed as $-\v''=( \tfrac{\mu_*+1}{\r}-1)\v$.
Thus, since $\v'(0)=\v'(1)=0$, integrating in $[0,1]$ yields $0 = ( \tfrac{\mu_*+1}{\r}-1)\int_0^1 \v$ and hence, $\r = \mu_*+1>1$, because $\v(x)>0$ in $(0,1)$. Consequently, by  \cite[Lem. 13.1]{Amann}, $i_P(\mc{K}(\mu_*,\cdot),0)=0$ and therefore, \eqref{v.9} holds.
\par
Now, let denote by $v_\mu$ the minimal positive solution of \eqref{v.3} for $\mu \in (0,\mu_c)$. Since it is linearly asymptotically stable for all $\mu \in (0,\mu_c)$ and neutrally stable for $\mu=\mu_c$, combining the Schauder formula with the analysis of \cite[Sect. 7]{AL98} it is easily seen that
$i_P(\mc{K}(\mu,\cdot),v_\mu)=1$. Therefore, by \eqref{v.9} and the excision property, taking into account that $v_\mu$ is non-degenerate, we    find that there exists $\eta>0$ such that
\[
  i_P(\mc{K}(\mu,\cdot),B\setminus {\overline B}_\eta(v_\mu))=-1\qquad \hbox{for all}\;\; \mu \in [0,\mu_c/2],
\]
where $B_\eta(v_\mu)$ stands for the ball of radius $\eta$ centered at $v_\mu$ in $\mc{C}^1[0,1]$. Consequently, since
$\displaystyle \lim_{\mu \to 0}v_\mu=0$, it is plain that, for sufficiently small $\eta>0$,
\begin{equation}
\label{v.10}
  i_P(\mc{K}(0,\cdot),\O )=-1, \quad \hbox{ where }\O= B\setminus {\overline B}_\eta(0).
\end{equation}
Next, note that the positive solutions of \eqref{iv.3} are the positive fixed points of
the compact operator $\mc{M}:\R\times \mc{C}^1[0,1]\to \mc{C}^1[0,1]$ defined by
\begin{equation*}
\mc{M}(\e,v)=(-D^2+1)^{-1}\left[ a(x)|v|^{p }\, \sgn(v) \, g(\e v' ) \,  h(\e v )
\right].
\end{equation*}
Since $\mc{K}(0,\cdot)=\mc{M}(0,\cdot)$, it follows from  \eqref{v.10} that $i_P(\mc{M}(0,\cdot),\O )=-1$ for sufficiently small $\eta>0$. Moreover, for sufficiently small $\e>0$, \eqref{iv.3} cannot admit a solution on $\p \O$. On the contrary, suppose
that there exists a sequence $\{(\e_n,v_n)\}_{n\geq 1}$  of solutions of \eqref{iv.3} such that $\lim_{n\to\infty}\e_n=0$ and $v_n\in \p\O$ for all $n\geq 1$. Then, $v_n=\mc{M}(\e_n,v_n)$
for all $n\geq 1$ and, by compactness, there exists a subsequence of $v_n$, relabeled by $n$, such that
$\lim_{n\to \infty}v_n=v_0\in\p\O$. Since $(\e,v)=(0,v_0)$ must be a positive solution of \eqref{v.3}, this contradicts the fact that \eqref{v.3} cannot admits positive solutions on $\p\O$. Therefore, by the homotopy invariance of the fixed point index, $i_P(\mc{M}(\e,\cdot),\O )=-1$ for sufficiently small $\e>0$. Finally, the Leray-Schauder continuation theorem  \cite[p. 63]{LS} ends the proof.
\end{proof}

\begin{rem}
In this section we confined ourselves  to considering the case where  the function $a(x) $ satisfies condition $(a_2)$. However,  similar conclusions could be established even when the function $a(x)$ changes sign   finitely many times.
Actually, the existence of multiple continua of solutions could be shown in this case.
Indeed, as observed in \cite{GL00} and then rigorously proven in \cite{FZ}, the problem
\begin{equation}
\left\{ \begin{array}{ll}
-v'' =a(x)v^{p },
\qquad  0<x<1, \\[1ex] v'(0)=v'(1)=0
\end{array}\right.
\end{equation}
 might possess a high number of positive solutions according to the number of changes of sign of the weight function $a(x)$.    The conjecture of \cite{GL00} has been recently proven
 in \cite{FLG} for symmetric weight functions $a(x)$.
\end{rem}

\section{Pointwise behavior of the regular solutions as $\l\to  \infty$}
\label{s6}

\noindent
In this section we ascertain the limiting profile of the regular  positive solutions
of \eqref{1.1} as $\l\to  \infty$, should they exist, when $f(u)$ satisfies $(f_2)$ with $p\in (0,1]$ and $a(x) $ satisfies $(a_2)$. This analysis also provides us with the pointwise behavior of the regular
positive solutions separated away from zero when $p>1$. So,   through this section we suppose that \eqref{1.1} possesses a sequence of regular positive solutions, $\{(\l_n,u_n)\}_{n\ge1}$, such that
\begin{equation}
\label{6.1}
  \lim_{n\to\infty}\l_n= \infty.
\end{equation}
We recall that the assumption $(a_2)$ on the function $a(x)$ entails  that any regular positive solution $u$ of \eqref{1.1}  is decreasing in $[0,1]$, as a result of its concavity in $[0,z)$ and its convexity in $(z,1]$, and,  in particular,
\begin{equation}
\label{6.2}
  \|u\|_{L^\infty(0,1)}=u(0) \quad \hbox{and}\quad \|u'\|_{L^\infty(0,1)} = -u'(z).
\end{equation}
We  stress that $u$ might not be strictly positive if $p\in (0,1)$, as already pointed out
 in Section \ref{s1}.  However, should $u$ vanish, this would  happen  in the interval $(z, 1]$, i.e., necessarily $u(x)>0$ for each $x\in [0, z]$.
\par
The next result characterizes the pointwise limit of $\{u_n(x)\}_{n\ge1}$, as $n\to \infty$, for every $x\in [0,1]$.

\begin{lemma}
\label{levi.1} Assume
$(f_1)$ and $(a_2)$.
Let $\{(\l_n,u_n)\}_{n\ge1}$ be a sequence of regular positive solutions of  \eqref{1.1}  such that \eqref{6.1} holds. Then,   for a.e. $x\in  [0,1]$ there exists
\begin{equation}
\label{6.3}
  \lim_{n\to \infty}u_n(x)\in \{0,\infty\} .
\end{equation}
\end{lemma}
\begin{proof}
Integrating the  equation of \eqref{1.1} on $[0,z]$ yields
$$
  \int_0^z a(x)f(u_n(x))\,dx =\frac{1}{\l_n} \frac{-u_n'(z)}{\sqrt{1+(u'_n(z))^2}}<\frac{1}{\l_n}
  \quad \hbox{for all}\;\; n,
$$
and thus, letting $n\to \infty$, we find that $\lim_{n\to\infty}  \int_0^z a(x)f(u_n(x))\,dx =0$.
The convergence in $L^1(0,z)$ entails that there is  a subsequence $\{u_{n_h}\}_{h\ge1}$ of
$\{u_{n}\}_{n\ge1}$ such that
$$
  \lim_{h\to \infty} a(x) f(u_{n_h}(x))=0 \quad \hbox{a.e. in }\;\; [0,z].
$$
Consequently, as $a(x)>0$ a.e. in  $[0,z]$, we find that $\lim_{h\to \infty}   f(u_{n_h}(x))=0$ a.e. in $[0,z]$. Similarly, as $\int_0^1  a(x)f(u_n(x))\,dx=0$ for all $n$, we also have that
\[
  \lim_{n\to\infty}  \int_z^1 a(x)f(u_{n}(x))\,dx =0\]
and therefore there is  a subsequence $\{u_{n_k}\}_{k\ge1}$ of
$\{u_{n}\}_{n\ge1}$ such that $\lim_{k\to \infty} f(u_{n_k}(x))=0$ a.e. in $[1,z]$.
Since  this argument can be repeated  for any possible subsequence of $\{u_{n}\}_{n\ge1}$, we infer that
$\lim_{k\to \infty} f(u_{n}(x))=0$  a.e. in $[0,1]$. As  $f(u)>0$ for all $u>0$, \eqref{6.3} holds.
\end{proof}

Note that Lemma \ref{levi.1} holds regardless the nature of the growth of $f(u)$ at $u=0$, i.e., without any restriction on the size of $p>0$.

\par

\begin{lemma}
\label{le6.2}
Assume
$(f_1)$ and $(a_2)$.
Let $\{(\l_n,u_n)\}_{n\ge1}$ be a sequence of regular positive solutions of  \eqref{1.1}  such that \eqref{6.1} holds. Then,    the cluster points of the sequence $\{ u_n(0) \}_{n\ge1}$ are either  $0$ or  $\infty$.  Moreover, using \eqref{6.2}, as soon as  $p\in (0,1]$, one has that
\begin{equation}
\label{6.4}
  \lim_{n\to\infty}\|u_n\|_{L^\infty(0,1)}=\lim_{n\to \infty} u_n(0) = \infty.
\end{equation}
\end{lemma}
\begin{proof}
Suppose that there are a constant $K>0$ and  a subsequence, relabeled by $n$, of $\{(\l_n,u_n)\}_{n\ge1}$ such that
\begin{equation}
\label{6.5}
  u_n(0)=\|u_n\|_{L^\infty(0,1)} \leq K  \qquad \hbox{for all}\;\; n.
\end{equation}
We will show that this is impossible if $p\in (0,1]$, while it implies
\begin{equation}
\label{6.6}
   \lim_{n\to \infty} u_n(0) =0
\end{equation}
 if $p>1$.
\par
We first consider the case where $p\in (0,1]$. Let us define the auxiliary function
$$
  \tilde f(u)= \left\{ \begin{array}{ll} f(u) &\quad \hbox{if}\;\; u\le K,\\
  \frac{f(K)}{K}u &\quad \hbox{if}\;\; u >K.
  \end{array}\right.
$$
By construction, $\tilde f \in\mc{C}(\R)$ and, since $f(u)$ satisfies  $(f_1)$ with $p\in(0,1]$,
there exists $c>0$ such that $\tilde f(u) \ge  c \, u$ for all $u\ge 0$. Then, for every $n$, $(\l_n,u_n)$ solves the auxiliary boundary value problem
\begin{equation*}
   \left\{ \begin{array}{ll}
   \displaystyle -\left( \frac{u'}{\sqrt{1+(u')^2}}\right)' =\l a(x) \tilde f(u), & \quad 0<x<1, \\[1ex]    u'(0)=u'(1)=0. & \end{array}\right.
\end{equation*}
Thus, each $u_n$ satisfies
\begin{equation*}
-u_n''(x)  =  \l_n a(x) \tilde f(u_n(x)) \big[1 + (u'_n(x))^2\big]^\frac{3}{2} \ge
 \l_n a(x)  \, c \, u_n(x)  \quad \hbox{a.e. in } (0,z)
\end{equation*}
and hence,  $u_n$ is a strictly positive   supersolution  of the  problem
\begin{equation}
\label{6.7}
   \left\{ \begin{array}{ll} -w''=  c \, \l_n a(x)   \, w, & \quad 0<x<z, \\
  w(0)=w(z)=0. &
  \end{array}\right.
\end{equation}
Let $\mu_1 $ denote the unique  positive eigenvalue, with a  corresponding positive eigenfunction $\varphi_1$, of the weighted eigenvalue problem
\begin{equation*}
   \left\{ \begin{array}{ll} -\v''= \mu \,a(x) \, \v, & \quad 0<x<z, \\
  \v(0)=\v(z)=0.  & \end{array}\right.
\end{equation*}
If we pick a sufficiently large $n$ so that $ c \,  \l_n  > \mu_1$, then a suitable multiple of $\varphi_1$  provides  us with a positive  subsolution of  \eqref{6.7} smaller than $u_n$, thus yielding the existence of a positive solution  of \eqref{6.7}. This solution  would be a principal
eigenfunction associated to $c\l_n>\mu_1$,   contradicting the uniqueness   of $\mu_1$. So, \eqref{6.4} holds in case $p\in (0,1]$.
\par
Now, consider the case where  $p>1$. Then, setting
$$
  \hat f(u)= \left\{ \begin{array}{ll} f(u) &\quad \hbox{if}\;\; u\le K,\\
  f(K) &\quad \hbox{if}\;\; u >K,
  \end{array}\right.
$$
\eqref{6.5} entails  that $f(u_n)=\hat f(u_n)$, for all $n$, and thus $(\l_n,u_n)$ solves
\begin{equation*}
   \left\{ \begin{array}{ll}
   \displaystyle -\left( \frac{u'}{\sqrt{1+(u')^2}}\right)' =\l a(x) \hat f(u), & \quad 0<x<1, \\[1ex]    u'(0)=u'(1)=0. & \end{array}\right.
\end{equation*}
Integrating the equation above in $(0,z)$ yields
$$
 \int_0^z a(x) \hat f(u_n(x))\,dx
 = \frac{1}{ \l_n} \frac{-u'_n(z)}{\sqrt{1+(u'_n(z))^2}}
 < \frac{1}{ \l_n}
$$
and consequently, by \eqref{6.1},
$$
  \lim_{n\to\infty}\int_0^z  a(x) \hat f(u_n(x))\,dx =0.
$$
This implies that,  for a subsequence, still labeled by $n$, $\lim_{n\to\infty} \big( a(x) \hat f(u_n(x))\big) =0$ a.e. in $[0,z]$ and hence $\lim_{n\to\infty}   \hat f(u_n(x))  =0$ a.e. in $[0,z]$.
The definition of  $\hat f(u)$ yields $\lim_{n\to\infty}    u_n(x)   =0$ a.e. in $[0,z]$.
Therefore, since  $u_n$ is decreasing, it becomes apparent that
\begin{equation}
\label{6.8}
  \lim_{n\to\infty}u_n(x)=0 \quad \hbox{for all}\;\; x\in (0,1].
\end{equation}
Suppose, by contradiction, that \eqref{6.6} does not hold.  Then, due to  \eqref{6.5}, there exists a subsequence, again labeled by $n$, such that
\begin{equation}
\label{6.9}
  \lim_{n\to \infty}u_n(0)=L\in (0,K].
\end{equation}
 By the concavity of $u_n$ in $[0,z)$, for any given $y\in (0,z)$, we have that
$$
u'_n(y)\leq \frac{u_n(y)-u_n(0)}{y}.
$$
Hence,  by \eqref{6.8} and \eqref{6.9}, we find that, for a further subsequence, still labeled by $n$,
\begin{equation}
\label{6.10}
   \lim_{n\to\infty} u_n'(y)\leq \lim_{n\to\infty} \frac{u_n(y)-u_n(0)}{y} = -\frac{L}{y}.
\end{equation}
By the concavity, we also have that
\begin{equation}
\label{6.11}
  u_n(x)\leq u_n(y)+u_n'(y)(x-y) \quad \hbox{for all}\;\; x\in (y,z).
\end{equation}
Note that the right hand side of \eqref{6.11} vanishes at $x_n =  y-\frac{u_n(y)}{u_n'(y)}$. Thanks to \eqref{6.8} and \eqref{6.10}, we get $\lim_{n\to \infty} x_n= y < z$ and hence, for sufficiently large $n$, $x_n<z$. This forces $u_n$ to vanish in $(0,z)$. As this  is impossible, we conclude that $L=0$, i.e., \eqref{6.6} holds.
\end{proof}

 Under  assumption $(f_1)$ there exists $M>0$,
not necessarily unique, such that $f(M)= \|f\|_{\infty}$. Throughout the rest of this section, $M$ is chosen so that $f(u)<f(M)$ for all $u>M$. Under assumption $(f_2)$, $M$ is uniquely determined.
\par

\begin{lemma}
\label{levi.3}
Assume
$(f_2)$ and $(a_2)$. Let $\{(\l_n,u_n)\}_{n\ge1}$ be a sequence of regular positive solutions of  \eqref{1.1}  such that \eqref{6.1} holds.
Then,
\begin{equation}
\label{6.12}
  u_n(1) < M  \quad \hbox{for sufficiently large}\;\; n.
\end{equation}
\end{lemma}
\begin{proof}
On the contrary, assume that there is a subsequence,
labeled again by $n$, such that $u_n(1)\geq M$
for all $n $. Then, since each $u_n$ is   decreasing, we have that  $u_n(x)\geq M$ for all $x\in [0,1]$. Accordingly, since $f \in\mc{C}^1[M,\infty)$, the next identity holds for every $x\in [0,1]$
\begin{equation}
\label{6.13}
  \l_n a(x)= \left( \frac{1}{f(u_n(x))} \frac{-u_n'(x)}{\sqrt{1+(u_n'(x))^2}} \right)' +
  \left( \frac{1}{f(u_n(x))}  \right)' \frac{u_n'(x)}{\sqrt{1+(u_n'(x))^2}}.
\end{equation}
Integrating \eqref{6.13} in $[0,1]$ yields
\begin{align*}
\l_n \int_0^1 a(x) \, dx & = \int_0^1 \left( \frac{1}{f(u_n(x))}  \right)' \frac{u_n'(x)}{\sqrt{1+(u_n'(x))^2}} \, dx
\\
& = - \int_0^1 \frac{f'(u_n(x))}{f^2(u_n(x))} \frac{(u_n'(x))^2}{\sqrt{1+(u_n'(x))^2}}\,dx.
\end{align*}
As  $u_n(x)\geq M$ and $f(u)$ is decreasing in $[M,\infty)$, we have that
$f'(u_n(x))\leq 0$ for all $x\in [0,1]$ and $n$. Hence, we obtain a contradiction with $(a_2)$, which  requires $\int_0^1a(x) \, dx <0$. Therefore, \eqref{6.12} holds.
\end{proof}

Throughout the remainder of this section
we will assume
that, in the  case where $p>1$,
the sequence $\{(\l_n,u_n)\}_{n\geq 1}$, in addition to
\eqref{6.1}, satisfies
\begin{equation}
\label{6.14}
  \liminf _{n\to \infty} u_n(0) >0.
\end{equation}
Under this assumption, according to Lemma \ref{le6.2}, the condition \eqref{6.4} must hold, regardless
the size of $p>0$. Naturally, this property fails to be true in case $p>1$ for the small regular positive solutions of \eqref{1.1} whose existence is guaranteed  by Theorem \ref{thv.1}.

\begin{lemma}
\label{le4.4} Assume
$(f_2)$ and  $(a_2)$.
Let $\{(\l_n,u_n)\}_{n\ge1}$ be a sequence of regular positive solutions of  \eqref{1.1}  satisfying \eqref{6.1} and \eqref{6.14}.
Let $x_\o$ be the unique point in the interval $(z,1)$
where $\int_0^{x_\o} a(x) \,dx=0$. Then, for sufficiently large $n$, there exists a unique $x_n\in (0,x_\o)$ such that
\begin{equation}
\label{6.15}
  u_n(x_n)=M.
\end{equation}
\end{lemma}
\begin{proof}
Under the condition \eqref{6.14}, Lemma \ref{le6.2} and Lemma \ref{levi.3} imply that $u_n(0)>M$ and $u_n(1)<M$ for sufficiently large $n$. Thus, since $u_n$ is decreasing in $[0,1]$,
there exists a unique $x_n\in (0,1)$ for which \eqref{6.15} holds. Necessarily, we have $u_n(x)\ge M$ for all $x\in [0,x_n]$. Thus, integrating \eqref{6.13} in $[0,x_n]$ we find that
\[
  \l_n\int_0^{x_n}a(x) \, dx = \frac{1}{f(M)}\frac{-u_n'(x_n)}{\sqrt{1+(u_n'(x_n))^2}}-
  \int_0^{x_n} \frac{f'(u_n(x))}{f^2(u_n(x))} \frac{(u_n'(x))^2}{\sqrt{1+(u_n'(x))^2}}\,dx >0,
\]
because $f'(u)<0$ if $u>M$. Hence, we conclude that $\int_0^{x_n}a(x)\,dx >0$ for sufficiently large $n$,
and therefore $x_n<x_\o$. This ends the proof.
\end{proof}

\begin{lemma}
\label{le6.5}
 Assume  $(f_2)$ and  $(a_2)$. Let $\{(\l_n,u_n)\}_{n\ge1}$ be a sequence of regular positive solutions of  \eqref{1.1}  satisfying \eqref{6.1} and \eqref{6.14}.  Then, for every $\eta \in (0,z)$, there exists $n_0 \in\N$ such that $x_n\in (z-\eta,x_\o)$ for all  $n\geq n_0$.
\end{lemma}
\begin{proof}
On the contrary, suppose that there exists $\eta\in (0,z)$ such that $[0,z-\eta]$ contains a subsequence of $\{x_n\}_{n\geq 1}$, still labeled by $n$. Then, without loss of generality, we can further assume that $\lim_{n\to\infty}x_n=x_*\in [0,z-\eta]$. Since $u_n$ is decreasing, it follows from \eqref{6.15} and Lemma \ref{levi.1} that
\begin{equation}
\label{6.16}
  \lim_{n\to \infty}u_n (x) = 0 \quad \hbox{uniformly in} \;\; [z-\tfrac{\eta}{2},1].
\end{equation}
Hence, by the concavity of $u_n$ in $[0,z)$, we see that, for sufficiently large $n$,
$$
  u_n'(z-\tfrac{\eta}{2})<\frac{u_n(z-\tfrac{\eta}{2})-u_n(x_n)}{z -\tfrac{\eta}{2}-x_n}= \frac{u_n(z-\tfrac{\eta}{2})-M}{z-\tfrac{\eta}{2}- x_n }  < -\frac{1}{2}\frac{M}{z-\tfrac{\eta}{2}- x_*}.
$$
In view of \eqref{6.16}, this  forces $u_n$ to vanish, by its concavity, somewhere in $[0,z]$, for large $n$. This contradiction ends the proof.
\end{proof}

\begin{lemma}
\label{levi.6}
 Assume   $(f_2)$ and  $(a_2)$.
 Let $\{(\l_n,u_n)\}_{n\ge1}$ be a sequence of regular positive solutions of  \eqref{1.1}  satisfying \eqref{6.1} and \eqref{6.14}.
 Suppose, in addition, that, for some $y \in [0,1)$, $\e>0$ and $n_0\in \N$, one has that $y+\e \leq x_n$ for all $n\geq n_0$. Then, there holds
\begin{equation*}
  \lim_{n\to \infty} u_n (x)=\infty \quad \hbox{uniformly in}\;\; [0,y].
\end{equation*}
\end{lemma}
\begin{proof}
Since $u_n$ is decreasing for all $n\geq 1$, it suffices to show that $\lim_{n\to \infty}u_n(y)=\infty$.
This is an easy consequence of Lemma \ref{levi.1}, because $u_n(y)>M$ for all $n\geq n_0$.
\end{proof}

As a byproduct of  Lemmas \ref{le6.5} and \ref{levi.6}, the next result holds.

\begin{corollary}
\label{covi.1}
 Assume  $(f_2)$ and  $(a_2)$.
 Let $\{(\l_n,u_n)\}_{n\ge1}$ be a sequence of regular positive solutions of  \eqref{1.1}  satisfying \eqref{6.1} and \eqref{6.14}. Then, for every $\eta \in (0,z)$,
\begin{equation*}
  \lim_{n\to \infty} u_n(x) =\infty \quad \hbox{uniformly in}\;\; [0,z-\eta].
\end{equation*}
\end{corollary}

The next result does estimate  the grow-up rate of $u_n$ to infinity in $[0,z)$ as $n\to \infty$.

\begin{theorem}
\label{levi.7}
 Assume  $(f_2)$ and  $(a_2)$.
 Let $\{(\l_n,u_n)\}_{n\ge1}$ be a sequence of regular positive solutions of  \eqref{1.1}  satisfying \eqref{6.1} and \eqref{6.14}.
Then, for every $\eta \in (0,z)$, there exists a constant $C>0$ and an integer $n_0 \ge 1$ such that
\begin{equation}
\label{6.17}
  u_n(x)\geq C \l_n^\frac{1}{q}\quad \hbox{for all}\;\; x \in [0,z-\eta] \;\; \hbox{and}\;\;  n\geq n_0.
\end{equation}
\end{theorem}
\begin{proof}
Indeed, by Lemma \ref{le6.5}, there exists $n_0\in \N$ such that $x_n\geq z-\frac{\eta}{2}$ for all $n\geq n_0$. Thus, fixing $n\geq n_0$ and $x\in [0,z-\eta]$, and integrating in $[x,x_n]$
the identity  \eqref{6.13} we find that
\begin{align*}
  \l_n \int_x^{x_n}a(s) \, ds
  & =   \frac{1}{f(u_n(x))} \frac{u_n'(x)}{\sqrt{1+(u_n'(x))^2}} +
   \frac{1}{f(M)} \frac{-u_n'(x_n)}{\sqrt{1+(u_n'(x_n))^2}}
   \\
   & \hspace{4.5cm} +\int_{x}^{x_n}
    \left( \frac{1}{f(u_n(s))}  \right)' \frac{u_n'(s)}{\sqrt{1+(u_n'(s))^2}}\,ds
    \\
    & < \frac{1}{f(M)} - \int_{x}^{x_n}
   \left( \frac{1}{f(u_n(s))}  \right)' \frac{-u_n'(s)}{\sqrt{1+(u_n'(s))^2}}\,ds.
\end{align*}
Since $u_n(s)> M$ for each $s\in (x,x_n)$, $u_n$ is decreasing and  $f$ is decreasing in $(M,\infty)$,
 we find that
$$
- \left( \frac{1}{f(u_n(s))}  \right)' =  \frac{f'(u_n(s)) \, u'_n(s) }{(f(u_n(s)))^2}  >0 \quad \hbox{in } (x,x_n).
$$
Thus, the previous estimate implies that
\[
  \l_n \int_x^{x_n}a(s) \, ds<
 \frac{1}{f(M)}- \int_{x}^{x_n} \left( \frac{1}{f(u_n(s))}  \right)' \,ds=
  \frac{1}{f(u_n(x))}
\]
for all $x< x_n$ and $n\geq n_0$. Consequently, since  for every $n\geq n_0$, one has that
$x \leq z-\eta < z-\frac{\eta}{2} \leq x_n$, we find that, for all $x\in [0,z-\eta]$,
\[
  \frac{1}{f(u_n(x))} > \l_n \int_x^{x_n}a(s) \, ds  > \l_n \int_{z-\eta}^{z-\frac{\eta}{2}}a(s)\,ds
\]
and then
\[
  \frac{f(u_n(x))}{u_n^{-q}(x)} \int_{z-\eta}^{z-\frac{\eta}{2}} a(s)\,ds < \frac{u_n^q(x)}{\l_n}.
\]
By Corollary \ref{covi.1}, it follows from \eqref{1.2} that $\lim_{n\to \infty}\frac{f(u_n(x))}{u_n^{-q}(x) }=h$ uniformly in $[0,z-\eta]$. Therefore, we conclude that
$$
  \liminf_{n\to\infty} \frac{u_n^q(x)}{\l_n} \geq h \int_{z-\eta}^{z-\frac{\eta}{2}} a(s)\,ds \quad
\hbox{uniformly in } [0,z-\eta].
$$
Hence  the estimate \eqref{6.17}   follows.
\end{proof}

\begin{lemma}
\label{levi.8}
 Assume  $(f_2)$ and  $(a_2)$.
 Suppose further that
 \begin{equation}
 \label{6.18}
\supess{[z+\varepsilon, x_\o]} a <0 \quad \text{for all small } \varepsilon >0.
\end{equation}
Let $\{(\l_n,u_n)\}_{n\ge1}$ be a sequence of regular positive solutions of  \eqref{1.1}  satisfying \eqref{6.1} and \eqref{6.14}.   Then, there exists
\begin{equation}
\label{6.19}
  \lim_{n\to \infty}x_n =z.
\end{equation}
\end{lemma}
\begin{proof}
On the contrary, assume that \eqref{6.19} is not true. Then, owing to Lemma \ref{le6.5},
there exist $\eta>0$ and a subsequence, still labeled by $n$,
such that $z+\eta \leq x_n$ for all $n\geq 1$. Then, since for every $n\geq 1$ and $x\in [z,1]$,
\begin{equation}
\label{4.17}
  \left( \frac{1}{\sqrt{1+(u'_n(x))^2}}\right)'=\l_n a(x)f(u_n(x))u_n'(x) \ge 0,
\end{equation}
integrating in $[z,x_n]$, we obtain
\begin{align*}
  1 & \geq \frac{1}{\sqrt{1+(u'_n(x_n))^2}}  - \frac{1}{\sqrt{1+(u'_n(z))^2}}
  =
  \l_n \int_{z}^{x_n} a(x)f(u_n(x))u_n'(x)\,dx
\\
& \geq  \l_n \int_{z+\frac{\eta}{2}}^{x_n} a(x)f(u_n(x))u_n'(x)\,dx
\geq
-\l_n  \supess{[z+\frac{\eta}{2},x_\o]} a \,  \int_{z+\frac{\eta}{2}}^{x_n} f(u_n(x))u_n'(x)\,dx
\\
& \geq
 -  \l_n    \supess{[z+\frac{\eta}{2},x_\o]} a \,  \int_{M}^{u_n(z+\frac{\eta}{2})}f(s)\,ds  .
\end{align*}
Therefore, by \eqref{6.1}, we obtain $\lim_{n\to \infty} \int_{M}^{u_n(z+\frac{\eta}{2})}f(s)\,ds =0$,
which implies $\lim_{n\to \infty} u_n(z+\tfrac{\eta}{2}) =M$, while, according to Lemma \ref{levi.6},
$\lim_{n\to \infty} u_n(z+\tfrac{\eta}{2}) =\infty$. This contradiction ends the proof.
\end{proof}

\begin{rem}
Under $(a_2)$,  the condition  \eqref{6.18} holds if, for instance,  the function $a(x) $ is continuous in $(z,x_\o]$.
\end{rem}

As a direct consequence of Lemma \ref{levi.8}, the next result holds.

\begin{corollary}
\label{covi.2}
Assume  $(f_2)$,   $(a_2)$ and  \eqref{6.18}.
Let $\{(\l_n,u_n)\}_{n\ge1}$ be a sequence of regular positive solutions of  \eqref{1.1}  satisfying \eqref{6.1} and \eqref{6.14}. Then, for every $\eta \in (0,1-z)$,
\begin{equation*}
  \lim_{n\to \infty} u_n (x)= 0 \quad \hbox{uniformly in}\;\; [z+\eta,1].
\end{equation*}
\end{corollary}
\begin{proof}
According to Lemma \ref{levi.8}, for every  $\eta \in (0,1-z)$, there exists $n_0$ such that,  for all $n\geq n_0$,  $x_n<z+\tfrac{\eta}{2}$ and hence $u_n(z+\eta) < u_n(x_n)= M$. Then, Lemma \ref{levi.1} yields $\lim_{n\to \infty} u_n(z+\eta)=0$. Since $u_n$ is decreasing, the conclusion follows.
\end{proof}

Finally, the next result estimates the decay rate of $u_n$ in the interval $(z,1]$.

\begin{theorem}
\label{levi.9}
Assume  $(f_2)$,  $(a_2)$ and    \eqref{6.18}.
Let $\{(\l_n,u_n)\}_{n\ge1}$ be a sequence of regular strictly positive solutions of  \eqref{1.1}  satisfying \eqref{6.1} and \eqref{6.14}. Then, for every $\eta \in (0,1-z)$, there exist  $C
>0$ and   $n_0 \in\N$ such that
\begin{equation}
\label{6.20}
  u_n(x)\leq C \l_n^{-\frac{1}{p}}\quad \hbox{for all}\;\; x \in [z+\eta,1] \;\; \hbox{and}\;\;  n\geq n_0.
\end{equation}
\end{theorem}
\begin{proof}
Pick $x \in [z+\eta,1]$. By \eqref{6.19}, there exists $n_0\in \NN$ such that
\begin{equation}
\label{6.21}
 z-\frac{\eta}{2}< x_n < z+\frac{\eta}{2}<z+\eta \leq x \qquad \hbox{for all $n\ge n_0$.}
 \end{equation}
Note that $0<u_n(t)\le M$ for all $n\geq n_0$ and $t\in [x_n,x]$. Thus, for any $n\geq n_0$,
since  $u_n(t)$ is decreasing, the composition $f(u_n(t))$ is  decreasing in $  [x_n,x]$ and hence
\begin{equation}
0 <\frac{1}{f(u_n(x_n))} \leq   \frac{1}{f(u_n(t))}  \leq   \frac{1}{f(u_n(x))}  \quad \hbox{for all } t\in [x_n,x].
\end{equation}
\par
Suppose $z\leq x_n$. Then, from the differential equation in \eqref{1.1}, we get, for all $t\in[x_n,x]$,
\begin{equation}
0 <- \l_n a(t)\  =
  \left( \frac{u_n'(t)}{\sqrt{1+(u_n'(t))^2}}\right)' \frac{1}{f(u_n(t))}
  \le  \left( \frac{u_n'(t)}{\sqrt{1+(u_n'(t))^2}}\right)' \frac{1}{f(u_n(x)}
\end{equation}
and hence integrating in $[x_n,x]$
\begin{align}
- \l_n\int_{x_n}^x a(t)\,dt &\le
\frac{1}{f(u_n(x))} \int_{x_n}^x \left( \frac{u_n'(t)}{\sqrt{1+(u_n'(t))^2}}\right)' \,dt
\\
&=
\frac{1}{f(u_n(x))}\left( \frac{u_n'(x)}{\sqrt{1+(u_n'(x))^2}}- \frac{u_n'(x_n)}{\sqrt{1+(u_n'(x_n))^2}}\right)  \le \frac{1}{f(u_n(x))}.
\end{align}
By \eqref{6.21}, we have that, for all $n\ge n_0$,
\[
  \int_{x_n}^x (-a(t))\,dt \ge  \int_{z+\frac{\eta}{2}}^{z+\eta}(-a(t))\,dt >0,
\]
and hence
\[
\l_n \int_{z+\frac{\eta}{2}}^{z+\eta}(-a(s))\,ds \le \frac{1}{f(u_n(x))}
\]
for all $x \in [z+\eta,1]$. Then,  the  estimate \eqref{6.20}   follows readily from the fact that
$\lim_{n\to \infty}\frac{f(u_n(x)}{u_n^p(x)}=1$, which is a consequence of \eqref{1.2} by Corollary \ref{covi.2}.
\par
Now, assume that $x_n< z$.
Then, as above, from the differential equation in \eqref{1.1}  we get, for all $t\in[z,x]$,
\begin{equation}
0 <- \l_n a(t)
  \le  \left( \frac{u_n'(t)}{\sqrt{1+(u_n'(t))^2}}\right)' \frac{1}{f(u_n(x)}
\end{equation}
and hence integrating in $[z,x]$
\begin{align}
- \l_n\int_{z}^x a(t)\,dt
&\le
 \frac{1}{f(u_n(x))}\left( \frac{u_n'(x)}{\sqrt{1+(u_n'(x))^2}}- \frac{u_n'(z)}{\sqrt{1+(u_n'(z))^2}}\right)  \le \frac{1}{f(u_n(x))}.
\end{align}
Thus, thanks again to \eqref{6.21}, we find, for all $n\ge n_0$,
\[
 \l_n   \int_{z }^{z+\eta}(-a(t))\,dt \le   \l_n\int_{z}^x (-a(t))\,dt
  \leq \frac{1}{f(u_n(x))}
\]
and the argument of the previous case allows to complete the proof.
\end{proof}

The next result establishes that, in addition, the solutions $u_n$ are rather flat on
\[
  I_\eta  = [0,z-\eta]\cup[z+\eta,1]
\]
for sufficiently small  $\eta>0$ and large $n$.

\begin{lemma}
\label{levi.10}
Assume  $(f_2)$,  $(a_2)$ and \eqref{6.18}.
Let $\{(\l_n,u_n)\}_{n\ge1}$ be a sequence of regular strictly positive solutions of  \eqref{1.1}  satisfying \eqref{6.1} and \eqref{6.14}.
Then, for every $\eta >0$ small enough, there exist  $C >0$ and $n_0\in \NN$ such that
\begin{equation}
\label{6.22}
  |u_n'(x)|\leq C \quad \hbox{for all}\;\; x\in I_\eta\;\;\hbox{and}\;\; n\geq n_0
\end{equation}
and, actually,
\begin{equation}
\label{4.21b}
\lim_{n\to \infty}u_n'(x) = 0\quad \hbox{uniformly in} \;\; [z+\eta,1].
\end{equation}
\end{lemma}
\begin{proof}
Fix $\eta\in (0,z)$ and $x\in [0,z- \frac{\eta}{2}]$.
By \eqref{6.19}, there exists $n_0\in \NN$ such that
$
z-\frac{\eta}{2} \le x_n$,  {for all $n\ge n_0$.}
Hence, it follows that  $u_n(t)\ge  M$ for all   $n\ge n_0$ and all $t \in [0,z-\frac{\eta}{2}]$. Thus, since $u_n$ is decreasing, the composition $f(u_n(t))$ is increasing in $[0,z-\frac{\eta}{2}]$.
Consequently, integrating the differential equation in \eqref{1.1} in $[0,x]$, we find that
\begin{equation}
\label{6.23}
  \l_n f(u_n(x)) \int_0^x a(t)\,dt \ge   \int_0^x  \l_n  f(u_n(t)) \, a(t)\,dt
= \frac{-u_n'(x)}{\sqrt{1+(u_n'(x))^2}}.
 \end{equation}
Suppose that \eqref{6.22} is false. Then, for a subsequence, still labeled by $n$, we have that
\begin{equation*}
  \lim_{n\to \infty}u_n'(z-\eta)=-\infty,
\end{equation*}
which implies $\lim_{n\to \infty}u_n'(x)=-\infty$ uniformly in $[z-\eta,z]$ and hence
\begin{equation}
\label{6.24}
  \lim_{n\to \infty}  \frac{-u_n'(x)}{\sqrt{1+(u_n'(x))^2}} =1\quad \hbox{uniformly in}\;\;  [z-\eta,z].
\end{equation}
On the other hand, integrating the differential equation in $[z-\eta,z-\frac{\eta}{2}]$ yields
\[
  \frac{u_n'(z-\eta)}{\sqrt{1+(u_n'(z-\eta))^2}}-  \frac{u_n'(z-\frac{\eta}{2})}{\sqrt{1+(u_n'(z-\frac{\eta}{2}))^2}}=\int_{z-\eta}^{z-\frac{\eta}{2}}
  \l_n f(u_n(x))a(x)\,dx
\]
and then, owing to \eqref{6.23} and \eqref{6.24},
\[
  \frac{u_n'(z-\eta)}{\sqrt{1+(u_n'(z-\eta))^2}}-  \frac{u_n'(z-\frac{\eta}{2})}{\sqrt{1+(u_n'(z-\frac{\eta}{2}))^2}}\geq
  \int_{z-\eta}^{z-\frac{\eta}{2}}  \frac{-u_n'(x)}{\sqrt{1+(u_n'(x))^2}} \frac{a(x)}{\int_0^x a(t)\,dt}  \,dx.
\]
Therefore, letting $n\to \infty$ in this inequality, we find that
\[
0\ge  \int_{z-\eta}^{z-\frac{\eta}{2}}  \frac{a(x)}{\int_0^x a(t)\,dt} \,dx > 0,
\]
which is impossible. This contradiction provides us with the uniform bound for $u_n'$
on $[0,z-\eta]$.

Next, we prove \eqref{6.22}, which obviously yields a uniform bound for  $u_n'$
on $[z+\eta,1]$.
According to Lemma \ref{levi.1} and  Lemma \ref{levi.8}, we have that $\lim_{n\to \infty} u_n(x)  =0$ uniformly in $[z+\eta,z]$. Since, for every $x \in [z+\eta,1]$ and $n\geq 1$, $u_n(1)=u_n(x)+\int_x^1 u_n'(t)\,dt$, we infer that $\lim_{n\to \infty} \int_x^1 u_n'(t)\,dt =0$. Consequently,  as $u_n$ is convex on $(z,1]$, also \eqref{6.22} is proven.
\end{proof}

As a byproduct of Lemma \ref{levi.10} the next result holds.

\begin{corollary}
\label{covi.3}
Assume  $(f_2)$,  $(a_2)$ and \eqref{6.18}.
Let $\{(\l_n,u_n)\}_{n\ge1}$ be a sequence of regular strictly positive solutions of  \eqref{1.1}  satisfying \eqref{6.1} and \eqref{6.14}.
Then, the following holds:
\begin{enumerate}
\item[{\rm (i)}] for every $\eta \in (0,z)$, $\lim_{n\to \infty}\frac{u_n(x)}{u_n(0)}=1$ uniformly in
$[0,z-\eta]$;
\item[{\rm (ii)}] for every $\eta \in (0,1-z)$, $\lim_{n\to \infty}u_n(x) =  \lim_{n\to \infty}u_n'(x) = 0$ uniformly in $[z+\eta,1]$.
\end{enumerate}
\end{corollary}
\begin{proof}
For every $x\in [0,z-\eta]$ and $n\geq 1$, we have  that $u_n(x)=u_n(0)+\int_0^x u_n'(t)\,dt$ and hence
\begin{equation}
\label{6.25}
  \frac{u_n(x)}{u_n(0)}=1+\frac{\int_0^x u_n'(t)\,dt}{u_n(0)}.
\end{equation}
Since,  by Lemma \ref{levi.10},
\[
  \Big| \frac{\int_0^x u_n'(t)\,dt}{u_n(0)} \Big|\leq \frac{C (z-\eta)}{u_n(0)},
\]
conclusion (i) follows from Corollary \ref{covi.1} by letting $n\to \infty$ in \eqref{6.25}.
As for the proof of   (ii), the conclusion follows from  Lemma \ref{levi.1}, Lemma \ref{levi.8} and Lemma \eqref{levi.10}.
\end{proof}

At the light of  these results, for sufficiently large $\l$, the regular positive solutions of
\eqref{1.1} bounded away from zero have the profile already shown in Figure \ref{Fig01}.
\par
Although, due to Lemma \ref{levi.8}, $\lim_{n\to\infty}x_n=z$, in general it is unknown whether or not $x_n=z$. According to Corollaries \ref{covi.1} and \ref{covi.2}, the solutions grow-up to infinity in the interval $(0,z)$, whereas decay to zero on $(z,1)$.   Theorems \ref{levi.7} and  \ref{levi.9} provide us with some sharp  estimates for the growth and decay rates of $u_n$ in $(0,z)$ and $(z,1)$, respectively. According to Corollary \ref{covi.3}, the larger is $\l$ the flatter are the solutions on $(0,z)$ and $(z,1)$.

\section{Regularity versus singularity}
\label{s7}

Our aim in this section is to discuss  the existence and the non-existence of singular solutions of problem \eqref{1.1}.

\subsection{A general   regularity criterion}

\noindent Based on some ideas from  our previous papers \cite{LGO-19, LGO-JDE, LGO-20}, we  establish here a    criterion for ascertaining  the local regularity  of the bounded variation solutions of the equation
\begin{equation}
\label{7.1}
   \displaystyle
   -\left( \frac{u'}{\sqrt{1+(u')^2}}\right)' =h(x), \quad 0<x<1,
\end{equation}
under the assumption
\begin{itemize}
\it
\item[$(h_1)$] $h\in  L^1(0,1)$ and there exist  $z \in (0,1)$,  $\d_1>0$ and $\d_2 >0$  such that  $h(x)\ge 0$ a.e.  in $(z-\d_1,z)$ and $h(x) \le 0$ a.e. in $(z,z+\d_2)$.
\end{itemize}
A bounded variation solution of    \eqref{7.1}  is a function $u  \in BV(0,1)$ such that
\begin{align}
\label{hh}
\int_0^1 \frac{Du^a  D\phi^a}{\sqrt{1+ (Du^a)^2}}\,dx
+\int_0^1  \frac{Du^s}{ |Du^s | }   \, D\phi^s
=\int_0^1 h\phi\,dx
\end{align}
for every $\phi \in BV(0,1)$  with essential support in $(0,1)$ and such  that $ |D\phi ^{s}|$ is absolutely continuous with respect to  $|Du^{s}|$.
\par
From the proof of \cite[Prop. 3.6]{LGO-19} we infer that, under  $(h_1)$, every    bounded   variation  solution $u$  of \eqref{7.1}  satisfies the  following  conditions:

 \begin{itemize}
\item
$u$ is concave   in $(z-\d_1, z)$ and convex  in $(z, z+\d_2)$,
\begin{align}
  u_{|{(z-\d_1, z)}}\in W^{2,1}_{\text{\rm loc}}(z-\d_1, z)\cap W^{1,1}(z-\d_1, z),\\
  u_{|{(z,z+\d_2)}}\in W^{2,1}_{\text{\rm loc}}(z,z+\d_2)\cap W^{1,1}(z,z+\d_2),
\end{align}
and
\begin{equation}
\label{7.2}
- \left(\frac{ u' }{\sqrt{1 +{u'}^2 }}\right)' = h(x ) \; \text{ a.e. in }  (z-\d_1,z+\d_2);
\end{equation}
\item
either $u\in W^{2,1}_{\text{\rm loc}} (z-\d_1,z+\d_2)$,  or else
$$
    u(z^-) \ge  u(z^+)\quad \hbox{and}\quad u'(z^-) = -\infty =  u'(z^+),
$$
where $u'(z^-)$ and $u'(z^+)$  stand for the left and right Dini derivatives of $u$ at $z$.
 \end{itemize}

The following theorem determines   whether $u$ is regular or not,
depending on the behavior  of $h $   near its nodal point $z$;  more precisely, on the integrability properties of the function $(\int_x^z h(t)\, dt) ^{-\frac{1}{2}}$, as expressed by conditions \eqref{7.5} and \eqref{7.6} below.  Note that, under assumption $(h_1)$, the continuous function $\int_x^z h(t)\, dt $ is non-increasing in $(z-\d_1,z)$ and non-decreasing in $(z,z+\d_2)$; thus, either
\begin{equation}
\label{7.3}
\int_x^z h(t)\, dt >0 \quad \text{for all } x\in (z-\d_1,z+\d_2) \setminus \{z\},
\end{equation}
or there is $x_0\in (z-\d_1,z)$ such that
$$\int_x^z h(t)\, dt =0 \quad\text{for all }   x\in [x_0,z],$$
 or there is $x_0\in (z,z+\d_2)$ such that
 $$\int_x^z h(t)\, dt =0 \quad\text{for all }   x\in [z, x_0]. $$
 Hence,
\eqref{7.3} is complementary of
\begin{equation}
\label{7.4}
\int_{ x_0}^z h(t)\, dt =0 \quad \hbox{for some}\;\;x_0\in (z-\d_1,z+\d_2) \setminus \{z\}.
\end{equation}

\begin{theorem}
\label{th7.1}
Assume  $(h_1)$ and let $u$ be  a bounded   variation  solution of \eqref{7.1}.
 Then, the following assertions are true:
 \begin{enumerate}
\item[{\rm (a)}]
$u\in W^{2,1}_{\text{\rm loc}} (z-\d_1,z+\d_2)$ if
 \begin{equation}
\label{7.5}
\text{either}
 \;\;
\int_{z-\d_1}^z \left(\int_x^z h(t)\, dt\right) ^{-\frac{1}{2}}dx =\infty,
 \;\;  \text{or} \;\;
 \int_z^{z+\d_2} \left(\int_x^z h(t)\, dt\right) ^{-\frac{1}{2}}dx =\infty ;
\end{equation}
 \item[{\rm (b)}]
 $   u(z^-) >   u(z^+) $ if \eqref{7.3} holds and there are $x_1 \in (z-\d_1,z)$, $x_2 \in (z,z+\d_2)$ such that
\begin{equation}
\label{7.6}
\begin{array}{cc}
\displaystyle
\int_{x_1}^z \left(\int_x^z h(t)\, dt\right) ^{-\frac{1}{2}}dx <\infty,
\qquad
 \int_z^{x_2} \left(\int_x^z h(t)\, dt\right) ^{-\frac{1}{2}}dx <\infty,
 \\[2mm]
 \displaystyle
\int_{x_1}^{x_2} \left(\int_x^z h(t)\, dt\right) ^{-\frac{1}{2}}dx \leq u(x_1)
- u(x_2).
\end{array}
\end{equation}
\end{enumerate}
\end{theorem}
It is understood   that condition  \eqref{7.5}  is    satisfied  whenever \eqref{7.4} holds.

\begin{proof}
By $(h_1)$, either \eqref{7.3}, or \eqref{7.4},  holds.  Let us  prove Part (a).
Assume \eqref{7.4} with $x_0\in (z-\d_1,z)$, the argument being similar in case  $x_0\in (z,z+\d_2)$.
Then, integrating  the equation \eqref{7.2} on $(x_0,z)$
yields
\begin{equation}
\label{7.7}
\frac{-u'(x_0) }{\sqrt{1 +(u'(x_0))^2 } } = \frac{ -u'(z^-) }{\sqrt{1 +(u'(z^-))^2 } } -  \int_{x_0}^z h(t)\, dt = \frac{ -u'(z^-) }{\sqrt{1 +(u'(z^-))^2 } }.
\end{equation}
As   $u'\in W^{1,1}_{\text{\rm loc}}(z-\d_1,z)$,  and hence $\frac{|u'(x_0)| }{\sqrt{1 +(u'(x_0))^2 } }<1$,
it follows  from \eqref{7.7}  that   $u'(z^-)$ is finite. Therefore,  $u \in W^{2,1}_{\text{\rm loc}} (z-\d_1,z+\d_2)$.  Now suppose \eqref{7.3}. Then, for every  $t\in  (z-\d_1,z)$, integrating \eqref{7.2} in $(t,z)$, we obtain that
\begin{equation}
\label{7.8}
\frac{- u'(t) }{\sqrt{1 +(u'(t))^2 } }= \frac{- u'(z^-) }{\sqrt{1 +(u'(z^-))^2 } } - \int_t^z h(s)\, ds
\end{equation}
and thus
\begin{equation}
\label{7.9}
 - u'(t)  = \frac{ \frac{ -u'(z^-) }{\sqrt{1 +(u'(z^-))^2 } } - \int_t^z h(s)\, ds}
 {\sqrt{1 -\frac{u'(z^-) }{\sqrt{1 +(u'(z^-))^2 } } - \int_t^z h(s)\, ds}}
  \frac{1}  {\sqrt{1+ \frac{u'(z^-) }{\sqrt{1 +(u'(z^-))^2 } } + \int_t^z h(s)\, ds}}.
\end{equation}
Assume   \eqref{7.5}.   Without loss of generality we can suppose  that
\begin{equation}
\label{7.10}
\int_{z-\d_1}^z \left(\int_x^z h(t)\, dt\right) ^{-\frac{1}{2}}dx =\infty,
\end{equation}
because the proof is similar if  $\int_z^{z+\d_2}\left(\int_x^z h(t)\, dt\right) ^{-\frac{1}{2}}dx =\infty$. As   the function $\int_t^z h(s)\, ds$ is continuous and positive for  $t\in(z-\d_1,z)$, the condition  \eqref{7.10} can be expressed  as
\begin{equation}
\label{7.11}
\int_x^z \left(\int_t^z h(s)\, ds\right) ^{-\frac{1}{2}}dt =\infty \quad \text{for all } x\in  (z-\d_1,z).
\end{equation}
To prove that $u\in W^{2,1}_{\text{\rm loc}} (z-\d_1,z+\d_2)$, suppose, on the contrary, that $u'(z^-)=-\infty$, that is, $ \frac{ -u'(z^-) }{\sqrt{1 +(u'(z^-))^2 }} =1$. Hence,  \eqref{7.9} implies  that, for every $x\in (z-\d_1,z)$,
\begin{equation}
\label{7.12}
 - u'(x)  = \frac{1 - \int_x^z h(t)\, dt}
{\sqrt{2 - \int_x^z h(t)\, dt} }
  \frac{1}  {\sqrt{ \int_x^z h(t)\, dt} }.
\end{equation}
 As  there exists   $\eta \in (0,\d_1)$ such that $  \int_x^z h(t)\, dt \le \frac{1}{2}$ for all $x\in (z-\eta,z)$, \eqref{7.12} implies
\begin{equation}
 - u'(x)  \ge     \frac{1}{2\sqrt{ 2}}  \left( \int_x^z h(t)\, dt \right)^{-\frac{1}{2}}.
\end{equation}
  Therefore, by \eqref{7.11}, integrating this inequality in $(z-\eta, z)$ yields
\begin{equation}
u(z-\eta) - u(z^-)   \ge  \frac{1}{2\sqrt{ 2}}    \int_{z-\eta}^z  \left( \int_x^z h(t)\, dt \right)^{-\frac{1}{2}}dx = \infty,
\end{equation}
which is a contradiction, because   $u\in L^\infty(z-\d_1,z+\d_2)$.
This ends the proof of Part (a).
\par
 In order to prove Part (b),  observe that the first two inequalities in \eqref{7.6}
are  equivalent to
\begin{equation}
\label{i6c}
\int_{z-\d_1}^{z+\d_2} \left(\int_x^z h(t)\, dt\right) ^{-\frac{1}{2}}dx < \infty.
\end{equation}
Moreover, without loss of generality we  can   suppose  that     $u'(x_1)\le0$ and $u'(x_2)\le0$. Indeed, otherwise there is  $\hat x_1\in (x_1, z)$  such that  $u(\hat x_1) \ge u(x_1)$ and $u'(\hat x_1)\le0$, or $\hat x_2\in (z,x_2)$ with $u(\hat x_2) \le u(x_2)$ and  $u'(\hat x_2) \le 0$.  Replacing $x_1$ by $\hat x_1$,  or $x_2$ by $\hat x_2$,  we are done.
\par
 We claim that, under condition \eqref{7.3},
\begin{equation}
\label{7.13}
 0\le  -u'(t) <    \left(\int_t^z h(s)\, ds\right) ^{-\frac{1}{2}}\;\; \text{ for all }  t\in[x_1,x_2]\setminus\{z\}.
\end{equation}
  Pick $t\in[x_1,z)$. Since $u'(x_1) \le0$ and  $u(x) $ is concave in $[x_1,z)$, we have that
\begin{equation}
0\le \frac{- u'(t) }{\sqrt{1 +(u'(t))^2 } } < 1
\quad \text{for all $t\in [x_1,z)$.}
\end{equation}
and
\begin{equation}
0 \le  \frac{-u'(z^-) }{\sqrt{1 +(u'(z^-))^2 }} \le 1.
\end{equation}
  Hence, by \eqref{7.3}, it follows from \eqref{7.8} that
\begin{equation}
0\le \frac{- u'(z^-) }{\sqrt{1 +(u'(z^-))^2 } } - \int_t^z h(s)\, ds <1 \quad \text{for all $t\in [x_1,z)$.}
\end{equation}
 Consequently, since $1+ \frac{ u'(z^-) }{\sqrt{1 +(u'(z^-))^2 } } \ge 0$,  the validity of \eqref{7.13} for all $t \in [x_1,z)$ follows easily from \eqref{7.9}. As the proof of \eqref{7.13} for $t\in(z,x_2]$ proceeds similarly, the technical details are omitted here.
\par
Next, pick   $x \in (z-\d_1, z+\d_2)\setminus\{z\}$. Integrating \eqref{7.13}  we obtain the estimates
\begin{align}
\label{7.14}
u(x) - u(z ^-) &< \int_x^z \left(\int_t^z h(s)\, ds\right) ^{-\frac{1}{2}}dt\;\; \text{ for all }  x\in(z-\d_1,z),
\\
\label{7.15}
u(z^+) - u(x) &< \int_z^x\left(\int_t^z h(s)\, ds\right) ^{-\frac{1}{2}}dt\; \;\text{ for all }  x\in(z,z+\d_2).
\end{align}
Taking $x=x_1$ in \eqref{7.14} and $x=x_2$ in \eqref{7.15} and adding up the two inequalities yields
$$
   u(x_1) - u(z^-) + u(z^+) - u(x_2) < \int_{x_1}^{x_2} \left(\int_x^z h(t)\, dt\right) ^{-\frac{1}{2}}dx.
$$
Thus, thanks to \eqref{7.6}, we find
$$
\int_{x_1}^{x_2}  \left(\int_t^z h(s)\, ds\right) ^{-\frac{1}{2}} dt >
   - u(z^-) + u(z^+) + \int_{x_1}^{x_2}  \left(\int_t^z h(s)\, ds\right) ^{-\frac{1}{2}} dt
$$
and therefore, $ u(z^-) > u(z^+)$, which  ends the proof.
\end{proof}

\begin{rem}
It is straightforward to see that similar conclusions hold  by  imposing  in $(h_1)$, alternatively, that
$h(x)\le 0$ a.e. in   $(z-\d_1,z)$ and $h(x) \ge 0$ a.e. in $(z,z+\d_2)$.
\end{rem}

\begin{rem}
\label{re7.2}
 The conditions   \eqref{7.5} and \eqref{7.6}  measure   the smoothness of the function $h(x)$ at the  nodal point $z$.
Indeed,   \eqref{7.5}  holds true, in particular, when $h(z^-)= $ ${\rm ess }  \lim_{x\to z^-}h(x) $ $ =0$ and $h(x)$ has a bounded slope  on the left  of $z$, or  when  $h(z^+)= $ ${\rm ess } \lim_{x\to z^+}h(x)=0$ and $h(x)$ has a bounded slope  on the right  of $z$, while   \eqref{7.5}  fails
if, for instance,  both $h(z^-)$ and $h(z^+)$  exist, are finite  and $h(z^-)>0 > h(z^+)$.
This way a classical regularity result,  requiring   the function $h(x)$ to be globally Lipschitz on $(0,1)$ (see, e.g., \cite{Giu}), is  significantly improved in the frame of equation  \eqref{7.1}.
\end{rem}

\subsection{Non-existence of singular solutions}

\noindent
A direct consequence of Theorem \ref{th7.1} is  the following  result,  which guarantees the regularity of all possible solutions of the problem \eqref{1.1} if the function $a(x) $ satisfies  $(a_2)$ and
\begin{itemize}
\it
\item[$(a_4)$]
$\displaystyle \hbox{either}\;\; \int_{0}^z \left( \int_x^za(t)\,dt\right)^{-\frac{1}{2}}dx =\infty,\;\;
\hbox{or}\;\; \int_z^{1} \left( \int_x^za(t)\,dt\right)^{-\frac{1}{2}}dx =\infty$.
\end{itemize}

\begin{theorem}
\label{th7.2}
 Assume $(a_2)$, $(a_4)$, $\l>0$,  and
\begin{itemize}
\item[$(f_4)$] $f\in \mc{C}(\R)$ satisfies  $f(u)\ge0$ if $u\ge0$.
\end{itemize}
Then, any positive solution $(\l,u)$ of \eqref{1.1} is regular.
\end{theorem}
\begin{proof}
Let $(\l,u)$, with $\l>0$, be a positive solution of  \eqref{1.1},   and set
$$
h(x) = \l a(x) f(u(x)) \quad \text{for a.e. } x\in (0,1).
$$
The  composite function  $ f(u(x))$ lies in  $L^\infty(0,1)$   and, by $(f_4)$,
satisfies  $   f(u(x)) \ge 0$ for a.e. $x\in(0,1)$. Thus, by $(a_2)$, $h$ satisfies $(h_1)$, with $\d_1=z$ and $\d_2=1-z$, and either  \eqref{7.4} holds for some $x_0\in [0,1] \setminus \{z\} $, or \eqref{7.3} and \eqref{7.5} hold. By Theorem \ref{th7.1} (a),    $u$ is regular.
\end{proof}

Theorem \ref{th7.2}  holds true regardless the particular behavior
of $f(u)$ at zero and at infinity: it only requires   the continuity and positivity of  $f(u)$. Thus, it is a quite general and versatile result  that applies to a large variety of situations and allows to complete and sharpen several previous statements, such as the ones obtained  in   \cite{LOR1,    LOR2, LGO-19, LGO-JDE}.
Indeed, assuming further $(a_2)$ and $(a_4)$, the results in \cite[Thms. 1.1--1.6]{LOR1}, in \cite[Thms. 5.13 and 5.14]{LGO-19} and in \cite[Thms. 1.1 and 6.1]{LGO-JDE}, combined with  Theorem \ref{th7.2}, provide the existence and the multiplicity of regular solutions. In particular, thanks to \cite[Thms. 1.1 and 6.1]{LGO-JDE} and Theorem \ref{th7.2} a wrong assertion  in   \cite[Thm. 7.1]{LOR2} can be  corrected and the  situation completely clarified. Theorem \ref{th7.2} also guarantees  that,  in the frame of Theorem \ref{thiii.4}, one has, under $(a_2)$ and $(a_4)$,
$\mc{S}_{bv}^+= \mc{S}_r^+$ and $\mf{C}_{bv,\l_0}^+=\mf{C}_{r,\l_0}^+$, as illustrated in  Figure \ref{Fig02}.

\subsection{Existence of singular solutions}

\noindent
In this section we   assume that $(a_4)$ fails, i.e.,
\begin{itemize}
\it
\item[$(a_5)$]
$\displaystyle \int_{0}^z \left( \int_x^za(t)\,dt\right)^{-\frac{1}{2}}dx <\infty\;\;
\hbox{and}\;\; \int_z^{1} \left( \int_x^za(t)\,dt\right)^{-\frac{1}{2}}dx <\infty$.
\end{itemize}
The next result shows that   \eqref{1.1} can admit singular solutions   under $(a_5)$.

\begin{theorem}
\label{th7.3}
Assume   $(a_2)$ and  $(a_5)$. Then, the following assertions are true:
\begin{enumerate}
\item[{\rm (i)}] for every $p>1$ and $q\in (0,1)$, there exists a function $f(u)$ satisfying $(f_1)$ for which \eqref{1.1} admits a singular solution $(\l_s,u_s)$ for some $\l_s>0$;
\vskip1mm
\item[{\rm (ii)}] for every $q\in (0,1)$ and $\overline \l>0$, there exist $\e>0$ and a function $f(u)$ satisfying $(f_1)$ with $p=1$ such that \eqref{1.1} admits a singular solution $(\l_s,u_s)$
  for some  $\l_s >0$ with $|\l_s- \overline  \l| <\e$.
\end{enumerate}
\begin{proof}

For any given $p>1$ and $q\in (0,1)$, let $\tilde f\in\mc{C}^1(\R)$ be such that $\tilde f(u)>0$ and $\tilde f'(u)\geq 0$ for all $u>0$ and
\begin{equation*}
  \lim_{u\to 0^+}\frac{\tilde f(u)}{u^{p}}=1
  \quad \text{and} \quad
  \lim_{u\to  \infty}\frac{\tilde f(u)}{u^{q}}=1.
\end{equation*}
Then, thanks to  \cite[Thm. 1.1]{LGO-JDE}, the auxiliary problem
\begin{equation}
\label{7.16}
   \left\{ \begin{array}{ll}
   \displaystyle
   -\left( \frac{u'}{\sqrt{1+(u')^2}}\right)' =\l a(x) \tilde f(u), & \quad 0<x<1, \\[1ex]
   u'(0)=u'(1)=0. & \end{array}\right.
\end{equation}
possesses a singular solution $(\l_s,u_s)$ for some  $ \l_s>0$. Let $M>0$ be such that
\begin{equation}
\label{7.17}
  M> u_s(0)=\|u_s\|_\infty
\end{equation}
and consider any function $f\in \mc{C}^1(\R)$ such that
\begin{equation}
\label{7.18}
  f(u) =  \left\{ \begin{array}{ll} \tilde f(u) & \quad \hbox{if}\;\; u\leq M,\\[1ex]
  \tilde g(u) & \quad \hbox{if}\;\; u > M, \end{array}\right.
\end{equation}
where $\tilde g\in \mc{C}[M,\infty)$ is any  function such that
\begin{equation}
\label{7.19}
  \lim_{u\to \infty}\frac{\tilde g(u)}{u^{-q}}= h,
\end{equation}
for some constant $h>0$.
Then, by construction, $f(u)$ satisfies $(f_1)$ and $(\l_s,u_s)$ is a singular solution of \eqref{1.1}, thus proving Part (i).
\par
To prove Part (ii) one can proceed as follows.
For any given $\overline \l>0$, let $\tilde f\in \mc{C}^1(\R)$ be any function satisfying
\begin{equation}
  \lim_{u\to 0^+}\frac{\tilde f(u)}{u }=1
    \quad \text{and} \quad
  \lim_{u\to  \infty} \tilde f(u) = \frac{1}{ {\overline \l} {\int_0^z a(x)\,dx} }.
\end{equation}
  As due to \cite[Thm. 1.1]{LGO-20} the singular solutions of \eqref{7.16}  bifurcate from infinity at $\l_\infty= \overline \l$, there exist  $\e>0$ and a singular solution  $(\l_s,u_s)$ of \eqref{7.16}  for some $\l = \l_s>0 $ with $|\l_s-\overline \l|<\e$.
\par
Let $M>0$ be satisfying \eqref{7.17} and consider any function $f(u)$ of the form \eqref{7.18} with
$\tilde g\in \mc{C}[M,\infty)$ satisfying \eqref{7.19}. Then, $f(u)$ satisfies $(f_1)$ and $(\l_s,u_s)$ is a singular solution of \eqref{1.1}.
\end{proof}
\end{theorem}

When $a(x)$ has a jump discontinuity at $z$, that is,
$$
 \limess{x\to z^-}a(x)>0>  \limess{x\to z^+}a(x),
$$
  our next result shows  that the problem \eqref{1.1}  cannot admit regular solutions separated away from zero for sufficiently large $\l>0$.

\begin{theorem}
\label{th7.4}
Assume  $(f_2)$,  $(a_2)$,
\begin{itemize}
\it
\item[$(a_6)$] there exist constants $A>0$, $B>0$ and $\eta>0$ such that
\begin{align}
a(x) \ge A \; \text{ for a.e. } x\in (z-\eta, z) \quad \text{ and } \quad
a(x) \le -B \; \text{ for a.e. } x\in (z, z+\eta),
\end{align}
\end{itemize}
and \eqref{6.18}.
Then, the problem \eqref{1.1} cannot admit regular solutions separated away from zero
for sufficiently large $\l>0$.
\end{theorem}

  By  Remark \ref{re7.2},  $(a_6)$ implies $(a_5)$.
We conjecture   that, more generally,  Theorem \ref{th7.4} remains true if  $a(x)$ satisfies $(a_5)$  instead of $(a_6)$.

\begin{proof}
Assume, by contradiction, that \eqref{1.1} possesses a sequence of positive regular solutions, $\{(\l_n,u_n)\}_{n\geq 1}$,  such that
\begin{equation}
\label{7.20}
\lim_{n\to \infty} \l_n = \infty
\quad \text{and} \quad
\liminf_{n\to \infty} u_n(0) >0.
\end{equation}
By Lemma \ref{le4.4}, for sufficiently large $n$ there exists a unique $x_n\in (0,1)$ such that $u_n(x_n)=M$. From Lemma \ref{levi.8}, we   know that $\lim_{n\to \infty}x_n=z$.
Without   loss of generality, we can  suppose that, for every $n\ge1 $,
$x_n\in (z-\eta, z+\eta)$. We claim that, in addition,
\begin{equation}
\label{7.21}
   \lim_{n\to \infty} u_n(z)=M.
\end{equation}
To prove this, we will distinguish, for each $n$, between two different cases: either $x_n>z$, or $x_n\le z$.
Suppose that $x_n>z$. Then, integrating in $[z,x_n]$ the
identity
\begin{equation}
\label{7.22}
\left( \frac{1}{\sqrt{1+(u'_n(x))^2}}\right)'=\l_n a(x)f(u_n(x))u_n'(x)
\end{equation}
yields
\[
 \int_z^{x_n}a(x)f(u_n(x))u'_n(x)\,dx =
  \frac{1}{\l_n}\left( \frac{1}{\sqrt{1+(u_n'(x_n))^2}}-\frac{1}{\sqrt{1+(u_n'(z))^2}}\right)
  \le \frac{1}{\l_n}.
\]
Thus, we infer from the first limit in \eqref{7.20} that
\begin{equation}
\label{7.23}
  \lim_{n\to \infty} \int_z^{x_n}a(x)f(u_n(x))u'_n(x)\,dx=0.
\end{equation}
Moreover, we have that
\begin{align*}
\int_z^{x_n}a(x)f(u_n(x))u'_n(x)\,dx & =\int_z^{x_n}(-a(x))f(u_n(x))(-u'_n(x))\,dx
\\ &
\geq
B\int_z^{x_n}f(u_n(x))(-u'_n(x))\,dx
= B \int_M^{u_n(z)}f(s)\,ds.
\end{align*}
Consequently,
the following inequalities hold
\begin{equation}
 0\le B\int_M^{u_n(z)}f(s)\,ds \leq \int_z^{x_n}a(x)f(u_n(x))u'_n(x)\,dx.
\end{equation}
Similarly, if $x_n\le z$,  we  can obtain  that
\begin{equation}
0\le A\int^M_{u_n(z)}f(s)\,ds \leq -\int^z_{x_n}a(x)f(u_n(x))u'_n(x)\,dx.
\end{equation}
Since $A>0$ and $B>0$,
from \eqref{7.23} we can  conclude  that $\lim_{n\to \infty} \int_M^{u_n(z)}f(s)\,ds=0$. Therefore, \eqref{7.21} holds.  Next, integrating \eqref{7.22} in $[z-\eta,z+\eta]$ yields
\[
\int_{z-\eta}^{z+\eta}a(x)f(u_n(x))u'_n(x)\,dx = O(\l_n^{-1})\quad \hbox{as}\;\; n\to \infty,
\]
or, equivalently,
\[
  \int_{z-\eta}^{z}a(x)f(u_n(x))(-u'_n(x))\,dx =-\int_{z}^{z+\eta}a(x)f(u_n(x))(-u'_n(x))\,dx + O(\l_n^{-1}).
\]
  Thus,  arguing as above, we get
\[
 A  \int_{u_n(z)}^{u_n(z-\eta)}f(s)\,ds \leq
  \|a\|_{L^\infty(0,1)}
 \int_{u_n(z+\eta)}^{u_n(z)}f(s)\,ds +O(\l_n^{-1}).
\]
Therefore, letting $n\to \infty$ in this   estimate, \eqref{7.21} and Corollaries \ref{covi.1} and  \ref{covi.2} imply
\begin{equation}
\label{7.24}
 A \int_{M}^{\infty}f(s)\,ds \leq
   \|a\|_{L^\infty(0,1)} \int_{0}^{M}f(s)\,ds<\infty.
\end{equation}
Since  $A>0$,
\eqref{7.24} entails $\int_{M}^{\infty}f(s)\,ds<\infty$,  which is impossible, as the assumption $q\in (0,1)$, made in $(f_1)$,  implies that $\int_M^\infty f(s) \,ds =\infty$. This contradiction shows that \eqref{1.1} cannot admit,  for large $\l>0$, positive regular solutions separated away from $0$.
\end{proof}

Under conditions $(f_1)$ and $(a_2)$, the existence of solutions separated away from zero for sufficiently large $\l>0$  is guaranteed, thanks to Lemma  \ref{le6.2}, when $p\in (0,1)$  by \cite[Thm. 1.2]{LOR1}, or when $p=1$ by \cite[Thm. 1.4,  Rem. 1.9]{LOR1},
Whereas in case $p>1$, thanks to \cite[Thm. 1.5]{LOR1}, it is  known that  \eqref{1.1} admits, at least, two positive solutions for sufficiently large $\l>0$.
According to Theorem \ref{th7.4}, under conditions $(f_2)$, $(a_2)$ and $(a_6)$, all the  solutions of \eqref{1.1} for sufficiently large $\l>0$ are singular, except the ones perturbing from zero, whose existence was  discussed in Section \ref{s5}.

\begin{rem}
The proof of Theorem \ref{th7.4} actually provides us
with singular solutions as $\l \to \infty$ for a much wider
family  of functions $f(u)$ than those satisfying $(f_1)$. Indeed, to fix ideas suppose that
\begin{equation}
\label{7.25}
  a(x)=\left\{ \begin{array}{ll} A & \quad \hbox{in}\;\; [0,z),\\
  -B & \quad \hbox{in}\;\; (z,1],\end{array}\right.
\end{equation}
for two positive constants, $A, B>0$, such that
\[
  \int_0^1 a(x)\,dx = Az-B(1-z)<0=(A+B)z-B<0.
\]
Let $\{(\l_n,u_n)\}_{n\geq 1}$ be a sequence of positive regular solutions of \eqref{1.1}
satisfying \eqref{7.20}.
Then, integrating in $[0,1]$ the identity \eqref{7.22} yields $\int_0^1 a(x)f(u_n(x))u'_n(x)\,dx=0$
for all $n\geq 1$. Thus,
\[
  \int_0^za(x)f(u_n(x))u'_n(x)\,dx =-\int_z^1 a(x)f(u_n(x))u'_n(x)\,dx
\]
and hence, by \eqref{7.25},    we find that, for every $n\geq 1$,
\begin{equation}
\label{7.26}
  A \int_{u_n(z)}^{u_n(0)}f(s)\,ds = B \int_{u_n(1)}^{u_n(z)}f(s)\,ds.
\end{equation}
Consequently, letting $n\to \infty$,  from  the analysis   done in Section \ref{s6}, we infer that
\begin{equation}
\label{7.27}
  A \int_{M}^\infty f(s)\,ds = B \int_0^M f(s)\,ds.
\end{equation}
Therefore, \eqref{7.27} is necessary in order  that \eqref{1.1} can admit a regular solution
for sufficiently large $\l>0$. Obviously, it fails to be true when $f(u) $ satisfies $(f_1)$ with $q\in (0,1)$. Yet,  even when $f(u)$ has a sufficiently fast decay at
infinity so that $\int_M^\infty f(s)\,ds<\infty$, the identity \eqref{7.27} will never be satisfied, unless
\begin{equation*}
   B= \frac{\int_{M}^\infty f(s)\,ds}{\int_0^M f(s)\,ds} A.
\end{equation*}
In all these cases, it is possible to show that \eqref{1.1} cannot admit a regular solution for sufficiently large $\l>0$, regardless the decay rate of $f(u)$ at infinity. This result sharpens Theorem \ref{th7.4} in the special case   when  $a(x)$ satisfies \eqref{7.25}.
\par
Note that if $f(u)$ and $a(x)$ satisfy $(f_1)$ and $(a_2)$, then the identity \eqref{7.26}
restricts the size of $u_n(0)=\|u_n\|_{L^\infty(0,1)}$ so that $(\l_n,u_n)$ can be a regular solution of \eqref{1.1}. Thus, under these assumptions,
any sufficiently large solution must be singular.
\end{rem}

More generally, when $a(x)$ satisfies $(a_5)$,  instead of  $(a_6)$, the next result holds.

\begin{proposition}
\label{pr6.1}
Assume  $(f_1)$, $(a_2)$, $(a_5)$ and \eqref{6.18}. Let $\{(\l_n,u_n)\}_{n\geq 1}$ be a sequence of positive solutions of \eqref{1.1} satisfying \eqref{7.20}.
Suppose, in addition, that there exist constants $\eta>0$ and $C>0$ such that
\begin{equation}
\label{7.28}
  \l_n f(u_n(x))\geq C \quad \hbox{if}\;\; 0<|x-z|<\eta.
\end{equation}
Then, for sufficiently large $n$, $(\l_n,u_n)$ is a singular  solution of \eqref{1.1} with $u_n(z^-) > u_n(z^+)$.
\end{proposition}

\begin{proof}
Let us set, for every $n\ge1$,
\begin{equation}
h_n(x) = \l_na(x) f(u_n(x)) \quad   \text{for a.e. } x\in [0,1].
\end{equation}
For each $n$, the function $h_n(x) $ satisfies  assumption $(h_1)$ and
\begin{equation}
\int_x^z h_n(t)\, dt \ge C \int_x^z a(t) \, dt >0  \quad \hbox{if}\;\; 0<|x-z|<\eta.
\end{equation}
In addition, $\int_x^z h_n(t)\, dt >0$ for all $x\in [0,1] \setminus \{z\}$,   and hence
\begin{equation}
\int_{z-\eta}^{z+\eta} \left(\int_x^z h_n(t)\, dt\right) ^{-\frac{1}{2}}dx
\leq
 \frac{1}{\sqrt{C}} \int_{z-\eta}^{z+\eta}  \left( \int_x^z a(t)\,dt \right)^{-\frac{1}{2}} dx < \infty.
\end{equation}
Moreover,   by  Corollaries \ref{covi.1} and  \ref{covi.2}, we already  know that
$\lim_{n\to\infty}u_n(z-\eta) = \infty$ and $\lim_{n\to\infty}u_n(z+\eta) =0$, which implies that, for
sufficiently large $n$,
\begin{equation}
u_n(z-\eta) - u_n(z+\eta) \ge \int_{z-\eta}^{z+\eta} \left(\int_x^z h_n(t)\, dt\right) ^{-\frac{1}{2}}dx .
\end{equation}
 Therefore, the conclusion can be inferred from Theorem \ref{th7.1} (b).
 \end{proof}

\begin{rem}
By the definition of $x_n$, we have that
$$
  \lim_{n\to \infty} \left( \l_n f(u_n(x_n))\right)=\lim_{n\to \infty} \left( \l_n f(M)\right)=\infty
$$
and, due to Lemma \ref{levi.8}, $\lim_{n\to \infty}x_n=z$. Thus, the condition \eqref{7.28} seems rather natural to hold. Unfortunately, we were not able to
exclude the existence of some sequence $\{y_{n_k}\}_{k\geq 1}$
such that $\lim_{k\to \infty} y_{n_k}=z$ and $\lim_{k\to \infty}\left( \l_{n_k} f(u_{n_k}(y_{n_k}))\right) = 0$ in the general case when $a(x)$ satisfies $(a_5)$.
  So, it remains an open problem to characterize the existence of positive singular solutions of \eqref{1.1} when $F(u)$  is sublinear at infinity, unlike what we were able to do  in \cite{LGO-JDE, LGO-20} for potentials which are linear or superlinear at infinity.
\end{rem}


\begin{thebibliography}{99}


\bibitem{Amann}
    H. Amann,
    Fixed point equations and nonlinear eigenvalue problems in ordered Banach spaces,
    \emph{SIAM Rev.} \textbf{18} (1976), 620--709.

\bibitem{AL98}
    H. Amann and J. L\'{o}pez-G\'{o}mez,
    A priori bounds and multiple solutions for superlinear indefinite elliptic problems,
    \emph{ J. Differential Equations} \textbf{146} (1998), 336--374.

\bibitem{AmFuPa}
  L. Ambrosio, N. Fusco and D. Pallara,
\emph{Functions of Bounded Variation and Free Discontinuity Problems},
Clarendon Press, Oxford, 2000.

\bibitem{Anz83}
 G. Anzellotti,
 The Euler equation for functionals with linear growth,
\emph{Trans. Amer. Math. Soc.} \textbf{290}  (1985),   483--501.

\bibitem{BDGM}
 E. Bombieri, E. De Giorgi and M. Miranda, Una maggiorazione a priori relativa alle ipersuperfici minimali non parametriche,
\emph{Arch. Ration. Mech. Anal.} \textbf{32} (1969), 255--267.


\bibitem{BoHaObOmJDE} D. Bonheure, P. Habets, F. Obersnel and P. Omari,
Classical and non-classical solutions of a prescribed curvature equation,
 \emph{J. Differential Equations} \textbf{243} (2007),  208--237.


\bibitem{BoHaObOmTS} D. Bonheure, P. Habets, F. Obersnel and P. Omari,
Classical and non-classical positive solutions of a prescribed curvature equation with singularities,   \emph{Rend. Istit. Mat. Univ. Trieste} \textbf{39} (2007), 63--85.

\bibitem{BL}
    K.J. Brown and S.S. Lin,
    On the existence of positive eigenfunctions for an eigenvalue problem with indefinite weight function,
   \emph{J. Math. Anal. Appl.}   \textbf{75} (1980), 112--120.

    \bibitem{BuGr}
M. Burns and  M. Grinfeld,
Steady state solutions of a bi-stable quasi-linear equation with saturating flux,
\emph{European J. Appl. Math.} \textbf{22} (2011),   317--331.




\bibitem{CaDMLePa}  M. Carriero, G. Dal Maso, A.  Leaci and E. Pascali,
 Relaxation of the nonparametric Plateau problem with an obstacle,
 \emph{J. Math. Pures Appl.} \textbf{67} (1988),   359--396.


\bibitem{CMM96}
P. Cl\'ement, R. Man\'asevich and E. Mitidieri,
On a modified capillary equation,
 \emph{J. Differential Equations} \textbf{124} (1996), 343--358.




\bibitem{CZ91}
C.V. Coffman and W.K. Ziemer,
A prescribed mean curvature problem on domains without radial symmetry,
 \emph{SIAM J. Math. Anal.} \textbf{22} (1991), 982--990.


\bibitem{CF}
P. Concus  and  R. Finn,
On a class of capillary surfaces,
 \emph{J. Analyse Math.}  \textbf{23} (1970), 65--70.



\bibitem{CDCO}
C. Corsato, C. De Coster and P. Omari,
The Dirichlet problem for a prescribed anisotropic mean curvature equation: existence, uniqueness and regularity of solutions,
\emph{J. Differential Equations}   \textbf{260} (2016), 4572--4618.



\bibitem{CDCOOS}
C. Corsato, C. De Coster,  F. Obersnel, P. Omari and A. Soranzo,
 A prescribed anisotropic mean curvature equation modeling the corneal shape: a paradigm of nonlinear analysis,
\emph{Discrete Contin. Dyn. Syst. Ser. S}  \textbf{11} (2018),  213--256.


\bibitem{COZ}
C. Corsato, P. Omari and F. Zanolin,
Subharmonic solutions of the prescribed curvature equation
\emph{Commun. Contemp. Math.}   \textbf{18} (2016), 1550042, 1--33.



\bibitem{CR}
     M. G. Crandall and P. H. Rabinowitz,
     Bifurcation from simple eigenvalues,
     \emph{J. Funct. Anal.} \textbf{8} (1971), 321--340.


\bibitem{DCH}
  C. De Coster  and P.  Habets,
 {\em Two-point boundary value problems: lower and upper solutions,}
 Elsevier, Amsterdam, 2006.

 \bibitem{Em}
M. Emmer, Esistenza, unicit\`a e regolarit\`a  nelle superfici di equilibrio nei capillari,
\emph{Ann. Univ. Ferrara} \textbf{18} (1973), 79--94.



\bibitem{FZ}
G. Feltrin and F. Zanolin,
Multiplicity of positive periodic solutions in the superlinear indefinite case via coincidence degree, \emph{J. Differential Equations} \textbf{262} (2017), 4255-4291.


\bibitem{FLG}    M. Fencl and J. L\'{o}pez-G\'{o}mez, Global bifurcation diagrams of
positive solutions for a class of 1-D superlinear indefinite problems, arXiv:2005.09369;19.05.2020.


\bibitem{FRLG}
   S. Fern\'andez-Rinc\'on and J. L\'opez-G\'omez,
   Spatially heterogeneous Lotka-Volterra competition,
   \emph{Nonlinear Anal.} \textbf{165} (2017), 33--79.

\bibitem{FRLGB}
S. Fern\'andez-Rinc\'on and J. L\'opez-G\'omez, The Picone identity: A device to get optimal uniqueness results and global dynamics in Population Dynamics, arVix:1911.05066;20.11.2019.


\bibitem{Fi80}
R. Finn, The sessile liquid drop. I. Symmetric case,
\emph{Pacific J. Math.} \textbf{88} (1980), 541--587.


\bibitem{Fi}
R. Finn, \emph{Equilibrium Capillary Surfaces}, Springer, New York, 1986.



\bibitem{Ge}
C. Gerhardt, Boundary value problems for surfaces of prescribed mean curvature,
\emph{J. Math. Pures Appl.} \textbf{58} (1979),  75--109.

\bibitem{Ge85}
C. Gerhardt,
Global $C^{1,1}$-regularity for solutions of quasilinear variational inequalities,
\emph{Arch. Ration. Mech. Anal.} \textbf{89} (1985),  83--92.



\bibitem{Giu}
E. Giusti, Boundary value problems for non-parametric surfaces of prescribed mean curvature,
\emph{Ann. Sc. Norm. Super. Pisa Cl. Sci.}   \textbf{3} (1976),  501--548.


\bibitem{GL00}
 R. G\'{o}mez-Re\~{n}asco and J. L\'{o}pez-G\'{o}mez,
 The effect of varying coefficients on the dynamics of a class of
superlinear indefinite reaction diffusion equations,
 \emph{J. Differential Equations}  \textbf{167} (2000), 36--72.

\bibitem{GL01}
 R. G\'{o}mez-Re\~{n}asco and  J. L\'{o}pez-G\'{o}mez,
 The uniqueness of the stable positive solution for a class of
superlinear indefinite reaction diffusion equations,
 \emph{Differential Integral Equations} \textbf{14} (2001), 751--768.

 \bibitem{GMT}
E. Gonzalez, U. Massari and I. Tamanini,
Existence and regularity for the problem of a pendent liquid drop,
\emph{Pacific J. Math.} \textbf{88} (1980),  399--420.



\bibitem{Hu85}
G. Huisken,
Capillary surfaces over obstacles,
\emph{Pacific J. Math.}   \textbf{117} (1985),  121--141.


\bibitem{KuRo}  A. Kurganov and P.  Rosenau,
 On reaction processes with saturating diffusion,
 \emph{Nonlinearity} \textbf{19} (2006), 171--193.


\bibitem{LU}
O.A. Ladyzhenskaya and N.N. Ural'tseva,  Local estimates for gradients of solutions of non-uniformly elliptic and parabolic equations,
\emph{Comm. Pure Appl. Math.} \textbf{23} (1970), 677--703.


\bibitem{Le}
V.K. Le,
Some existence results of non-trivial solutions of the precribed mean curvature equation,
  \emph{Adv. Nonlinear Stud.} {\bf 5} (2005), 133--161.

\bibitem{LS}
    J. Leray and J. Schauder,
    Topologie et \'equations fonctionelles,
    \emph{Ann. Sci. \'Ec. Norm. Sup\'er.} \textbf{51} (1934), 45--78.



\bibitem{LG01}
    J. L\'{o}pez-G\'{o}mez,
    \emph{Spectral Theory and Nonlinear Functional Analysis},
   Chapman and Hall/CRC Press, Boca Raton, 2001.


\bibitem{LG13}
          J. L\'{o}pez-G\'{o}mez,
          \emph{Linear Second Order Elliptic Operators},
          World Scientific, Singapore, 2013.



\bibitem{LGO-19}
     J. L\'opez-G\'omez and P. Omari, Global components of positive bounded variation solutions of a one-dimensional indefinite quasilinear Neumann problem, \emph{Adv. Nonlinear Stud.} \textbf{19} (2019),
437--473.

\bibitem{LGO-JDE} J. L\'opez-G\'omez and P. Omari, Characterizing the formation of singularities in a
superlinear indefinite problem related to the mean curvature operator,
\emph{J. Differential Equations}   \textbf{269} (2020), 1544--1570.

\bibitem{LGO-20} J. L\'opez-G\'omez and P. Omari, Singular versus regular solutions in a quaslilinear indefinite problem with an asymptotically linear potential, \emph{Adv. Nonlinear Stud.} \textbf{20} (2020), 557--578.

\bibitem{LOR1}
     J. L\'opez-G\'omez, P. Omari and S. Rivetti,
     Positive solutions of one-dimensional indefinite capillarity-type problems
     \emph{J. Differential Equations}  \textbf{262} (2017), 2335--2392.

\bibitem{LOR2}
 J. L\'opez-G\'omez, P. Omari and S. Rivetti,
 Bifurcation of positive solutions for a one-dimensional indefinite quasilinear Neumann problem,
 \emph{Nonlinear Anal.} \textbf{155} (2017), 1--51.


\bibitem{Ma}
M.  Marzocchi,
Multiple solutions of quasilinear equations involving an area-type term,
 \emph{J. Math. Anal. Appl.}  {\bf 196} (1995), 1093--1104.



\bibitem{Na}
 M.  Nakao,
 A bifurcation problem for a
quasi-linear elliptic boundary value problem,
\emph{Nonlinear Anal.}  \textbf{14} (1990), 251--262.

\bibitem{NS2}
W.M. Ni and J. Serrin,
Existence and non-existence theorems for quasilinear  partial differential
equations. The anomalous case,
\emph{Accad. Naz. Lincei - Atti dei Convegni} \textbf{77} (1985),  231--257.


\bibitem{ObOmDIE}
F. Obersnel and P. Omari, Existence and multiplicity results for the prescribed mean curvature equation via lower and upper solutions,
\emph{ Differential Integral Equations} \textbf{22} (2009),  853--880.


 \bibitem{ObOmDCDS}
F. Obersnel and  P. Omari, Existence, regularity and boundary behaviour of bounded variation solutions of a one-dimensional capillarity equation, \emph{Discrete Contin. Dyn. Syst.} \textbf{33} (2013),   305--320.


 \bibitem{ObOm20}
F. Obersnel and  P. Omari, Revisiting the sub- and super-solution method for the classical radial solutions of the mean curvature equation, \emph{Open Math.} \textbf{18} (2020),   1185--1205.


\bibitem{ObOmRiNA}
F. Obersnel, P. Omari and S. Rivetti,
Existence, regularity and stability properties of periodic solutions of a capillarity equation in the presence of lower and upper solutions,
\emph{Nonlinear Anal. Real World Appl.} \textbf{13}  (2012), 2830--2852.


\bibitem{ObOmRi}
F. Obersnel, P. Omari and S. Rivetti, Asymmetric Poincar\'e inequalities and solvability of capillarity problems,
\emph{J. Funct. Anal.} \textbf{267} (2014),  842--900.



\bibitem{Picone}
     M. Picone,
    Sui valori eccezionali di un parametro da cui dipende un'equazione differenziale ordinaria del second'ordine,
    \emph{Ann. Sc. Norm. Super. Pisa Cl. Sci.} \textbf{11} (1910), 1--144.


\bibitem{Rabgb}
     P.H. Rabinowitz,
     Some global results for nonlinear eigenvalue problems,
     \emph{J. Funct. Anal.} \textbf{7} (1971), 487--513.

\bibitem{Ra71}
P.H.\ Rabinowitz, A global theorem for nonlinear eigenvalue problems and applications, in  \emph{Contributions to Nonlinear Functional Analysis,}
Proc. Sympos. Math. Res. Center, Univ. Wisconsin, Madison,
 Academic Press, New York (1971),  11-36.

 \bibitem{Se}
J. Serrin, The problem of Dirichlet for quasilinear elliptic differential equations with many independent variables,
\emph{Philos. Trans. Roy. Soc. London} \textbf{264} (1969) 413--496.

\bibitem{Se88}
J. Serrin,
Positive solutions of a prescribed mean curvature problem,
in {\it Calculus of variations and partial differential equations} Trento 1986,
Lecture Notes in Math. \textbf{1340}, Springer, Berlin  (1988).


\bibitem{Te}
R. Temam, Solutions g\'en\'eralis\'ees de certaines \'equations du type hypersurfaces minima,
\emph{Arch. Ration. Mech. Anal.} \textbf{44} (1971/72), 121--156.


\end{thebibliography}
\end{document}